\documentclass[11pt]{amsart}
\usepackage[margin=1in]{geometry}
\usepackage{tikz}
\usepackage{eucal}
\usepackage{mathrsfs}
\usepackage{amssymb}

\usepackage{xcolor}
\colorlet{highlight}{black}
\colorlet{secondrevision}{black}
\colorlet{thirdrevision}{black}

\usepackage{setspace}
\usepackage[all]{xy}

\newtheorem{thm}{Theorem}[subsection]
\newtheorem{prin}[thm]{Principle}
\newtheorem{prop}[thm]{Proposition}

\theoremstyle{definition}
\newtheorem{defn}[thm]{Definition}
\theoremstyle{remark}
\newtheorem{rmk}[thm]{Remark}
\newtheorem{example}[thm]{Example}

\newcommand{\cA}{\mathcal{A}}
\newcommand{\cB}{\mathcal{B}}
\newcommand{\cC}{\mathscr{C}}
\newcommand{\calC}{\mathcal{C}}
\newcommand{\cD}{\mathcal{D}}
\newcommand{\cE}{\mathcal{E}}
\newcommand{\cF}{\mathcal{F}}
\newcommand{\cG}{\mathscr{G}}

\newcommand{\cO}{\mathcal{O}}

\newcommand{\cW}{\mathcal{W}}
\newcommand{\cY}{\mathcal{Y}}

\newcommand{\K}{\mathbb{K}}
\newcommand{\R}{\mathbb{R}}
\newcommand{\Z}{\mathbb{Z}}
\newcommand{\Ch}{\mathrm{Ch}}
\newcommand{\GL}{\mathrm{GL}}
\newcommand{\point}{\mathbf{pt}}
\newcommand{\bimod}[2]{#1\text{--mod--}#2}

\newcommand{\Sympcat}{\mathbf{Symp}}
\newcommand{\DGcat}{\mathbf{DGcat}}
\newcommand{\Ainfcat}{\mathbf{A_\infty cat}}

\DeclareMathOperator{\Ob}{Ob}

\DeclareMathOperator{\Hom}{Hom}
\DeclareMathOperator{\Homtwo}{2-Hom}
\DeclareMathOperator{\shom}{hom}
\DeclareMathOperator{\cHom}{\mathscr{H}om}
\DeclareMathOperator{\Mod}{Mod}
\DeclareMathOperator{\Perf}{Perf}
\DeclareMathOperator{\Fun}{Fun}
\DeclareMathOperator{\Sh}{Sh}
\DeclareMathOperator{\Symp}{Symp}

\DeclareMathOperator{\Coh}{Coh}

\DeclareMathOperator{\End}{End}
\DeclareMathOperator{\Aut}{Aut}
\DeclareMathOperator{\Pic}{Pic}
\DeclareMathOperator{\Ad}{Ad}
\newcommand{\op}{\mathrm{op}}
\newcommand{\co}{\mathrm{co}}
\newcommand{\coop}{\mathrm{co,op}}
\newcommand{\gp}{\mathrm{gp}}

\newcommand{\bDelta}{\boldsymbol{\Delta}}
\newcommand{\bDeltaint}{\boldsymbol{\Delta}_\mathrm{int}}
\newcommand{\id}{\mathrm{id}}
\newcommand{\Set}{\mathrm{Set}}
\newcommand{\Cat}{\mathbf{Cat}}
\newcommand{\Map}{\mathrm{Map}}
\newcommand{\VectK}{\mathrm{Vect}_{\K}}
\newcommand{\Man}{\mathrm{Man}}
\newcommand{\cEnd}{\mathcal{E}\mathrm{nd}}
\newcommand{\coker}{\mathrm{coker}}

\newcommand{\Xbf}[1]{X(\mathbf{#1})}

\title[Poisson Geometry and Monoidal Fukaya Categories]{Poisson Geometry, Monoidal Fukaya Categories, and Commutative Floer Cohomology Rings}

\author{James Pascaleff}

\address{University of Illinois at Urbana-Champaign}

\date{First version March 2018. Major revision January 2022. This version March 2023}

\email{jpascale@illinois.edu}

\begin{document}

\maketitle

\begin{abstract}
  We describe connections between concepts arising in Poisson geometry and the theory of Fukaya categories. The key concept is that of a symplectic groupoid, which is an integration of a Poisson manifold. The Fukaya category of a symplectic groupoid is monoidal, and it acts on the Fukaya categories of the symplectic leaves of the Poisson structure. Conversely, we consider a wide range of known monoidal structures on Fukaya categories and observe that they all arise from symplectic groupoids. We also use the picture developed to resolve a conundrum in Floer theory: why are some Lagrangian Floer cohomology rings commutative? 
\end{abstract}

\section{Introduction}
\label{sec:intro}

The concept of a Poisson manifold, namely a manifold equipped with a Poisson bracket $\{\cdot,\cdot\}$ on its space of smooth functions, is a natural generalization of the concept of a symplectic manifold. Symplectic forms on a given manifold $M$ correspond bijectively to the Poisson brackets that are nondegenerate in the sense that every vector $X \in T_pM$ is generated by a derivation of the form $g \mapsto \{f,g\}$ for some function $f$. Whereas the local structure theory of symplectic manifolds is essentially trivial due to the Darboux theorem, the local structure theory of Poisson manifolds is extremely complicated; for instance, it contains the theory of arbitrary Lie algebras.

As natural as the generalization from symplectic to Poisson structures is from the point of view of differential geometric structures on manifolds, from the point of view of Floer theory, it can be argued that the generalization is completely \emph{unnatural}. There is a reason that we have a good theory of pseudo-holomorphic curves in symplectic manifolds \cite{gromov85}: while one may reasonably develop the local theory of pseudo-holomorphic curves in any almost complex manifold, in order to have compact moduli spaces, one needs to control the energy of the curves. Gromov's insight was that the natural geometric way to get such control is to assume that the almost complex structure is tamed by a symplectic form. It is clear from the argument that the nondegeneracy of the symplectic form is really essential, and any attempt to weaken this condition (such as in some versions of symplectic field theory) requires great care.

Nevertheless, by broadening the perspective we can see that there are other tracks to follow. It turns out that Poisson geometers do not ignore symplectic structures as trivial. In fact, a modern perspective introduced by Weinstein is that a powerful way to study a Poisson manifold is to associate to it a symplectic manifold of twice the dimension, a \emph{symplectic realization}. The nicest symplectic realizations are the \emph{symplectic integrations}, and such an object by definition carries the structure of a \emph{symplectic groupoid}. 

Given a symplectic groupoid $(G,\omega)$, we can consider the Fukaya $A_{\infty}$-category $\cF(G)$ of the underlying symplectic manifold. The additional groupoid structure gives us additional higher-categorical structure on $\cF(G)$, namely, it makes the Fukaya $A_{\infty}$-category $\cF(G)$ monoidal, and equips it with a duality functor. These functors are represented by Lagrangian correspondences that encode the groupoid structure.

This idea is not new: M.~Gualtieri has informed the author that he has presented this idea as early as 2009 \cite{gualtieri-talk} (see also section \ref{sec:coisotropic-fukaya}), and related ideas appear in the 2014 ICM article by C.~Teleman \cite{teleman-gauge-icm} (see section \ref{sec:gdual}). It is also quite possible that the same idea has occurred to others. This note has several aims:
\begin{enumerate}
\item to explicate how the theory of functors between Fukaya categories from Lagrangian correspondences works in the context of symplectic groupoids, leading essentially to the notion of a monoidal category with duality, 
\item to survey how this idea neatly ties together a range of known monoidal and duality structures on Fukaya categories, as well as monoidal module actions on Fukaya categories,
\item in what is perhaps the most original contribution of the note, to show that the theory of symplectic groupoids allows us to resolve a conundrum in Floer theory, namely, the question of why some Lagrangian Floer cohomology rings are commutative while others are not. Our answer is that commutativity is explained by the presence of a symplectic groupoid structure, together with standard arguments from the theory of monoidal categories.
\end{enumerate}

\subsection{Outline}

The basic conceptual move is to change focus from Poisson manifolds to so-called \emph{symplectic groupoids}, introduced by Weinstein. A symplectic groupoid is a differential-geometric object that has both an underlying symplectic manifold $(G,\omega)$ as well as an associated Poisson manifold $(M,\pi)$; note that these two structures live on different underlying manifolds. Roughly speaking, the relationship between $(G,\omega)$ and $(M,\pi)$ is parallel to the relationship between a Lie group and its Lie algebra. Also of note is that not every Poisson manifold arises this way; by definition $(M,\pi)$ is \emph{integrable} if it does.

The groupoid structure $(G,\omega)$ is encoded by structure maps between $G$ and $M$. In accordance with Weinstein's creed ``Everything is a Lagrangian,'' these structure maps are also encoded by Lagrangian correspondences between various copies of $G$ \textcolor{highlight}{and $\overline{G}$, which is $G$ with the sign of the symplectic form reversed.} By the Wehrheim-Woodward theory of quilts, these correspondences induce functors on the Fukaya $A_{\infty}$-category $\cF(G)$. This immediately leads to the following guiding principle.
\begin{prin}
  Let $(G,\omega)$ be a symplectic manifold, and assume that there is a good theory of Fukaya categories and functors from Lagrangian correspondences for $G$. Then a symplectic groupoid structure on $(G,\omega)$ induces
  \begin{itemize}
  \item a monoidal product $\otimes :\cF(G)\times \cF(G) \to \cF(G)$,
  \item a distinguished object $\cO \in \Ob \cF(G)$, which is a unit object for $\otimes$, and 
  \item an equivalence functor $\cD: \cF(G) \to \cF(\overline{G})$, to be thought of as a kind of duality.
  \end{itemize}
\end{prin}

\textcolor{highlight}{Our definition of what it means to have a ``good theory'' of Lagrangian correspondences is made precise in Section \ref{sec:good-sol}.}

If this were the whole story, we might be led to believe that Poisson geometry itself is a distraction, and the real story is about symplectic groupoids (which, among other properties, happen to determine Poisson structures). But there is more structure present. The underlying Poisson manifold $(M,\pi)$ carries a singular integrable distribution $\pi^\#(T^*M) \subseteq TM$. The leaves of the corresponding foliation are symplectic manifolds. These leaves are precisely the isomorphism classes of objects in the groupoid (thought of as a category). Hence if $F \subseteq M$ is a leaf $G$ acts on $F$. In line with the creed, this is represented by a Lagrangian correspondence from $G \times F$ to $F$. Thus we find a principle that relates 

\begin{prin}
  Let $(G,\omega)$ be a symplectic groupoid integrating $(M,\pi)$, and let $F \subseteq M$ be a symplectic leaf. Assuming a good theory of Lagrangian correspondences for $G$ and $F$, we find that
  \begin{itemize}
  \item There is a functor $\rho : \cF(G) \times \cF(F) \to \cF(F)$ that makes $\cF(F)$ into a monoidal module $A_{\infty}$-category for the monoidal $A_{\infty}$-category $\cF(G)$.
  \end{itemize}
\end{prin}

It is possible to go further, and study how the Morita theory of symplectic groupoids relates the monoidal Morita theory of monoidal categories, but we will limit ourselves to exploring the two principles above for the purposes of this paper. Another connection between Poisson geometry and Floer theory arises from the possibility to promote coisotropic submanifolds of $(M,\pi)$ to \emph{Lagrangian subgroupoids} of $(G,\omega)$, see Section \ref{sec:coisotropic-fukaya}.

While we do not claim such constructions are possible in anything approaching the full generality of all symplectic groupoids (in particular, we have not spelled out any geometric hypothesis on $(G,\omega)$ that makes it possible to define the Fukaya category and have a good theory of functors from Lagrangian correspondences), we will investigate it in several cases to show that it reproduces known or at least expect monoidal structures in the theory of Fukaya categories, thereby unifying them as instances of this guiding principle. 

\begin{rmk}
  We can attempt to couch our perspective in the language of categorification. To a mathematical object, one may attempt to attach a ``category number'' that measures ``how categorified'' the object is. The following represents the author's opinion. Manifolds have category number zero, because the intersection of two cycles is a number. Symplectic manifolds have category number one, because the intersection of two Lagrangian submanifolds is a graded vector space whose Euler characteristic recovers the intersection number of the corresponding cycles. Symplectic groupoids have category number two, because there is an additional monoidal structure.
\end{rmk}

The other direction we shall explore is that, in the presence of a monoidal structure, standard arguments from the theory of monoidal categories can be brought to bear on the calculation of Floer cohomology rings, in particular, we find that this framework gives a geometric \emph{a priori} reason why several Floer cohomology rings are graded commutative, see Section \ref{sec:commutativity}.

Lastly, a warning and an apology: at many points in this note, we will indicate how something ``should'' work, without giving complete details of the construction, or even necessarily a precise statement. These statements should be understood to be of a speculative or conjectural nature. It is my belief that all such statements in this note are correct in the sense that they can be set up and proved in all reasonable cases by an elaboration of known constructions in Floer theory. While I can understand that some may object to such an approach, this note is based on ideas that are rather simple once you see them, and my attitude is that the essential simplicity of the ideas should not be obscured by a premature effort to construct everything in complete detail.

Another possible justification for this formal approach is that, though this author envisions that all of the constructions described here can be made in terms of pseudo-holomorphic curve theories, the same ideas should apply, at least in outline, to any other ``version of the Fukaya category:'' $\cD$-modules or deformation quantization modules, or Nadler's approach in terms of Lagrangian skeleta and categorical Morse homology, or Tamarkin's approach in terms of microlocal sheaves. The ideas should also apply to any version of the Fukaya category yet to be conceived. Fundamentally, all that is required is that one can associate categories to symplectic manifolds, and functors to Lagrangian correspondences. 

\textcolor{highlight}{In this revised version, we have included an Appendix to make our formal arguments more precise. The Appendix recalls precise definitions of monoidal $A_{\infty}$- and $\infty$-categories (in several versions), as well the approach to $(\infty,2)$-categories via Segal conditions. As is well-known, $2$-categories and monoidal categories live at essentially the same level, so it is not surprising that they may be given parallel treatment. This also makes it possible for us to assign a precise meaning to the phrase ``good solution to the composition problem'' at the $(\infty,2)$-categorical level. The length of the Appendix is symptomatic of many expositions of higher category theory. These notions are applied to symplectic groupoids in Sections \ref{sec:good-sol} and \ref{sec:new-monoids}.}

\subsection{Relation to other work}
\label{sec:previous}

I am not the first person to consider the idea that the Fukaya category of a symplectic groupoid should be monoidal, although it seems this idea is not widely appreciated. Marco Gualtieri informs me that he has advocated for this idea in talks as early as 2009. The thesis of Aleksandar Suboti\'{c} \cite{subotic-thesis} uses the groupoid structure of a torus fibration in exactly this way. The work of Constantin Teleman on $G$-equivariant Fukaya categories \cite{teleman-gauge-icm} involves essentially the action of the symplectic groupoid $T^*G$ on a Hamiltonian $G$-manifold (see the ``proof'' of Conjecture 2.9, \textit{op.~cit.}). This action was exploited to great effect by Jonny Evans and Yank\i{} Lekili in their generation results for Hamiltonian $G$-manifolds \cite{talking-bout-my-g-generation}. There is also a MathOverflow thread\footnote{\texttt{https://mathoverflow.net/questions/19041/a-poisson-geometry-version-of-the-fukaya-category}} where related ideas are discussed.

In the recent work of D. Ben-Zvi and S. Gunningham \cite{benzvi-gunningham}, a certain symplectic groupoid (the group scheme $J$ of regular centralizers in a complex reductive group) appears in connection with the symmetries of categorical representations and Ng\^{o}'s work on the fundamental lemma \cite{ngo}. While Ben-Zvi and Gunningham use $\cD$-modules as their version of $A$-branes, it is natural to conjecture that their picture has an interpretation in any version of the Fukaya category; see \cite[Remark 2.7]{benzvi-gunningham}.

Further, I think it likely that the basic idea has occurred to other people. If you are one of those people, perhaps the point of this note is that the structure one obtains on the Fukaya category is already interesting even in cases where the underlying Poisson structure is in some sense uninteresting (most of the examples we consider have constant rank).

\subsection{Acknowledgments}
\label{sec:ack}

I would like to thank Rui Loja Fernandes for introducing me to Poisson geometry and symplectic groupoids. Paul Seidel provided crucial suggestions, and David Jordan provided answers to my very basic questions regarding monoidal categories. I would also like to thank Marco Gualtieri for sharing his perspective, as well as David Ben-Zvi and Sam Gunningham for sharing their work with me.

This note is based on a talk given on September 24, 2017 at the AMS Sectional Meeting in Orlando, Florida, as well as a talk given on December 1, 2017 at the workshop ``Categorification, Representation Theory and Symplectic Geometry'' at the Hausdorff Research Institute for Mathematics in Bonn, Germany. I thank Basak Gurel and Viktor Ginzburg, the organizers of the special session in Orlando, and Anne-Laure Thiel and Daniel Tubbenhauer, the organizers of the workshop in Bonn, for the opportunity to speak.

Part of this work was completed while the author was a member of the Institute for Advanced Study during the 2016--17 special year on Homological Mirror Symmetry. The author was partially supported by NSF grant DMS-1522670.
 
\textcolor{highlight}{In preparing this revised version, I am grateful to Nate Bottman for conversations that spurred me to understand precisely how the notions of monoid objects in higher category theory could be applied to the problem of monoidal Fukaya categories. I also benefited from conversations with Ezra Getzler about simplicial objects and derived algebraic geometry. I thank the referee for many thoughtful comments and suggestions that improved the paper considerably.}

\textcolor{highlight}{\subsection{Terminology}
  \label{sec:terminology}
  Unless otherwise specified, when we refer to the ``Fukaya category'' we mean an $A_{\infty}$-category, and related notions such as ``modules'' or ``functors'' are the $A_{\infty}$ versions. The phrase ``cohomology-level Fukaya category'' refers to the ordinary category obtained by taking the degree-zero cohomology of all morphism complexes; this is also known as the homotopy category of the Fukaya category or the Donaldson category. Other homotopy-coherent categorical notions, such as $\infty$-categories, are explicitly marked. See the Appendix for more background information.}

\section{Symplectic groupoids and integrations}
\label{sec:groupoids}

\subsection{Definitions}
\label{sec:groupoid-defs}

First we begin with the basic category-theoretic notions.
\begin{defn}
  A \emph{groupoid} is a small category in which all morphisms are invertible.
\end{defn}
A groupoid may be presented as follows: Given are two sets $M$ and $G$. $M$ is the set of objects, and $G$ is the set of morphisms, that is, the disjoint union of all the morphism spaces between all pairs of objects. Also given is a map $s: G \to M$ that takes a morphism to its source object, a map $t: G \to M$ that takes a morphism to its target object, a map $e : M \to G$ that takes an object to its identity morphism, a map $i : G \to G$ that takes a morphism to its inverse, and a map $m : G\times_{s,t} G \to G$ that takes a pair of composable morphisms to their composition. The statement that these data form a category in which all morphisms are invertible can then be formulated as a list of axioms that $s,t,e,i,m$ must satisfy.

The corresponding ``Lie'' notion just involves replacing sets and functions with smooth manifolds and smooth maps.
\begin{defn}
  A \emph{Lie groupoid} is presented as $(G,M,s,t,e,i,m)$ where $G$ and $M$ are endowed with the structure of smooth manifolds, and all structural maps are smooth maps.
\end{defn}

Now comes a crucial concept introduced by Weinstein \cite{weinstein}. 
\begin{defn}
  Let $(G,M,s,t,e,i,m)$ be a Lie groupoid. A \emph{multiplicative symplectic structure} on $G$ is a symplectic form $\omega \in \Omega^2(G)$ such that
  \begin{equation}
    m^*\omega = \pi_1^*\omega + \pi_2^*\omega
  \end{equation}
  holds as an identity in $\Omega^2(G\times_{s,t}G)$. A \emph{symplectic groupoid} consists of a Lie groupoid with a multiplicative symplectic structure.
\end{defn}
For the most part we shall notate a symplectic groupoid as $(G,\omega)$, suppressing the rest of the groupoid structure. This actually does have the potential to cause confusion, because there are pairs $(G,\omega)$ where $\omega$ is multiplicative for more than one groupoid structure on $G$. We hope the reader will be able to understand what is intended from context.

\subsection{Integrability of Poisson manifolds}
\label{sec:integrability}

We now recall the connection to Poisson geometry. This material is well-known, but we include it for context.

Associated to a Lie groupoid $(G,M,s,t,e,i,m)$, there is associated a corresponding Lie algebroid, defined by linearization of the groupoid structure along the image of the identity section $e: M \to G$. This is analogous to the construction of the Lie algebra associated to a Lie group, which is precisely the case where $M$ is a single point. The algebroid consists of $(M,E,a, [\cdot,\cdot])$, where $M$ is as before, $E$ is a vector bundle on $M$, $a : E \to TM$ is a map of vector bundles over $M$, and $[\cdot,\cdot]$ is a bracket on sections of $E$, which satisfy a list of axioms obtained by differentiation of the groupoid axioms. A Lie algebroid is called \emph{integrable} if it arises as the associated Lie algebroid of some Lie groupoid, and we say that the Lie groupoid is \emph{an integration} of the Lie algebroid. Integrations do not necessarily exist nor are they necessarily unique when they do. This area has been much studied; see for instance \cite{crainic-fernandes-lie,crainic-fernandes-poisson} and references therein.

On the other hand, a Poisson manifold $(M,\pi)$ has an associated Lie algebroid, namely $E = T^*M$, $\pi^\#: T^*M \to TM$, and $[\alpha,\beta] = d\pi(\alpha,\beta)$. When $(G,\omega)$ is a symplectic groupoid, the associated Lie algebroid is of this form. That is to say, given a symplectic groupoid $(G,\omega)$ there is a unique Poisson structure $\pi$ on $M$ such that $G$ is an integration of the Lie algebroid associated to this Poisson structure. We then say that $(G,\omega)$ is a \emph{symplectic integration} of $(M,\pi)$. A Poisson manifold $(M,\pi)$ is called \emph{integrable} if it admits a symplectic integration. Integrations are not unique when they exist, but the $s$-simply connected\footnote{\textcolor{highlight}{\emph{$s$-simply connected} means that the fibers of $s : G\to M$ are simply connected.}} integration is unique (analogous to the simply connected integration of a Lie algebra).

Since the main object of study in this paper is a symplectic groupoid, all Poisson manifolds that appear are integrable. Our perspective is that a symplectic groupoid is a symplectic manifold with an extra structure that, in particular, encodes an integrable Poisson structure on another manifold. 

\subsection{Action on the symplectic leaves}
\label{sec:action}

Given a groupoid $G$, the set of objects $M$ is partitioned into isomorphism classes. If $F \subseteq M$ is one such isomorphism class, then $G$ acts on $F$ in the sense that there is a map
\begin{equation}
  a : G \times_M F \to F
\end{equation}
where $G \times_M F$ is the fiber product of $s : G \to M$ and the inclusion $F \to M$. This map sends a pair $(g,x) \in G\times F$ such that $s(g) = x$ to $t(g) \in F$. 

In the context of symplectic groupoids, the isomorphism classes are the leaves of the symplectic foliation on $(M,\pi)$. Thus a symplectic groupoid acts on the leaves of the symplectic foliation.

\subsection{Everything is a Lagrangian}
\label{sec:everything}

The structure of a symplectic groupoid is very rich from the point of view of Lagrangian correspondences. The condition that the symplectic structure be multiplicative translates into the condition that the graph of multiplication is a Lagrangian correspondence. Let us recall the notation that if $(G,\omega)$ is a symplectic manifold, then we write $\overline{G}$ to mean $(G,-\omega)$.

\begin{prop}
\label{prop:everything-lagrangian}
  Let $(G,M,s,t,e,i,m,\omega)$ be a symplectic groupoid. Then
  \begin{enumerate}
  \item the graph of multiplication
    \begin{equation*}
      \mathsf{m} = \{(x,y,z)\mid z = m(x,y)\}
    \end{equation*}
    is Lagrangian in $\overline{G}\times \overline{G} \times G$;
  \item the identity section $e : M \to G$ is a Lagrangian embedding, so 
    \begin{equation*}
      \mathsf{e} = \{e(x) \mid x \in M\}
    \end{equation*}
    is a Lagrangian submanifold of $G$;
  \item the inversion map $i : G\to G$ is an antisymplectomorphism, so its graph
    \begin{equation*}
      \mathsf{i} = \{(x,i(x)) \mid x \in G\}
    \end{equation*}
    is Lagrangian in $G \times G$.
  \end{enumerate}
\end{prop}

An analogous result holds for the action of $(G,\omega)$ on the symplectic leaves of $M$.

\begin{prop}
  Let $G,M$, etc., as above, and let $F \subseteq M$ be an isomorphism class. Then $F$ carries a symplectic structure $\omega_F$ induced by the Poisson structure of $M$, and the graph of the action
  \begin{equation*}
    \mathsf{a} = \{(g,x,y) \mid s(g) = x, t(g) = y\} = \bigcup_{x,y \in F} \Hom_G(x,y)
  \end{equation*}
  is Lagrangian in $\overline{G}\times \overline{F} \times F$.
\end{prop}

\section{\textcolor{highlight}{Fukaya categories, functors, and monoidal structures}}
\label{sec:functors}

\subsection{Categories}

In this section we shall recall in brief outline how functors between Fukaya categories arise from Lagrangian correspondences. The culmination of this theory is meant to be a $2$-category (or even $(\infty,2)$-category) whose objects are symplectic manifolds, which enhances Weinstein's original proposal \cite{weinstein-category}. \textcolor{highlight}{For a fuller treatment of the construction of this $2$-category, see the articles by Wehrheim, Woodward, Ma'u and Bottman \cite{ww-quilted,ww-functoriality,ww-pseudoholomorphic-quilts,ww-composition,mww,bottman-wehrheim,bottman,bottman-witch,bottman-carmeli}, as well as the treatment by Fukaya \cite{fukaya-correspondence}.}

Given a symplectic manifold $X$, the Fukaya $A_{\infty}$-category $\cF(X)$ is a triangulated $A_\infty$-category that is generated by Lagrangian branes. A Lagrangian brane has a geometric support, which is a Lagrangian submanifold $L \subset X$. The passage from Lagrangian branes to Lagrangian submanifolds is neither one-to-one nor onto in general. Given two Lagrangian branes $K,L$ whose supports intersect transversely, their morphisms form a cochain complex,
\begin{equation}
  \shom^*(K,L) = \bigoplus_{p\in K\cap L} \K_p
\end{equation}
where $\K_p$ is a certain one-dimensional vector space attached to the intersection point $p \in K\cap L$ (its precise definition depends on the brane structures). There is a differential
\begin{equation}
  \mu^1 : \shom^*(K,L) \to \shom^*(K,L)[1]
\end{equation}
defined by counting inhomogeneous pseudo-holomorphic strips with boundary on $K$ and $L$. The next piece of structure is the composition,
\begin{equation}
  \mu^2 : \shom^*(L_1,L_2) \otimes \shom^*(L_0,L_1) \to \shom^*(L_0,L_2),
\end{equation}
defined for a triple $(L_0,L_1,L_2)$ of branes by counting inhomogeneous pseudo-holomorphic triangles. This composition is not necessarily associative at chain level, but the failure of associativity is trivialized by a homotopy operator $\mu^3$ that takes $3$ inputs, has degree $-1$, and counts pseudo-holomorphic quadrilaterals. This hierarchy continues to all orders with higher homotopies $\mu^k$ that take $k$ inputs, have degree $2-k$, and count pseudo-holomorphic $(k+1)$-gons. The system of identities that these operators satisfy are called the $A_\infty$-equations.

Once the $A_{\infty}$-category of Lagrangian branes $\cB(X)$ is set up as in the previous paragraph, the full Fukaya $A_{\infty}$-category $\cF(X)$ is constructed from this category by a formal enlargement process that adds all sums, shifts, cones, and summands of objects. One way to define it is to consider the Yoneda embedding (see below) of $\cB(X)$ into $\Mod \cB(X)$, the category of modules over $\cB(X)$, and take \textcolor{highlight}{idempotent complete triangulated envelope} of the image.

\subsection{Modules}

Given an $A_\infty$-category $\cA$, we can form its category of modules
\begin{equation}
  \Mod \cA = \Fun(\cA^{\op},\Ch_\K)
\end{equation}
The objects are $A_\infty$-functors from $\cA$ to the differential graded category of chain complexes over $\K$, and $\Mod \cA$ is an $A_\infty$-category in its own right. Rather than spelling this out, we give the paradigmatic example from which the general definition can be inferred. Given an object $L \in \Ob \cA$, we can define a module $\cY_L$, the Yoneda module of $L$, whose value on the object $K \in \Ob \cA$ is
\begin{equation}
  \cY_L(K) = \shom^*(K,L)
\end{equation}
Analogously to the case of ordinary categories, the association $L\mapsto \cY_{L}$ extends to an $A_{\infty}$-functor $\cA \to \Mod \cA$ that is a quasi-equivalence onto its image.

\subsection{Lagrangian correspondences and the 2-category $\Sympcat$}

Now suppose that $(X,\omega_X)$ and $(Y,\omega_Y)$ are two symplectic manifolds. Denote by $\overline{X}$ the symplectic manifold $(X,-\omega_X)$, and by $\overline{X} \times Y$ the Cartesian product with the symplectic form $(-\omega_X) \times \omega_Y$. A Lagrangian submanifold $C \subset \overline{X} \times Y $ is called a Lagrangian correspondence from $X$ to $Y$. The project initiated by Wehrheim and Woodward around the year 2007 associates to a correspondence equipped with a brane structure a functor \textcolor{highlight}{\cite[Theorem 1.1]{mww}}\footnote{\textcolor{highlight}{The Ma'u-Wehrheim-Woodward construction yields a functor $\cF^{\sharp}(C) : \cF^{\sharp}(X) \to \cF^{\sharp}(Y)$, where $\cF^{\sharp}$ denotes the extended Fukaya $A_{\infty}$-category. In \eqref{eq:functor-from-corr} we are pre-composing this with the inclusion $\cF(X) \to \cF^{\sharp}(X)$, and post-composing with the functor $\cF^{\sharp}(Y) \to \Mod \cF(Y)$. This perspective is also used in \cite{fukaya-correspondence,gao-correspondence}.}}
\begin{equation}
  \label{eq:functor-from-corr}
  \cF(C) : \cF(X) \to \Mod \cF(Y)
\end{equation}
At the object level, this functor can be described neatly by declaring that, for $L \in \Ob \cF(X)$, the $\cF(Y)$-module $\cF(C)(L)$ has as its value on $K \in \Ob \cF(Y)$ the complex
\begin{equation}
  \cF(C)(L)(K) = \hom^*_{\overline{X} \times Y}(L \times K ,C)
\end{equation}
which is nothing but the Floer cochain complex computed in the product $\overline{X} \times Y$. While neat, this formulation is not well-suited to understanding the higher components of the module $\cF(C)(L)$. It is more effective to switch reformulate this complex as a so-called quilted Floer cohomology cochain complex $QCF^*(K,C,L)$ that depends on the three objects $K,C,L$ and whose differential counts quilted strips with an interior seam. Then one defines the higher components of the $A_\infty$-module $\cF(C)(L)$ by counting certain quilted polygons. Then one must show that the resulting modules $\cF(C)(L)$ are functorial with respect to $L$ in the $A_\infty$ sense, and this involves another class of quilted surfaces.

There are \textcolor{highlight}{several} more layers to the story that are related to one another. The first layer is the question of whether the module $\cF(C)(L) \in \Ob \Mod \cF(Y)$ is representable, which is to say, whether this module is equivalent to one of the form $\cY_K$ for some $K \in \Ob \cF(Y)$. There is a natural candidate for the Lagrangian submanifold of $Y$ on which the representing object could be supported, namely the geometric composition
\begin{equation}
  C \circ L = \{y \in Y \mid (\exists x \in L)((x,y) \in C)\}
\end{equation}
This need not be a manifold, but when it is, one strives to prove that there is a brane structure on it such that it represents $\cF(C)(L)$. Such a result is known in the subject as a ``geometric composition theorem.'' The second layer is to understand the sense in which $\cF(C)$ is functorial with respect to $C$. Namely, given correspondences $C_1 \subset \overline{X} \times Y$ and $C_2 \subset \overline{Y} \times Z$, to understand whether the composition $\cF(C_2)\circ \cF(C_1)$ be expressed in terms of the geometric composition $C_2 \circ C_1$ of the correspondences. \textcolor{highlight}{This is often possible when the composition $C_{2}\circ C_{1}$ is transverse and embedded; a general result of this form is given by \cite[Theorem 1.2]{mww}. Another approach is due to Lekili-Lipyanskiy \cite{lekili-lipyanskiy,lekili-lipyanskiy-corrigendum}.}

\textcolor{highlight}{The third layer is to express the functoriality of the composition operation itself as a $A_{\infty}$-bifunctor
  \begin{equation}
    \label{eq:mww-composition}
    \circ : \cF(\overline{X} \times Y) \times \cF(\overline{Y} \times Z) \to \cF(\overline{X} \times Z),
  \end{equation}
  but there now may be a problem precisely because geometric composition is not always possible. Ma'u-Wehrheim-Woodward solve this by passing from the Fukaya $A_{\infty}$-category $\cF(\overline{X}\times Y)$ to its extended version $\cF^{\sharp}(X,Y)$ whose objects are chains of correspondences from $X$ to $Y$. Another approach is to enlarge $\cF(\overline{X}\times Y)$ to the $A_{\infty}$-category of $A_{\infty}$-bimodules $\bimod{\cF(X)}{\cF(Y)}$, and take \eqref{eq:mww-composition} to be the tensor product over $\cF(Y)$. The formal arguments in this paper could be made using either version. In any case we refer to \eqref{eq:mww-composition} as the \emph{Ma'u-Wehrheim-Woodward composition functor}.}
The $\infty$-categorical aspect to this construction is currently under further development by Bottman and Wehrheim \cite{bottman-wehrheim,bottman,bottman-witch,bottman-carmeli}.

The ultimate package that this line of research produces is a $2$-category $\Sympcat$ (call it the ``Weinstein-Donaldson-Fukaya-Wehrheim-Woodward category'') whose objects are symplectic manifolds, whose $1$-morphisms are Lagrangian correspondences, and whose \textcolor{highlight}{spaces of} $2$-morphisms are Floer cohomology groups. The $2$-category has additional structure that we shall make use of.
\begin{enumerate}
\item There is a monoidal structure given by the Cartesian product $(X,Y) \mapsto X \times Y$ of symplectic manifolds. The unit object is $\point$.
\item There is an involution on objects $X \mapsto \overline{X}$ that reverses the sign of the symplectic form. This extends to an involutive autoequivalence of $\Sympcat$ that is covariant with respect to $1$-morphisms and contravariant with respect to $2$-morphisms.
\item Combining the two points above, we can regard $\overline{X} \times Y$ as the internal hom object in $\Sympcat$, meaning that there is an equivalence of categories
  \begin{equation}
    \Hom_\Sympcat(\point, \overline{X} \times Y) \cong \Hom_\Sympcat(X, Y)
  \end{equation}
  This is merely a restatement of the construction of $1$-morphisms from correspondences.
\end{enumerate}

On the other hand, there is also a $2$-category of small $A_\infty$-categories $\Ainfcat$. Because any $A_\infty$-category is equivalent to a DG category, \textcolor{highlight}{we expect that $\Ainfcat$ is equivalent to the $2$-category of DG categories $\DGcat$. Such an equivalence is known to hold when $\Ainfcat$ and $\DGcat$ are considered as $(\infty,1)$-categories (see Section \ref{sec:dg-vs-ainf}).} In $\Ainfcat$, the objects are small $A_\infty$-categories, the $1$-morphisms are $A_\infty$-functors, and the $2$-morphisms are $A_\infty$-natural transformations. \textcolor{highlight}{This category has a monoidal structure given by the appropriately defined tensor product of $A_\infty$-categories (see Section \ref{sec:ainf-tensor-product}).} It also has a duality that takes a category to its opposite $\cA \mapsto \cA^{\op}$; it is not obvious but true that this can be extended to an involution on $\Ainfcat$ that is covariant with respect to $1$-morphisms and contravariant with respect to $2$-morphisms \cite[Section E.6]{drinfeld-dg}.

\textcolor{highlight}{With these concepts, and following Wehrheim-Woodward, we may conceive of the Fukaya $A_{\infty}$-category as a functor 
\begin{equation}
  \cF: \Sympcat \to \Ainfcat
\end{equation}
between $2$-categories. As Wehrheim-Woodward observed \cite{wehrheim-talks}, this is nothing but the 2-categorical Yoneda embedding for the object $\point$, since $\cF(X) = \Hom_{\Sympcat}(\point,X)$.}



\textcolor{highlight}{\begin{rmk}
  Although we have considered $\Ainfcat$ and $\DGcat$ as $2$-categories above for philosophical reasons, for our applications it suffices to consider $\Ainfcat$ and $\DGcat$ as $(\infty,1)$-categories. This is because the homotopy-coherent monoidal structures we wish to study are defined by functors that satisfy relations \emph{up to $2$-isomorphism}: non-invertible $2$-morphisms do not play a role. This is precisely the same remark that makes the inductive definition of $(\infty,n)$-categories possible \cite[p.~6]{HTT}.
\end{rmk}}

\begin{rmk}
  A different way of describing the abstract structure that governs the $2$-category $\Sympcat$, borrowed from physics, is as a \emph{two-dimensional topological field theory with boundary conditions and codimension-one defects.} Codimension-one defects, \textcolor{highlight}{called \emph{interfaces} by Gaiotto-Moore-Witten \cite{GMW1,GMW2} and also known as \emph{domain walls,}} correspond to the seams in quilted Floer theory. This is just another way of saying that the structure is governed by the degenerations of quilted surfaces. In this language, the geometric composition problem corresponds to the problem of colliding the defects with each other and with the boundary conditions.
\end{rmk}

\subsection{\textcolor{highlight}{Good solutions}}
\label{sec:good-sol}
We shall now make more precise what we actually need from $\Sympcat$ for our formal arguments.
It what follows, $\cF(X)$ denotes the idempotent complete triangulated Fukaya $A_{\infty}$-category of $X$. We refer to Section \ref{sec:ordinary-nerves} for the definition of the $S$-colored simplex category $\bDelta_{S}$, and to Section \ref{sec:inf2} for the definition of Segal categories.


\begin{defn}[Good solution for a fixed collection of symplectic manifolds]
  \label{defn:good-sol1}
  Let $S = \{M_{i}\}_{i \in I}$ be a set whose elements are symplectic manifolds. Let $\bDelta_{S}$ be the $S$-colored simplex category. We say that $S$ \emph{admits a good solution to the composition problem} if there is a Segal category enriched in $A_{\infty}$-categories $X : N(\bDelta_{S}^{\op}) \to \Ainfcat$ whose value on $1$-simplices is
  \begin{equation*}
    X([M_{i},M_{j}]) = \cF(\overline{M}_{i}\times M_{j}),
  \end{equation*}
  and such that the span
  \begin{equation*}
    X([M_{i},M_{j}]) \times X([M_{j},M_{k}])\leftarrow X([M_{i},M_{j},M_{k}]) \to X([M_{i},M_{k}])
  \end{equation*}
  (where the leftward arrow is a quasi-equivalence by the Segal condition) induces a functor
  \begin{equation*}
    \cF(\overline{M}_{i}\times M_{j}) \times \cF(\overline{M}_{j}\times M_{k}) \to \cF(\overline{M}_{i}\times M_{k})
  \end{equation*}
  that is homotopic to the Ma'u-Wehrheim-Woodward composition functor.
\end{defn}

\begin{defn}[Good solution for a fixed subcategory of the Weinstein category]
  \label{defn:good-sol2}
  Let $S = \{M_{i}\}_{i \in I}$ be a set of symplectic manifolds, and let $C$ be a set of correspondences between the elements of $S$ such that all geometric compositions of elements of $C$ are transverse and embedded, and such that $C$ is closed under geometric composition. Let $C(M_{i},M_{j})$ denote the subset of $C$ consisting of correspondences from $M_{i}$ to $M_{j}$. We say that $(S,C)$ \emph{admits a good solution to the composition problem} if each object of $C(M_{i},M_{j})$ may be made into an object of $\cF(\overline{M}_{i}\times M_{j})$ that defines a functor $\cF(M_{i}) \to \cF(M_{j})$, and there is a Segal category enriched in $A_{\infty}$-categories $X : N(\bDelta_{S}^{\op}) \to \Ainfcat$ whose value on $1$-simplices is
  \begin{equation*}
    X([M_{i},M_{j}]) = \langle C(M_{i},M_{i})\rangle \subset \cF(\overline{M}_{i}\times M_{j}),
  \end{equation*}
  where $\langle C(M_{i},M_{j}) \rangle $ denotes the full $A_{\infty}$-subcategory with objects $C(M_{i},M_{j})$, and such that the induced composition
  \begin{equation*}
    \langle C(M_{i},M_{j})\rangle \times \langle C(M_{j},M_{k})\rangle \to \langle C(M_{i},M_{k})\rangle
  \end{equation*}
  is homotopic to the Ma'u-Wehrheim-Woodward composition functor. 
\end{defn}

If a set $S$ of symplectic manifolds admits a good solution to the composition problem, then it is possible to construct an $\infty$-category $\Sympcat_{S}$ that is the fragment of $\Sympcat$ with these objects. To do this, we compose the given map $X : N(\bDelta_{S}^{\op}) \to \Ainfcat$ with the $A_{\infty}$-nerve construction $N_{A_{\infty}} : \Ainfcat \to \Cat_{\infty}$ (see Section \ref{sec:ainf-nerve}). Then we pass from $\Cat_{\infty}$ to $\infty$-groupoids by throwing away the noninvertible morphisms. This yields a Segal category enriched over $\infty$-groupoids, which is what is usually meant by the term ``Segal category.'' It is known that the category of Segal categories and the category of $\infty$-categories are Quillen equivalent \cite{joyal-tierney}, and passing through this equivalence we obtain an $\infty$-category that we denote $\Sympcat_{S}$. Note that $\Sympcat_{S}$ is an $(\infty,1)$-category.

If a subcategory $(S,C)$ with fixed correspondences admits a good solution, we may apply the same construction to obtain a fragment $\Sympcat_{(S,C)}$ whose objects are the elements of $S$ and whose $1$-morphisms are the elements of $C$.

\begin{rmk}
  While we shall not prove in this paper that any particular collection of symplectic manifolds admits a good solution in the sense of these definitions, we shall remark on why these are reasonable \emph{Ans\"{a}tze} for our formal arguments. Ongoing work of Bottman, Wehrheim, and others proposes to construct $\Sympcat$ as an $(A_{\infty},2)$-category, where the homotopy associativity of all parts of the structure are governed by the $2$-associahedra \cite{bottman-wehrheim,bottman,bottman-witch,bottman-carmeli}. Just as it is known that various different theories of $(\infty,1)$-categories are equivalent (see Section \ref{sec:infty-1-stuff} and \cite{bergner,joyal-tierney}), it is expected that all different theories of $(\infty,2)$-categories are equivalent. For several classes of definitions of $(\infty,n)$-categories, this is proved in \cite{bsp}.  The concept of a Segal category enriched in $A_{\infty}$-categories is simply the version that is most natural for the arguments in this paper.
  
  For our desired application to symplectic groupoids, it suffices to construct a fragment of $\Sympcat$ as an $\infty$-category. Given an $(A_{\infty},2)$-category, it may be possible to construct a $(\infty,1)$-categorical nerve (throwing away non-invertible $2$-morphisms) that will be an $\infty$-category, similar to the $A_{\infty}$-nerve constructed by Faonte and Tanaka.
\end{rmk}

\subsection{\textcolor{highlight}{From symplectic groupoids to monoid objects}}
\label{sec:new-monoids}

We now show how symplectic groupoids naturally give rise to monoid objects in various categories. We refer to the Appendix for background on the theory of monoid objects in higher category theory. \textcolor{secondrevision}{Let $\bDelta$ denote the simplex category whose objects are finite nonempty ordinals $[n] = \{0<1 < \cdots < n\}$, and whose morphisms are monotonic maps. Let $\bDelta_{a}$ denote the augmented simplex category, which also contains the empty set $[-1] = \emptyset$. We often use the notation $\mathbf{n} = [n-1]$ for objects of $\bDelta_{a}$. The category $\bDelta_{a}$ carries a monoidal structure $\oplus$ given on objects by ordinal addition. See the Appendix for more details.} Let $\Man$ denote the category of smooth manifolds and smooth maps. We begin with a fact about Lie groupoids.
\begin{prop}
  Let $(G,M,s,t,e,i,m)$ be a Lie groupoid. Then there is a simplicial object in the category of smooth manifolds $X : \bDelta^{\op} \to \Man$ whose value on $[n]$ is
  \begin{equation*}
    X([n]) = G \times_{M} G \times_{M} \cdots \times_{M} G,
  \end{equation*}
  the set of composable sequences of arrows of length $n$, whose face maps are given by projecting out the first factor, composition of two consecutive arrows, or projecting out the last factor, and whose degeneracy maps are given by inserting unit elements.
\end{prop}
Note that the definition of a Lie groupoid implies that the fiber product in this definition is transverse. This proposition is nothing but the nerve construction applied internally in the category $\Man$. The fact that $X$ is a simplicial object encodes the associativity and unitality of the groupoid structure, but not the existence of inverses.

Now let $\Symp$ denote the Weinstein ``category,'' that is, the partial category where objects are symplectic manifolds, morphisms are Lagrangian correspondences, and composition is only defined when it is transverse and embedded. It is possible for $\Symp$ to contain a genuine category, meaning a collection of morphisms such that all compositions are defined and the collection is closed under composition.

One might guess that a symplectic groupoid $G$ gives rise to a simplicial object $X: \bDelta^{\op} \to \Symp$, but this is false. The culprit are the first and last face maps, which are meant to be projections. For symplectic manifolds $X$ and $Y$, the projection $X \times Y \to X$ is not given by a Lagrangian correspondence.\footnote{In higher category theory, this observation is expressed by the statement ``the Cartesian product of symplectic manifolds is a non-Cartesian monoidal structure on $\Sympcat$.'' See Section \ref{sec:monoids-in-symp}.} Our strategy for overcoming this difficulty is to simply delete the offending face maps from the category $\bDelta$. This yields the category $\bDeltaint \subset \bDelta$ whose morphisms are monotonic maps that preserve the minimum and maximum elements. There is an isomorphism $\bDeltaint^{\op} \cong \bDelta_{a}$ with the augmented simplex category, so we are led to consider augmented cosimplicial objects instead of simplicial objects. \textcolor{secondrevision}{Forgetting the projections leads to a structure that is a bit too weak, and we remedy this by requiring that the cosimplicial object extends to a monoidal functor from $(\bDelta_{a},\oplus)$; see Definition \ref{defn:noncartmonoid}.}

\begin{prop}
  The assignment \textcolor{secondrevision}{$X_{G}(\mathbf{0}) = \point$, $X_{G}(\mathbf{n}) = G^{\times n}$} may be extended to an augmented cosimplicial object $X_{G} : \bDelta_{a} \to \Symp$:
  \begin{equation}
    \label{eq:cosimp-main}
    \xymatrix { \point \ar[r] & G \ar@<1ex>[r]  \ar@<-1ex>[r]  &  G\times G \ar@<0ex>[l] \ar@<0ex>[r] \ar@<2ex>[r] \ar@<-2ex>[r] &G \times G \times G \ar@<1ex>[l] \ar@<-1ex>[l] } \cdots
  \end{equation}
  where the left-to-right correspondences insert $\mathsf{e} : \point \to G$, and the right-to-left correspondences come from applying $\mathsf{m} : G \times G \to G$ to two consecutive factors. \textcolor{secondrevision}{Furthermore, $X_{G}$ extends to a monoidal functor $(\bDelta_{a},\oplus) \to (\Symp,\times)$. In other words, $X_{G}$ is a monoid object in $(\Symp,\times)$.}
\end{prop}
\begin{proof}
  The cosimplicial identities are direct consequences of the associativity and unitality of the groupoid. The transversality of the compositions is guaranteed by the fact that $s, t: G \to M$ are submersions. \textcolor{secondrevision}{The fact that $X_{G}$ extends to a monoidal functor reflects the fact that all of the coface and codegeneracy maps are generated by the unit $\mathsf{e} : X_{G}(\mathbf{0}) \to X_{G}(\mathbf{1})$ and the composition $\mathsf{m} : X_{G}(\mathbf{2}) \to X_{G}(\mathbf{1})$. The natural isomorphism $J_{\mathbf{n},\mathbf{m}}: X_{G}(\mathbf{n})\times X_{G}(\mathbf{m}) \to X_{G}(\mathbf{n} \oplus \mathbf{m})$ is just the associativity constraint for the Cartesian product of manifolds.}
\end{proof}

In the Appendix, we present the definition of a monoidal $\infty$-category and a $\otimes$-monoid object in a monoidal $\infty$-category $(\calC,\otimes)$, which we take to be either $(\Sympcat,\times)$ or $(\Ainfcat,\otimes)$. This discussion culminates in Definitions \ref{defn:monoid-in-symp} and \ref{defn:monoid-in-ainfcat}, which may be read now. This definition is not entirely standard, but it is an $\infty$-categorical variation of the definition of a homotopy monoid introduced by Leinster \cite{leinster}; the Appendix derives it from standard notions in higher category theory.

The definitions may be summarized as saying that the diagram \eqref{eq:cosimp-main} in $\Symp$ may be lifted to a homotopy coherent diagram in $\Sympcat$ or $\Ainfcat$. The structure is encoded by an augmented cosimplicial object $X : N(\bDelta_{a}) \to \calC$ from the \emph{nerve} of $\bDelta_{a}$ to $\calC$, \textcolor{secondrevision}{together with the data of an extension to a monoidal functor from $(N(\bDelta_{a}),\oplus)$}. This is a way of encoding homotopy associativity: the cosimplicial identities that encode associativity and unitality are not required to hold strictly, but the higher-dimensional cells in $N(\bDelta_{a})$ index a family of coherent homotopies.

\begin{thm}
  \label{thm:monoid-main}
  Let $G$ be a symplectic groupoid. Suppose that the set $S = \{G^{\times n} \mid n \ge 0\}$ consisting of all Cartesian powers of $G$ admits a good solution to the composition problem. Then $G$ is a $\times$-monoid object in $\Sympcat_{S}$.

  Suppose either the previous hypothesis, or merely that the subcategory $(S,C)$ of $\Symp$ obtained as the image of the map $X_{G}: \bDelta_{a} \to \Symp$ admits a good solution to the composition problem. Also assume that each correspondence $G^{\times n} \to G^{\times m}$ defines a functor $\cF(G)^{\otimes n} \to \cF(G)^{\otimes m}$. Then $\cF(G)$ is an $\otimes$-monoid object in $\Ainfcat$. 
\end{thm}

\begin{proof}
  Suppose that $S$ admits a good solution. This implies that $\Sympcat_{S}$ exists as an $\infty$-category. Consider the ordinary category $\bDelta_{a}$ as a subcomplex of its nerve $N(\bDelta_{a})$. The assumption that our correspondences define functors says that the augmented cosimplicial object $X_{G} : \bDelta_{a} \to \Symp$ may be lifted to a map $X^{1}: \bDelta_{a} \to \Sympcat_{S}$. Because our correspondences are geometrically composable, and this composition is associative up to coherent homotopy (by the axioms of the Segal category from which $\Sympcat_{S}$ was constructed), this map may be extended over the higher cells in $N(\bDelta_{a})$ to obtain a map $X: N(\bDelta_{a}) \to \Sympcat_{S}$. \textcolor{secondrevision}{Since the original functor $X_{G}: (\bDelta_{a},\oplus) \to (\Symp,\times)$ is monoidal, we can choose these extensions over the higher cells so that $X_{G}$ extends to a morphism of monoidal $\infty$-categories $(N(\bDelta_{a}),\oplus) \to (\Sympcat_{S},\times)$, which is to say a $\times$-monoid object in $\Sympcat_{S}$.}

  In the case where we merely assume that $(S,C)$ admits a good solution, we use the fragment $\Sympcat_{(S,C)}$ that has a restricted class of $1$-morphisms. Because we have assumed that all elements of $C$ define functors, we have a map of $\infty$-categories $\Sympcat_{(S,C)} \to \Ainfcat$ that takes $G^{\times n}$ to $\cF(G)^{\otimes n}$. As before, we may lift the map $X_{G} : \bDelta_{a} \to \Symp$ to a map $X_{G}^{1} : \bDelta_{a} \to \Ainfcat$, and use the homotopy coherent solution to the composition problem to extend \textcolor{secondrevision}{$X_{G}$ to a monoidal functor $(N(\bDelta_{a}),\oplus) \to (\Ainfcat,\otimes)$.}
\end{proof}

\begin{rmk}
  The condition that each correspondence defines a functor $\cF(G)^{\otimes n} \to \cF(G)^{\otimes m}$ may be simplified. Every morphism $f : [n] \to [m]$ in $\bDelta_{a}$ is described by choosing a $k$-element subset of $[m]$ that is the image, and a partition of the domain $[n]$ into $k$ intervals on which $f$ is constant. The corresponding Lagrangian correspondence $G^{\times (n+1)} \to G^{\times (m+1)}$ is constructed by inserting $\mathsf{e}: \point \to G$ for each element not in the image of $f$, and inserting the $\ell$-fold composition correspondence $G^{\times \ell}\to G$ for each interval of length $\ell$ (when $\ell = 1$ this is the diagonal, and when $\ell = 2$, it is $\mathsf{m}$). Thus every correspondence in the diagram is obtained as the $\times$-product of the identity and $\ell$-fold composition correspondences, and all that is needed is that the $\ell$-fold composition defines a functor $\cF(G)^{\otimes \ell} \to \cF(G)$.
\end{rmk}

The moral of this story is that the problem of constructing homotopy coherent monoid objects is a subproblem of the general problem of homotopy coherent composition for Lagrangian correspondences.

Now we turn to the corresponding results for module objects. The proofs are formally the same as what has been done above. For monoid-module pairs, the relevant indexing \textcolor{secondrevision}{object is the pair $(\bDelta_{a},\bDelta_{a}^{+})$ consisting of the monoidal category $\bDelta_{a}$ and its module category $\bDelta_{a}^{+}$. The category $\bDelta_{a}^{+}$ contains} $\bDelta_{a}$ but has one more face map at each level to encode the action on the module. \textcolor{secondrevision}{See Definition \ref{defn:noncartmonoidmodulepair} for the concept of $\otimes$-monoid-module pair in a monoidal $\infty$-category $(\calC,\otimes)$.}

\begin{thm}
  Let $G$ be a symplectic groupoid and let $F$ be a symplectic leaf of the underlying Poisson structure. There is a diagram $X_{(G,F)} : \textcolor{secondrevision}{\bDelta_{a}^{+}} \to \Symp$ of the form
  \begin{equation}
    \xymatrix { F \ar[r]  & G \times F \ar@<1ex>[r] \ar@<-1ex>[l] \ar@<-1ex>[r]  &  G \times G \times F \ar@<-2ex>[l] \ar@<0ex>[l] }\cdots
  \end{equation}
  whose restriction to $\bDelta_{a} \subset \textcolor{secondrevision}{\bDelta_{a}^{+}}$ is obtained by taking $X_{G}$ times the fixed manifold $F$, and such that the additional face maps are given by the action correspondence $\mathsf{a} : G\times F \to F$.

  If the set $S = \{G^{\times n}, G^{\times n} \times F \mid n \geq 0\}$ admits a good solution of the composition problem, then this diagram may be lifted to a map $X : N(\bDeltaint^{+})^{\op} \to \Sympcat$ making $(G,F)$ into a $\times$-monoid-module pair in $\Sympcat$.

  If at least the subcategory of $\Symp$ consisting of the correspondences appearing in this diagram admits a good solution, and each correspondence defines a functor $\cF(G)^{\otimes n} \otimes \cF(F) \to \cF(G)^{\otimes m}\otimes \cF(F)$, then $(\cF(G),\cF(F))$ has the structure of a $\otimes$-monoid-module pair in $\Ainfcat$.
\end{thm} 

\begin{rmk}
  One may also formulate the notion of monoid objects for the \emph{Cartesian} monoidal structures on $\Sympcat$ and $\Ainfcat$. The Cartesian monoidal structure on $\Ainfcat$ is the product of categories. The Cartesian monoidal structure on $\Sympcat$ is the disjoint union of symplectic manifolds: because $\overline{X} \times (Y \coprod Z) = (\overline{X} \times Y) \coprod (\overline{X} \times Z)$, a Lagrangian correspondence $X \to Y \coprod Z$ is the same thing as a pair of Lagrangian correspondences $X \to Y$ and $X \to Z$ (provided we allow $\emptyset$ as a Lagrangian submanifold). This would allow us to use the simpler definition of monoid objects in an $\infty$-category (Definition \ref{defn:cartmonoid-oo}). The homotopy coherence problem is not really any easier if we do this, and interpreting the correspondence $\mathsf{m} \subset \overline{G} \times \overline{G} \times G$ in this formalism is challenging. This is why we prefer to formulate things a way that is formally more sophisticated but geometrically more natural.
\end{rmk}

\begin{rmk}
  \label{rmk:size}
  \textcolor{highlight}{There are situations where the idempotent complete triangulated Fukaya $A_{\infty}$-category $\cF(G)$ is not large enough for the monoidal structure to be defined,} but it may nevertheless happen that $\Mod \cF(G)$ admits a monoidal structure even though $\cF(G)$ does not. An example of this phenomenon is furnished by the wrapped Fukaya $A_{\infty}$-category of a cotangent bundle $T^{*}M$ regarded as a symplectic groupoid over $M$ (see section \ref{sec:cotangent} below). 
\end{rmk}

\begin{rmk}
  It seems that the first person to observe the expected existence of these structures was M.~Gualtieri \cite{gualtieri-talk} around the year 2009 when the quilt theory first came into general use in symplectic topology.
\end{rmk}

\subsection{\textcolor{highlight}{Monoid objects with duality}}
\textcolor{highlight}{So far we have not used the inversion correspondence $\mathsf{i} \subset G \times G$. We will now study how this correspondence equips the monoid objects $G$ and $\cF(G)$ with a duality involution. In this section, we will suppress questions of homotopy coherence, although it should be possible to extend the notion of homotopy-coherent monoid objects presented in the Appendix to include duality by enlarging the indexing category $\bDelta_{a}$. The reader may therefore assume that in this section we are dealing with cohomology-level Fukaya categories (what Wehrheim-Woodward refer to as the Donaldson category).}


We shall work in a $2$-category $\cC$ which has a symmetric monoidal product $\times$ and unit object $\point$. In the applications $\cC$ is either $\Sympcat$ or $\Ainfcat$.  To treat the $2$-morphisms properly, we assume that $\cC$ is not merely symmetric monoidal, but is also equipped with an involution $X \mapsto \overline{X}$ that is covariant with respect to $1$-morphisms and contravariant with respect to $2$-morphisms. We call such an involution a \textcolor{secondrevision}{$\co$-involution}\footnote{\textcolor{secondrevision}{Recall that a $2$-category has three different duals: $\cC^{\co}$ where the $2$-morphisms reversed, $\cC^{\op}$ where the $1$-morphisms are reversed, and $\cC^{\coop}$ where both $1$- and $2$-morphisms are reversed. A $\co$-involution is so called because it is a $2$-functor $\cC \to \cC^{\co}$.}}, and we assume it is compatible with the symmetric monoidal structure in the natural sense. We also assume that $\overline{X} \times Y$ is an internal Hom object, meaning that $\Hom_\cC(\point,\overline{X} \times Y) \cong \Hom_\cC(X,Y)$. There is then a canonical $1$-morphism $\Delta_{X} \in \Hom_{\cC}(\point, \overline{X}\times X)$ corresponding to the identity $1$-morphism of $X$. We further assume that given two $1$-morphisms $E,F \in \Hom_{\cC}(\point, \overline{X}\times X) \cong \Hom_{\cC}(X,X)$, we have isomorphisms
\begin{equation}
  \Homtwo(E,F) \cong \Homtwo(\Delta_{X},\overline{E}\times F) \cong \Hom(E\times \overline{F},\Delta_{\overline{X}}).
\end{equation}
These assumptions are all reasonable when thinking of $1$-morphisms as correspondences.

\begin{defn}
  Let $(\cC,\times, \point, \overline{(\cdot)})$ be a symmetric monoidal $2$-category with \textcolor{secondrevision}{$\co$-involution} as above. A \emph{monoid-with-duality object} in $\cC$ is an object $G$, together with $1$-morphisms $m : G \times G \to G$ and $e : \point \to G$, and a $1$-isomorphism $i : G \to \overline{G}$, together with a collection of $2$-isomorphisms
    \begin{equation}
    \label{eq:grouplike2cat}
    \begin{aligned}
      m \circ (1_G \times m) & \cong m \circ (m \times 1_G), \\
      m \circ (1_G \times e) & \cong 1_G,\\
      m \circ (e \times 1_G) & \cong 1_G,\\
      i \circ m &\cong \overline{m} \circ (i\times i) \circ \tau_G,\\
      m \circ (\overline{i} \times 1_{G}) \circ \Delta_G &=  e,\\
      m \circ (1_{G}\times \overline{i}) \circ \Delta_{\overline{G}} &= e,
    \end{aligned}
  \end{equation}
  where $\tau_G : G \times G \to G \times G$ is the map that swaps the factors, and $\Delta_{G} \in \Hom_{\cC}(\point,\overline{G}\times G)$ corresponds to the identity $1$-morphism $1_{G} \in \Hom_{\cC}(G,G)$.
\end{defn}

\begin{rmk}
  \label{rmk:hopf}
  This notion differs from the proper notion of group object in a monoidal category, which is the notion of a \emph{Hopf algebra}. A Hopf algebra is also equipped with a comultiplication $\Delta : G \to G\times G$ and a counit $\epsilon : G \to \point$ that satisfy several compatibility relations with $m$, $e$, and $i$. In this context $i$ is usually called the antipode. This structure does not exist on a general symplectic groupoid; see section \ref{sec:drinfeld-double} for further discussion.
\end{rmk}

\begin{prop}
\label{prop:monoid-to-monoidal}
  Suppose that $(G,m,e,i)$ is a monoid-with-duality object in a $2$-category $(\cC,\times,\point,\overline{(\cdot)})$ as above. Let $\cG = \Hom_\cC(\point,G)$ be the category of $1$-morphisms from the unit object to $G$. Then $m$ and $e$ induce monoidal structures $(\otimes, \cO)$ on $\cG$ and $\cG^\op$, and $i$ induces an equivalence $\cD : \cG \to \cG^{\op}$ such $\cD(E\otimes F) \cong \cD F \otimes \cD E$ for any objects $E$ and $F$ of $\cG$. For each object $E$ of $\cG$ there is an evaluation $2$-morphism $E\otimes \cD\overline{E} \to \cO$ and a coevaluation $2$-morphism $\cO \to E \otimes \cD\overline{E}$.
\end{prop}
\begin{proof}
  First observe that because $X \mapsto \overline{X}$ is a \textcolor{secondrevision}{$\co$-involution}, it induces an equivalence of categories
  \begin{equation}
    \Hom_\cC(\point,\overline{G}) \cong \Hom_\cC(\point,G)^\op = \cG^\op
  \end{equation}
  The operation $\otimes$ is defined by the composition
  \begin{equation}
    \Hom_\cC(\point,G)\times \Hom_\cC(\point,G) \to \Hom_\cC(\point, G\times G) \to \Hom_\cC(\point,G)
  \end{equation}
  where the second arrow is post-composition with $m$. The unit object $\cO$ is $e$ regarded as an object of $\Hom_\cC(\point,G)$. The first three axioms in \eqref{eq:grouplike2cat} imply that this is a monoidal structure. By taking $\overline{m}$ and $\overline{e}$, we also see that $\overline{G}$ is a monoid object, so $\cG^{\op}$ also has a monoidal structure, which we denote by the same symbols ($\otimes$,$\cO$).

  The statement that $\cD(E\otimes F) \cong \cD F \otimes \cD E$ follows from the fourth axiom in \eqref{eq:grouplike2cat}.

  The evaluation $2$-morphisms are obtained as follows. Given $E \in \cG$, let $\overline{E} \in \cG^{\op}$ denote the same object regarded as belonging to the opposite category. Apply the ambient symmetric monoidal product $\times$ to obtain a $1$-morphism $E \otimes \overline{E} : \point = \point\times \point \to G \times \overline{G}$. On the other hand, we have the diagonal $1$-morphism $\Delta_{\overline{G}}: \point \to G\times \overline{G}$. Post-composing these $1$-morphisms with $m \circ (1_{G} \times \overline{i})$, we obtain
  \begin{equation*}
    \begin{aligned}
      m \circ (1_{G} \times \overline{i}) &\circ \Delta_{\overline{G}} = e = \cO,\\
      m \circ (1_{G} \times \overline{i}) &\circ (E \times \overline{E}) = E\otimes \cD\overline{E}
    \end{aligned}
  \end{equation*}
  Because composition with a $1$-morphism is functorial with respect to $2$-morphisms, we obtain a map on $2$-morphism spaces:
  \begin{equation}
    \label{eq:two-hom-map}
    m \circ (1_{G} \times \overline{i}) : \Homtwo(E \times \overline{E},\Delta_{\overline{G}}) \to \Homtwo(E \otimes \cD\overline{E}, \cO). 
  \end{equation}
  The domain of this map may be identified with $\Hom_{\cG}(E,E)$, and taking the image of $1_{E} \in \Hom_{\cG}(E,E)$ under \eqref{eq:two-hom-map} yields a distinguished $2$-morphism $E\otimes \cD\overline{E} \to \cO$, which is our evaluation map. The coevaluation map is constructed by considering $2$-morphisms from $\Delta_{\overline{G}}$ to $E \times \overline{E}$ instead.\end{proof}

\begin{rmk}
In order for $\cD \overline{E}$ to really be the dual of $E$, these evaluation and coevaluation maps should satisfy compatibility relations, essentially saying that there are isomorphisms $\Homtwo(E,E) \cong \Homtwo(\cO, E\otimes \cD\overline{E})$. When this holds, $\cG$ then carries the structure of a \emph{rigid monoidal category}. Since we shall not use this property, we shall defer this question.   
\end{rmk}


In light of the foregoing discussion, we now have the following enhancement of Theorem \ref{thm:monoid-main}:
\begin{prop}
  \label{prop:groupoid-to-monoidal}
  Let $(G,M,s,t,e,i,m,\omega)$ be a symplectic groupoid. Then $G$ is a monoid-with-duality object in the Weinstein category $\Symp$ of symplectic manifolds and Lagrangian correspondences. Assuming a good solution the composition problem for $S = \{G^{\times n}\mid n \geq 0\}$, $G$ can be made into a monoid-with-duality object in the $2$-category $\Sympcat$, and hence the cohomology-level Fukaya category $H^{0}(\cF(G))$ admits a monoidal structure $(\otimes, \cO)$ and $\op$-equivalence $\cD : H^{0}(\cF(G)) \to H^{0}(\cF(G))^{\op}$ as in Proposition \ref{prop:monoid-to-monoidal}.
\end{prop}

\begin{proof}
  Given a symplectic groupoid $G$, the axioms of a monoid-with-duality are precisely the associativity of $m$, unitality of $e$, and the statement that $i$ acts like an inverse; these hold because $G$ is a groupoid. This proves the first assertion. The second assertion follows because the good solution hypothesis means that these identities, formulated as geometric compositions of correspondences, can be lifted to isomorphisms of functors between Fukaya categories, at least at the cohomology level.
\end{proof}

\subsection{Drinfeld doubles and Hopf algebra objects}
\label{sec:drinfeld-double}

As mentioned in remark \ref{rmk:hopf} above,
the ``true'' notion of a group-like object in a category is that of a Hopf algebra.
In the context of the 2-category $\Sympcat$, this means that in addition to the composition $\mathsf{m} : G \times G \to G$, the unit $\mathsf{e} : \point \to G$, and the inversion $\mathsf{i}: G \to \overline{G}$, there is also a comultiplication $\delta : G \to G\times G$ and counit $\epsilon : G \to \point$. In this context the inversion is called the antipode. The comultiplication and counit must satisfy the duals of the associativity and unitality axioms. There is a compatibility between the multiplication and comultiplication,
\begin{equation}
  \delta \circ m = (m\times m)\circ (1_{G} \times \tau_{G} \times 1_{G})\circ (\delta \times \delta).
\end{equation}
For a general symplectic groupoid, there does not seem to be any way to construct $\delta$ satisfying this axiom.

However, it is possible to construct $\delta$ for the case $G = T^{*}K$ is the cotangent bundle of a compact Lie group $K$. As we shall see in the next section, this symplectic manifold is a symplectic groupoid in \emph{two ways}, once as the cotangent bundle of a manifold, and second as the symplectic integration of the canonical Poisson structure on the dual of the Lie algebra $\mathfrak{k}^{*}$. We can use one symplectic groupoid structure to define the multiplication, and the transpose of the other to define the comultiplication. Since there are two choices for which groupoid structure corresponds to multiplication, we get two dual Hopf algebra objects in $\Sympcat$.

There is actually a natural source of symplectic manifolds carrying two compatible groupoid structures as $G = T^{*}K$ does. Namely, one takes a Poisson-Lie group $(H,\pi)$, where $H$ is a Lie group and $\pi$ is a Poisson structure on the underlying manifold of $H$ such that the group operation $H \times H \to H$ is a Poisson map. A symplectic integration $(G,\omega)$ of $(H,\pi)$ (which will necessarily have twice the dimension of $H$) is called a \emph{Drinfeld double} of $(H,\pi)$. There is a dual Poisson-Lie group $(H^{\vee},\pi^{\vee})$ and $(G,\omega)$ is also a symplectic integration of $(H^{\vee}, \pi^{\vee})$ (briefly, the Lie algebra $\mathfrak{h}$ of a Poisson-Lie group is a Lie bialgebra, and there is an involution on Lie bialgebras that takes $\mathfrak{h}$ to $\mathfrak{h}^{*}$). Thus $(G,\omega)$ carries two symplectic groupoid structures, and we expect the Fukaya category $\cF(G,\omega)$ to be a Hopf algebra in the 2-category $\Ainfcat$. 

\section{Examples of monoidal Fukaya categories}
\label{sec:examples}

In this section we run through several examples of symplectic groupoids and their associated monoidal Fukaya categories. We organize the examples according to the underlying Poisson structure.

\subsection{The case of $\pi$ nondegenerate}
\label{sec:case-symp}

Let $(M,\pi)$ be a nondegenerate Poisson manifold. Then taking $\omega = \pi^{-1}$, we get a symplectic manifold $(M,\omega)$. The most obvious symplectic integration of this Poisson manifold is $G = \overline{M} \times M$. This is known as the \emph{pairs groupoid}. 

Now consider the Fukaya $A_{\infty}$-category $\cF(\overline{M}\times M)$. This is precisely the category of Lagrangian correspondences $M \to M$. Such correspondences induce $A_{\infty}$-endofunctors of $\cF(M)$, and thus there is a functor
\begin{equation}
  \cF(\overline{M}\times M) \to \Fun(\cF(M),\cF(M)).
\end{equation}
Alternatively, \textcolor{highlight}{objects in $\cF(\overline{M}\times M)$ give rise to $A_{\infty}$-bimodules over $\cF(M)$ (modules with a left and a right action by $\cF(M)$), yielding a $A_{\infty}$-functor
\begin{equation}
  \cF(\overline{M}\times M) \to \bimod{\cF(M)}{\cF(M)}
\end{equation}
From a bimodule, one recovers an endofunctor by tensoring with the bimodule. }

\textcolor{highlight}{On $\cF(\overline{M} \times M)$:
\begin{itemize}
\item The monoidal product $\otimes$ is composition of correspondences in $\cF(\overline{M}\times M)$.
\item The unit object $\cO$ is the diagonal in $\overline{M}\times M$.
\item The duality $\cD$ is transpose of correspondences in $\overline{M}\times M$.
\end{itemize}
On $\Fun(\cF(M),\cF(M))$:
\begin{itemize}
\item The monoidal product $\otimes$ is composition of functors in $\Fun(\cF(M),\cF(M))$.
\item The unit object $\cO$ is the identity functor in $\Fun(\cF(M),\cF(M))$.
\item The duality $\cD$ is an equivalence $\Fun(\cF(M),\cF(M)) \to \Fun(\cF(M)^\op,\cF(M)^\op)$. Note that the latter category is indeed naturally regarded as the opposite of the former \cite[Section E.6]{drinfeld-dg}.
\end{itemize}
On $\bimod{\cF(M)}{\cF(M)}$:
\begin{itemize}
\item The monoidal product $\otimes$ is the tensor product of bimodules over $\cF(M)$.
\item The unit object $\cO$ is the bimodule $\cF(M)$ with its standard left and right actions.
\item The duality $\cD$ is an equivalence $\bimod{\cF(M)}{\cF(M)} \to \bimod{\cF(M)^\op}{\cF(M)^\op}$.
\end{itemize}}

To go some way toward justifying that these monoidal structures are indeed the same, we have the following proposition whose proof is immediate.
\begin{prop} Let $G = \overline{M} \times M$, regarded as a symplectic groupoid, and let $\mathsf{m}$, $\mathsf{e}$, and $\mathsf{i}$ be as in Proposition \ref{prop:everything-lagrangian}.
  \begin{itemize}
  \item The composition bifunctor $\cF(G) \times \cF(G) \to \cF(G)$ is represented by a Lagrangian brane in $\overline{G} \times \overline{G} \times G$ whose underlying Lagrangian submanifold is $\mathsf{m}$.
  \item The identity functor $\cF(M) \to \cF(M)$ is represented by a Lagrangian brane in $G$ whose underlying Lagrangian submanifold is $\mathsf{e}$.
  \item The transpose functor $\cF(\overline{G}) \to \cF(G)$ is represented by a Lagrangian brane in $G \times G$ whose underlying Lagrangian submanifold is $\mathsf{i}$.
  \end{itemize}
\end{prop}

In this case, there is a single isomorphism class/symplectic leaf consisting of the entire manifold $M$ itself. Thus there is an action of $\cF(G)$ on $\cF(M)$, which is nothing but the action of endofunctors of $\cF(M)$ on $\cF(M)$.

\subsection{The case of $\pi = 0$}
\label{sec:cotangent}
On any smooth manifold $M$, the zero tensor $\pi = 0$ defines a Poisson structure. A symplectic groupoid integrating this is $G = T^*M$ with its canonical symplectic form. There are no morphisms between different points of $M$, and the endomorphisms of the point $x \in M$ is $T^*_xM$ with the group structure given by addition of covectors. Thus the identity section is the zero section, and the inversion map is fiberwise negation.

Now \textcolor{highlight}{assume $M$ is real-analytic, and let us} interpret $\cF(G) = \cF(T^*M)$ as the infinitesimal (Nadler-Zaslow) Fukaya $A_{\infty}$-category. The Nadler-Zaslow correspondence \cite[Theorem 1.0.1]{nadler-zaslow}, \cite{nadler} is \textcolor{highlight}{a quasi-equivalence of $A_{\infty}$-categories} with constructible sheaves on $M$,
\begin{equation}
  \cF(G) \cong \Sh_c(M).
\end{equation}
The category $\Sh_c(M)$ has a well-known monoidal structure.
\begin{itemize}
\item The monoidal product $\otimes$ is the tensor product of constructible sheaves.
\item The unit object $\cO$ is the constant sheaf $\underline{\K}$ on $M$.
\item The duality $\cD$ is Verdier duality.
\end{itemize}

In this case, our justification that these monoidal structures correspond to each other passes through the Nadler-Zaslow correspondence: we can represent functors on $\Sh_c(M)$ as integral transforms, whose kernels are constructible sheaves on Cartesian powers of $M$. From such a kernel we obtain a Lagrangian brane from Nadler-Zaslow, and this can be compare with the structural Lagrangians of the symplectic groupoid structure. 
\begin{prop}
  Let $G = T^*M$ with the symplectic groupoid structure of fiberwise addition, and let $\mathsf{m}$, $\mathsf{e}$, and $\mathsf{i}$ be as in Proposition \ref{prop:everything-lagrangian}.
  \begin{itemize}
  \item Let $\K_{\Delta_3} \in \Sh_c(M\times M \times M)$ be the constructible sheaf that represents the tensor product on $\Sh_c(M)$. Then the singular support of $\K_\Delta$ is
    \begin{equation}
      SS(\K_{\Delta_3}) = \{(q_1,p_1,q_2,p_2,q_3,p_3) \mid q_1=q_2=q_3, \quad p_1+p_2+p_3 = 0\} \subset T^*(M\times M \times M)
    \end{equation}
    where we use $q$ to denote points of $M$ and $p$ to denote covectors. Under the isomorphism
    \begin{equation}
      T^*(M\times M \times M) \cong \overline{G} \times \overline{G} \times G
    \end{equation}
    that flips the sign of $p_1$ and $p_2$, $SS(\K_\Delta)$ corresponds to $\mathsf{m}$.
  \item Let $\K_M \in \Sh_c(M)$ be the constructible sheaf that is the unit for the tensor product. Then
    \begin{equation}
      SS(\K_M) = M \subset T^*M = G
    \end{equation}
    coincides with $\mathsf{e}$.
  \end{itemize}
\end{prop}

\textcolor{highlight}{The statement that fiberwise negation is Verdier duality is explained in \cite[Section 3.4.1]{nadler}. It is slightly difficult} to formulate in the framework of the proposition above because Verdier duality is a contravariant functor, and so it does not have a kernel \emph{per se}. On the other hand, it is natural to convert contravariant functors into covariant ones by pre-composing with Verdier duality itself. Then Verdier duality goes over to the identity functor, which is represented by a certain $\K_{\Delta_2}$ in $\Sh_c(M \times M)$, and
\begin{equation}
  SS(\K_{\Delta_2}) = \{(q_1,p_1,q_2,p_2) \mid q_1 = q_2, \quad p_1+ p_2 = 0\} \subset T^*(M\times M)
\end{equation}
and under the isomorphism $T^*(M\times M) = G\times G$, this Lagrangian corresponds to $\mathsf{i}$.


In this case, the isomorphism classes/symplectic leaves are the singleton sets $\{p \} \subset M$. The action correspondence $\mathsf{a} \subset T^*M \times \{p\} \times \{p\}$ is the cotangent fiber $T^*_pM$. The category $\cF(\{p\})$ is equivalent to $\Perf(\K)$, and the action
\begin{equation}
  \rho : \cF(T^*M) \times \Perf(\K) \to \Perf(\K)
\end{equation}
takes a constructible sheaf $\cE$ and a vector space $V$ to $\cE_p \otimes V$, the stalk of $\cE$ at $p$ tensored with $V$.

An alternative story involves a different version of the Fukaya category, namely the wrapped Fukaya $A_{\infty}$-category $\cW(T^{*}M)$. By results of Abbondandolo-Schwarz and Abouzaid \cite{abbondandolo-schwarz-loop-product,abouzaid-cotangent-generates,abouzaid-based-loops}, this category is equivalent to the DG category of DG-modules over the DG-algebra $C_{-*}(\Omega M)$ of chains on the based loop space of $M$, where the multiplication comes from concatenation of loops. More precisely, we can write
\begin{equation}
  \cW(T^{*}M) \cong \Perf C_{-*}(\Omega M)
\end{equation}
Such modules are also known as $\infty$-local systems. The monoidal operation is tensor product of local systems. One observes that this operation generally does not map a pair of objects in $\Perf C_{-*}(\Omega M)$ to one in $\Perf C_{-*}(\Omega M)$, but rather to one in $\Mod C_{-*}(\Omega M)$. Thus it is not $\cW(T^{*}M)$ itself but $\Mod \cW(T^{*}M)$ that carries a monoidal structure (compare remark \ref{rmk:size} above).

\subsection{The case of $\pi = 0$, $M$ integral affine}
\label{sec:torusbundle}
As before, we take $M$ a smooth manifold with $\pi = 0$. As symplectic integrations are not unique, it may be possible to find a symplectic integration that differs from $T^*M$. As an instance of this phenomenon, assume that $M$ carries an integral affine structure. This means that $M$ comes equipped with an atlas such that the transition functions between coordinate charts are restrictions of the action of $\GL(n,\Z) \ltimes \R^n$ on $\R^{n}$ to open sets. Then the lattices of integral covectors in each cotangent space are invariantly defined and form a local system of lattices $\Lambda_\Z^* \subset T^*M$. Another symplectic integration of the zero Poisson structure is then $G = T^*M/\Lambda_\Z^*$, where the symplectic structure descends from $T^*M$. Thus $G$ is a nonsingular Lagrangian torus fibration over $M$.

This is precisely the setting for Strominger-Yau-Zaslow mirror symmetry without corrections. This case has been studied by Aleksandar Suboti\'{c} in his 2010 Harvard Ph.D.~thesis and other unpublished work. Under mirror symmetry, $\cF(G)$ is equivalent to $D^b\Coh(X)$, where $X$ is the mirror Calabi-Yau variety. 
\begin{itemize}
\item The monoidal product $\otimes$ is the tensor product of coherent sheaves.
\item The unit object $\cO$ is the structure sheaf $\cO_X$.
\item The duality $\cD$ is coherent duality $\cHom(-,\cO_X)$. ($\cO_X$ is the dualizing sheaf on a smooth Calabi-Yau.)
\end{itemize}
\begin{prop}
  The Lagrangian branes involved in the monoidal structure of Suboti\'{c} are supported on the Lagrangian submanifolds $\mathsf{m}$,$\mathsf{e}$, and $\mathsf{i}$  that determine the groupoid structure of $G = T^*M/\Lambda_\Z^*$.
\end{prop}

Once again, the isomorphism classes/symplectic leaves are the singleton sets $\{p\} \subset M$. The action correspondence $\mathsf{a}$ is the torus fiber $L_p$, which corresponds to the skyscraper sheaf of a point on the mirror. Now $\cF(G)$ acts on $\Perf(\K)$ by tensoring with the fiber (not stalk) $\Hom(\cO_p,\cE)$ of the sheaf $\cE$ at $p$.

\subsection{The case of $\mathfrak{k}^*$}
\label{sec:gdual}

For a really interesting Poisson structure, one can consider the dual space of a Lie algebra $\mathfrak{k}$.\footnote{We use $\mathfrak{k}$ (German `\textit{k}') rather than $\mathfrak{g}$ for the Lie algebra both to indicate that we are mainly interested in compact groups and to avoid conflict with the use of the letter $G$ for the symplectic groupoid.} In this case, a symplectic integration is $G = T^*K$, where $K$ is a Lie group integrating $\mathfrak{k}$.

\textcolor{highlight}{Once again, using the Nadler-Zaslow quasi-equivalence, we may identify the infinitesimal Fukaya $A_{\infty}$-category of $G$ with $\Sh_c(K)$, constructible sheaves on $K$.} Now the monoidal structure of interest is the one that actually uses the fact that $K$ is a Lie group.
\begin{itemize}
\item The monoidal product $\otimes$ is the convolution of sheaves, namely, push-forward under the multiplication map $m : K \times K \to K$.
\item The unit object $\cO$ is the skyscraper sheaf at the identity element $e \in K$.
\item The duality $\cD$ is pull-back under the inversion map $i: K \to K$.
\end{itemize}

\begin{prop}
  Let $G = T^*K$, with the symplectic groupoid structure that integrates the Poisson structure on $\mathfrak{g}^*$, and let $\mathsf{m}$,$\mathsf{e}$, and $\mathsf{i}$ be as in Proposition \ref{prop:everything-lagrangian}.
  \begin{itemize}
  \item The singular support of the kernel for convolution corresponds to $\mathsf{m}$ under the map $T^*(K^3) \cong (T^*K)^3 \cong \overline{T^*K}\times \overline{T^*K} \times T^*K$ that flips the sign of the covector on the first two factors.
  \item The singular support of the unit object for convolution corresponds to $\mathsf{e}$ in $G = T^*K$.
  \item The singular support of the inversion map composed with Verdier duality corresponds to $\mathsf{i}$ in $G \times G$.
  \end{itemize}
\end{prop}

\begin{proof}
  The kernel for convolution is the trivial local system supported on the the graph of the group multiplication.
  \begin{equation}
    \K_* = \{(h,g,f) \mid f = hg \} \subset K \times K \times K
  \end{equation}
  The singular support of $\K_*$ is the conormal bundle of this graph, as a submanifold of $T^*(K\times K \times K)$.

  On the other hand, to describe $m$, we note that, using left multiplication, we can identify $T^*K \cong K \times \mathfrak{k}^*$. In the latter description, the groupoid structure is given by the action groupoid construction for the coadjoint action of $K$ on $\mathfrak{k}^*$. Thus if $(g,\xi) \in K \times \mathfrak{k}^*$ is a morphism, we have $s(g,\xi) = \xi$ and $t(g,\xi) = \Ad^*(g)(\xi)$. The composition $m((h,\eta),(g,\xi))$ is defined if and only if $\eta = \Ad^*(g)(\xi)$, and in that case it equals $(hg, \xi)$. Thus
  \begin{equation}
    \mathsf{m} \cong \{((h,\eta),(g,\xi),(f,\zeta) \mid f = hg,\  \eta = \Ad^*(g)(\xi),\  \zeta = \xi\} \subset (K\times \mathfrak{k}^*)^3
  \end{equation}

  In order to compare the two Lagrangians, we need to trace through the isomorphisms $TK \cong K \times \mathfrak{k}$ and $T^*K \cong K \times \mathfrak{k}^*$ given by left multiplication. The isomorphism $TK \cong K \times \mathfrak{k}$ takes $(g,X) \in K \times \mathfrak{k}$ to $(g,L_g(X)) \in TK$. To compute the tangent space to the graph $\{(h,g,f) \mid f = hg\} \subset K^3$, we represent points nearby to $h$, $g$, and $f$ by $g\exp(tX)$, $h\exp(tY)$, and $f\exp(tZ)$ respectively. Then the relation
  \begin{equation}
    f\exp(tZ) = h\exp(tY)g\exp(tX)
  \end{equation}
  can be expanded with respect to $t$ to give
  \begin{equation}
    f + tfZ + \cdots = hg + t(hYg + hgX) + \cdots
  \end{equation}
  So the relation between the tangent vectors is
  \begin{equation}
    fZ = hYg + hgX
  \end{equation}
  multiplying on the left by $f^{-1} = g^{-1}h^{-1}$ yields
  \begin{equation}
    Z = g^{-1}Yg + X
  \end{equation}
  as the relation that defines the tangent space to the graph. Now suppose that the element $(\xi,\eta,\zeta) \in (\mathfrak{k}^*)^3$ annihilates this space, where $\xi$ pairs with $X$, $\eta$ with $Y$, and $\zeta$ with $Z$. Because vectors of the form $(X,Y,Z) = (Z,0,Z)$ are in the space, we must have $\zeta = -\xi$, and because vectors of the form $(X,Y,Z) = (X,-\Ad(g)X,0)$ are in the space, we must have $\xi = \Ad(g^{-1})^*(\eta)$, or $\eta = \Ad(g)^*(\xi)$. Flipping the signs of $\xi$ and $\eta$ preserves the relation $\eta = \Ad(g)^*(\xi)$ and takes $\zeta = -\xi$ to $\zeta = \xi$, and thus the relations defining the conormal bundle of the graph of multiplication correspond to the relations defining $\mathsf{m}$.

  For the unit object, this is nothing but $\K_e$, where $e \in K$ is the identity element of the group. The singular support is $T^*_eK$, which consists of the pairs $(e,\xi)$ for $\xi \in \mathfrak{k}^*$, and indeed these are the identity morphisms for the groupoid structure.

  For inversion, the singular support is the conormal bundle of the graph of the group inversion
  \begin{equation}
    \{(h,g) \mid hg = e\}\subset K\times K
  \end{equation}
  Again writing $h\exp(tY)$ and $g\exp(tX)$ for nearby elements, we find
  \begin{equation}
    g^{-1}Yg + X = 0
  \end{equation}
  is the relation defining the tangent space, which implies that the cotangent space is defined by the relation $\eta = \Ad^*(g)(\xi)$. On the other hand, in the groupoid structure on $K\times \mathfrak{k}^*$, the inverse of $(g,\xi)$ as a morphism from $\xi$ to $\Ad^*(g)(\xi)$ is simply $(g^{-1},\Ad^*(g)(\xi))$ as a morphism from $\Ad^*(g)(\xi)$ to $\xi$.
\end{proof}

In this case, the symplectic leaves are coadjoint orbits, and the action has been studied quite extensively from the perspective of Fukaya categories, though possibly without mentioning the word ``groupoid.'' This action (or more precisely the action on a Hamiltonian $G$-manifold) is what appears in the work of Constantin Teleman \cite{teleman-gauge-icm} on 3D gauge theory and mirror symmetry, and it was exploited to great effect by Jonny Evans and Yank\i{} Lekili \cite{talking-bout-my-g-generation} to study the question of finding objects that generate the Fukaya category of a Hamiltonian $G$-manifold. The action correspondence is what those authors call the \emph{moment Lagrangian}.

The case of $G = T^*K$ has more structure than the other cases, because $G$ is a symplectic groupoid in two different ways, it is both the coadjoint action groupoid of $\mathfrak{k}$ and the cotangent bundle of $K$. Taking the transpose of the fiberwise addition correspondence gives a correspondence $\Delta : G \to G \times G$. Taking the monoidal structure considered above together with $\Delta$ makes $G= T^*K$ into a Hopf algebra in the $2$-category $\Sympcat$, and hence $\cF(G)$ is a Hopf algebra in $\Ainfcat$.

\subsection{The case of symplectic fibrations}
\label{sec:fibrations}

More general Poisson structures can sometimes be thought of as ``mixtures'' of the previous examples. As a specific example, let $(F,\omega_F)$ be a symplectic manifold, and let $\mu \in \Symp(F,\omega_F)$ be a symplectic automorphism. Consider the mapping torus
\begin{equation}
  M_\mu = \R \times F/\langle\tau\rangle
\end{equation}
Where $\tau$ is the diffeomorphism $\tau(t,x) = (t-1,\mu(x))$.
Let $\pi_F = \omega_F^{-1}$ denote the nondegenerate Poisson tensor on $F$. Taking the product with the zero Poisson structure on $\R$ gives a Poisson structure on $\R \times F$. This descends to a Poisson structure on $M_\mu$ be cause $\tau$ is a Poisson automorphism; call the result $\pi_\mu$. The symplectic leaves of this Poisson structure are precisely the fibers of $M \to \R/\Z$, which are symplectomorphic to $(F,\omega_F)$. 

We can get a symplectic integration of this Poisson structure by combining the previous examples. First, a symplectic integration of $(F,\pi_F)$ is $(F\times F,(-\omega_F) \times \omega_F)$, and a symplectic integration of $(\R,0)$ is $(T^*\R,\omega_{\text{can}})$. Their product $T^*\R \times \overline{F}\times F$ is an integration of $\R \times F$, and the automorphism $\tau$ lifts to an automorphism $\tilde{\tau}$ of the symplectic groupoid structure. Then we have that
\begin{equation}
  G_\mu = (T^*\R \times \overline{F} \times F)/\langle\tilde{\tau}\rangle
\end{equation}
is a symplectic integration\footnote{We get the same symplectic manifold by taking the symplectic mapping torus of $\mu\times \mu \in \Symp(\overline{F}\times F)$.} of $(M_\mu,\pi_\mu)$. Observe that there is a map $G_\mu \to T^*(\R/\Z)$.

The Fukaya category of things like $G_\mu$ contains some interesting objects. For instance, let $\phi \in \Symp(F,\omega_F)$ be another symplectic automorphism. Then the graph $\Gamma(\phi) \subset \overline{F} \times F$ is Lagrangian. One way to extend this to a Lagrangian in $G_\mu$ is to place $\Gamma(\phi)$ in fiber of $G_\mu \to T^*(\R/\Z)$ and cross with a cotangent fiber on $T^*(\R/\Z)$. Another way is to try ``cross with $\R/\Z$''. In order for it to close up, $\phi$ and $\mu$ must commute. It other words, it is the centralizer $C_{\Symp(F,\omega_F)}(\mu)$ that gives rise to Lagrangians in $G_\mu$. 

For each $p \in \R/\Z$, we have a symplectic leaf $F_p$ that is symplectomorphic to $(F,\omega)$, and $\cF(G_\mu)$ acts on $\cF(F_p)$ by endofunctors. Some objects of $\cF(G_\mu)$ act by a nonzero functor only on a single fiber, while others, such as the ones coming from $\phi \in C_{\Symp(F,\omega_F)}(\mu)$ act by a nonzero functor on all fibers.

Note also that there is an action of $G_\mu$ on the symplectic mapping torus $\R \times M_\mu$. \textcolor{highlight}{The Fukaya category of a symplectic mapping torus $\R \times M_{\mu}$ has been studied in the wrapped setting by Kartal \cite{kartal-dynamical, kartal-distinguishing}, who has shown that it coincides with a categorical mapping torus construction. We expect that, in general, there is a collection of $A_{\infty}$-functors $\cF(\R\times M_\mu) \to \cF(F_p)$ given by intersection with a fiber, and that the $A_{\infty}$-category $\cF(G_\mu)$ corresponds of $A_{\infty}$-endofunctors of $\cF(\R\times M_\mu)$ whose action respects this extra structure.}

It is perhaps interesting that even this case, which is rather simple from the point of view of Poisson geometry (the Poisson tensor has constant corank one, the symplectic foliation is a locally trivial fibration, and there is no variation in the symplectic form over the leaf space), already pushes us into some of the less-explored territory in the theory of Fukaya categories.

\subsection{Other module categories}
\label{sec:module-cats}

In addition to the module $A_{\infty}$-categories $\cF(F)$ for a symplectic leaf $F$, the $A_{\infty}$-category $\cF(G)$ may have other interesting module categories.

\textcolor{highlight}{\begin{itemize}
\item In the case of $G = T^*M$, we may consider twisted cotangent bundle $T^{*}M_{\beta}$. This is constructed by taking a closed $2$-form $\beta \in \Omega^{2}_{\mathrm{closed}}(M)$, and equipping $T^{*}M$ with the symplectic form $\omega_{\mathrm{can}} + \pi^{*}(\beta)$, where $\omega_{\mathrm{can}}$ is the canonical symplectic form on the cotangent bundle. The fiberwise addition correspondence is still Lagrangian when considered as a correspondence $T^{*}M \times T^{*}M_{\beta} \to T^{*}M_{\beta}$, so the groupoid $T^{*}M$ acts on the twisted cotangent bundle $T^*M_\beta$.
\item In the case of a torus bundle $G = T^*M/\Lambda_\Z^*$, where $M$ carries an integral affine structure, we may consider twisted versions as well. These are constructed by taking $T^{*}M_{\beta}$ as above, and dividing by the local $\Lambda_{\Z}^{*}$ actions, which preserve $\pi^{*}(\beta)$ since it is pulled back from the base. The resulting manifold $G_{\beta} = T^{*}M_{\beta}/\Lambda_{\Z}^{*}$ still admits a Lagrangian torus fibration but does not generally admit a Lagrangian section unless $\beta$ is exact. According to Abouzaid \cite{abouzaid-family,abouzaid-hms-wo-corr}, if $\pi_{2}(M) = 0$, then $\cF(G)$ is embeds as an $A_{\infty}$-category into the category of coherent analytic sheaves on a rigid analytic mirror variety, and $\cF(G_{\beta})$ embeds into a category of twisted sheaves, twisted by a Gerbe constructed from $\beta$. The action of $\cF(G)$ on $\cF(G_{\beta})$ thus corresponds to action of untwisted sheaves on twisted sheaves by tensor product.
\item In the case of a mapping torus $M_\mu$ with fiber $F$, the groupoid $G_\mu$ acts on the symplectic mapping torus $\R \times M_\mu$. The action of $\cF(G_{\mu})$ on $\cF(\R\times M_{\mu})$ is a parameterized version of the action of $\cF(\overline{F} \times F)$ on $\cF(F)$.
\item In the case of $G = T^*K$, $K$ a compact Lie group, one can consider a symplectic manifold $X$ with Hamiltonian $K$-action. Such actions have been studied from the Fukaya categorical point of view by Teleman \cite{teleman-gauge-icm} and Evans-Lekili \cite{talking-bout-my-g-generation}.
\end{itemize}}

\subsection{Singular torus fibrations}
\label{sec:singular-torus-fibrations}

We would also like to mention one other phenomenon which seems to fit at this point in our discussion. The torus bundles $G = T^*M/\Lambda_\Z^*$ we have considered so far are nonsingular in the sense that the map $\pi : T^*M/\Lambda_\Z^* \to M$ is a submersion. However, it is well-known that in almost all instances of mirror symmetry this condition cannot be satisfied; the Lagrangian torus fibers necessarily degenerate to singular fibers. Let $\pi : X \to M$ be such a Lagrangian torus fibration with singularities. Then there is no reason for $X$ to admit a groupoid structure. If we let $\Delta \subset M$ be the discriminant locus (critical values of $\pi$), then $\pi^{-1}(M\setminus \Delta)$ consists of nonsingular torus fibers, and it is a symplectic groupoid with $M \setminus \Delta$ as its manifold of objects.  On the other hand, sitting inside $X$ is the locus where $\pi$ is a submersion:
\begin{equation}
  X^{\gp} = \{x \mid D\pi(T_xX) = T_{\pi(x)}M\} \subset X.
\end{equation}
Clearly $X^{\gp}$ contains all nonsingular fibers, so $\pi^{-1}(M\setminus \Delta) \subset X^{\gp}$, but it also includes some parts of the singular fibers. For instance, when $\dim X = 4$, and the singular fiber is a nodal torus, $X^{\gp}$ contains the nodal torus minus the node.

We expect that the symplectic groupoid structure of $\pi^{-1}(M\setminus \Delta) \to M \setminus \Delta$ extends to a symplectic groupoid structure on $X^{\gp} \to M$. An important special case where this can be made clear is when the original fibration $X \to M$ comes from algebraic geometry. Namely, suppose that $\pi : X \to M$ is a proper algebraic completely integrable system. Then $X^{\gp}$ is the associated abelian group scheme over $M$. The notion of ``abelian group scheme over $M$'' is essentially the algebraic analogue of ``abelian groupoid with objects $M$''.

We remark that the lack of properness of $X^\gp \to M$ indicates that care is required in the definition of $\cF(X^\gp)$. Nevertheless, the groupoid $X^\gp$ acts on $X$, so it is likely that $\cF(X^\gp)$ admits a monoidal action on $\cF(X)$. It seems likely that $\cF(X^\gp)$ and $\cF(X)$ are equivalent categories, allowing us to transfer the monoidal structure to $\cF(X)$.

\begin{rmk}
  When $X$ has complex dimension two (real dimension 4), the theory of algebraic completely integrable systems is essentially Kodaira's theory of elliptic fibrations. See Table I at \cite[p.~604]{kodaira} for the group structure associated to each of the singular fibers in Kodaira's classification.
\end{rmk}

\subsection{Coisotropic submanifolds as objects of a Fukaya category}
\label{sec:coisotropic-fukaya}

So far, we have been using symplectic groupoids $(G, M)$ as a source of monoidal Fukaya categories, but we have not drawn a significant connection to the Poisson geometry of the induced Poisson structure on manifold of objects $M$. One way to draw such a connection is to use the idea that a coisotropic submanifold of $M$ can often be promoted to a Lagrangian submanifold of $G$. An \emph{Lagrangian integration} of a coisotropic submanifold $C \subseteq M$ is a Lagrangian \emph{subgroupoid} $L \subset G$ whose set of objects is $C$. 

For example, in the case $G = T^{*}M$, and $\pi = 0$, any submanifold $C \subseteq M$ is coisotropic. Then the conormal bundle $T^{*}_{C}M$ is a Lagrangian subgroupoid integrating $C$.

As in the case of integrations of a Poisson manifold, a smooth embedded Lagrangian subgroupoid integrating a coisotropic submanifold need not exist, but if it does, it gives us a way to treat a coisotropic submanifold as an object in a Fukaya $A_{\infty}$-category $\cF(G)$. In terms of the monoidal structure, the condition that $L$ be a subgroupoid translates into a relation like $L \otimes L \cong L$, meaning that $L$ is idempotent for the monoidal structure.

This idea leads to a proposal for a Floer cohomology group associated to a pair $C_{1}, C_{2}$ of coisotropic submanifolds in a Poisson manifold $M$. Namely, find a symplectic integration $G$ of $M$ and Lagrangian subgroupoids $L_{i}$ integrating $C_{i}$ ($i = 1,2$), and take the Floer cohomology $HF^{*}_{G}(L_{1},L_{2})$. Alternatively, using the duality, one could express this group as $HF^{*}(\cO,L_{2} \otimes \cD(L_{1}))$. It is in this form that M. Gualtieri \cite{gualtieri-talk} has proposed to use the monoidal structure of the Fukaya category of a symplectic groupoid as an avenue to define the Floer cohomology of more general branes in Poisson and generalized complex geometry.

\subsection{Speculations on family Floer theory and the ``orbit method''}
\label{sec:speculation}

While $\cF(G)$ acts on any symplectic leaf in $M$, it would be more satisfying to understand how the action on the various leaves fits together. At least in the cases of Sections \ref{sec:case-symp}, \ref{sec:cotangent}, and \ref{sec:fibrations}, it is possible to interpret the symplectic groupoid as endomorphisms of a sheaf of categories. Namely, consider the leaf space of the symplectic foliation on $(M,\pi)$. Taking the Fukaya categories of the leaves produces a sheaf of categories over the leaf space, and there is a functor from $\cF(G)$ to endomorphisms of this sheaf (at some level, a collection of endofunctors of the stalks). An interpretation as endomorphisms of an object gives another reason that the category is monoidal: endomorphisms can be composed.

Another perspective, suggested to me by David Ben-Zvi, is that the idea of attaching module categories to symplectic leaves is reminiscent of Kirillov's orbit method (corresponding to the case $T^*K$). It would be interesting to understand if and in what sense the symplectic leaves really give a complete set of irreducible module categories for the monoidal $A_{\infty}$-category $\cF(G)$.

One way to get a hint as to how this might work is to compare with family Floer theory \cite{family-fukaya,abouzaid-family,abouzaid-hms-wo-corr} and the Nadler-Zaslow correspondence \cite{nadler-zaslow,nadler} corresponding to the cases $G= T^*M/\Lambda_\Z^*$ and $G = T^*M$ respectively, which from this perspective are instances of the orbit method, the orbits being simply the points of $M$. In the case $G = T^*M$, objects of $\cF(T^*M)$ are transformed directly into sheaves on the orbit space $M$, whereas in the case $G = T^*M$, the functor lands in a category of sheaves on a rigid analytic space living over $M$. This indicates that the structure of the isotropy group of the orbit plays a significant role.

\section{Commutativity}
\label{sec:commutativity}

There is a connection between symplectic groupoid structures and commutativity. In the context of homological mirror symmetry, \textcolor{highlight}{symplectic homology $SH_*(G)$ is known to be the Hochschild homology of the wrapped Fukaya $A_{\infty}$-category $\cF(G)$ for wide classes of symplectic manifolds, including non-degenerate Liouville manifolds \cite{ganatra-thesis}.} In general, $\cF(G)$ is an $A_\infty$-category, or what is known after D.~Orlov as a ``noncommutative variety.'' The Hochschild homology of a noncommutative variety has no reason to carry a natural ring structure, while that of a commutative variety does have a ring structure, corresponding under Hochschild-Kostant-Rosenberg to the wedge product of differential forms.

In more elementary terms, we posit that there is a tension between the following two points:
\begin{itemize}
\item Given an object $L$ of the Fukaya category, the endomorphism algebra $\End^*(L) = HF^*(L,L)$ is an $A_\infty$-algebra. The reason why it carries is an $A_\infty$ structure is because operations on Floer cohomologies are governed by the operads formed by moduli spaces of marked Riemann surfaces: the operad tells you what operations and relations you are supposed to have, then you go construct them. Discs with marked boundary points and a distinguished output form an $A_\infty$ operad, and this operad has no commutativity relation.
\item In many cases where $\End^*(L)$ can be calculated, it turns out to be graded commutative. It is also true that in many cases $\End^*(L)$ turns out to be strictly noncommutative. What accounts for this?
\end{itemize}

In fact, the known theory of monoidal structures on categories provides neat, essentially formal arguments that show that certain endomorphism algebras are commutative. Combining this with our previous observations, we hold that the presence of a symplectic groupoid structure is the natural symplectic-geometric reason
 for commutative Floer cohomology rings.

\subsection{Endomorphisms of the unit object}
\label{sec:unit-object}

Let us suppose that $(G,\omega)$ is a symplectic groupoid and that there is a corresponding monoidal structure on $\cF(G)$ (for some version of the Fukaya category; the following argument is rather independent of the details).
\textcolor{highlight}{
\begin{thm}
  Let $\cO$ be the unit object in the monoidal $A_{\infty}$-category $\cF(G)$. Then the degree-zero Floer cohomology endomorphism algebra, $\End(\cO) = HF^0(\cO,\cO)$, is commutative.
\end{thm}
\begin{proof}
This is a version of the well-known Eckmann-Hilton argument, which we reproduce here following \cite[Prop.~2.2.10]{tensor-cat}. Part of the unitality of the monoidal structure is the existence of an isomorphism $u : \cO \otimes \cO \cong \cO$. Conjugating by $u$ induces an isomorphism
\begin{equation}
  \Psi:  \End(\cO\otimes \cO) \cong \End(\cO)
\end{equation}
Letting $1_\cO$ denote the identity morphism of $\cO$, we have the further unitality property $\Psi(1_\cO \otimes a) = a = \Psi(a \otimes 1_\cO)$. Now consider for $a, b \in \End(\cO)$: 
\begin{align*}
  ab &= \Psi(a\otimes 1_\cO) \Psi(1_\cO \otimes b) && \text{unitality property}\\
  &= \Psi((a \otimes 1_\cO)(1_\cO \otimes b)) && \text{$\Psi$ is a ring isomorphism}\\
  &= \Psi(a\otimes b) && \text{functoriality of $\otimes$}\\
  &= \Psi((1_\cO \otimes b)(a \otimes 1_\cO)) && \text{functoriality of $\otimes$}\\
  &= \Psi(1_\cO \otimes b)\Psi(a \otimes 1_\cO) && \text{$\Psi$ is a ring isomorphism}\\
  & = ba && \text{unitality property}
\end{align*}
\end{proof}}

There is a similar argument that applies to the full Floer \textcolor{highlight}{cohomology} endomorphism algebra $HF^*(\cO,\cO)$, but it contains signs. Indeed, $HF^i(\cO,\cO)$ is not part of $\End(\cO)$ but rather is $\Hom(\cO,\cO[i])$. We can treat $HF^*(\cO,\cO)$ as an instance of the homogeneous coordinate ring construction described in the next section.

\subsection{Homogeneous coordinate rings}
\label{sec:homog-coord-ring}

Now let $L$ be an invertible object in $\cF(G)$. In our case, this condition can be written as $L \otimes \cD(L) \cong \cO$. For one thing, this condition implies that tensoring with $L$ is an autoequivalence. Therefore $\End(L) \cong \End(\cO)$, and hence $\End(L)$ is also commutative.

Another thing we can do is form what we shall call the \emph{homogeneous coordinate ring}
\begin{equation}
  R = \bigoplus_{k \geq 0} \Hom(\cO,L^{\otimes k})
\end{equation}
There are several ways to make this space into a ring. One product is defined as
\begin{equation}
  \Hom(\cO, L^{\otimes m}) \otimes \Hom(\cO, L^{\otimes n}) \to \Hom(L^{\otimes n}, L^{\otimes (m+n)}) \otimes \Hom(\cO,L^{\otimes n}) \to \Hom(\cO,L^{\otimes (m+n)})
\end{equation}
where the first map is applying the equivalence $L^{\otimes n} \otimes -$ to the first factor, and the second map is composition of morphisms. Another product is defined by
\begin{equation}
  \Hom(\cO,L^{\otimes m}) \otimes \Hom(\cO,L^{\otimes n}) \to \Hom(\cO \otimes \cO, L^{\otimes m}\otimes L^{\otimes n}) \to \Hom(\cO,L^{\otimes (m+n)})
\end{equation}
Where the first map is tensor product of morphisms, and the second map uses isomorphisms $\cO \otimes \cO \cong \cO$ and $L^{\otimes n} \otimes L^{\otimes m} \cong L^{\otimes (m+n)}$. 

One can run arguments similar in spirit to the Eckmann-Hilton argument in this case, but it is not quite as simple as the case of the unit object. For one thing, it is critical that $\cF(G)$ be not just a monoidal category but a \emph{symmetric} monoidal category, meaning that for every pair of objects $X$ and $Y$ there is a chosen isomorphism $t_{X,Y} : X \otimes Y \to Y \otimes X$ satisfying several coherence properties. Symmetric monoidal structures correspond to \emph{abelian} symplectic groupoids, where the automorphism of $G\times G$ that swaps the factors commutes with the composition; this class includes the cotangent bundle $G = T^*M$ and torus bundles $G = T^*M/\Lambda_\Z^*$. In precisely this context, Dugger \cite{dugger} has developed a theory\footnote{I believe that similar ideas are well-known to experts in the theory of symmetric monoidal categories, but Dugger's article is the most complete treatment I could find.} for showing that rings such as our $R$ above are commutative up to sign. A translation of his result into our context is the following.

\begin{prop}[From Proposition 1.2 of \cite{dugger}]
  \label{prop:dugger}
  Let $\cC$ be a symmetric monoidal category with unit object $\cO$, and let $X_1,\dots,X_n$ be a collection of invertible objects in $\cC$. For $a \in \Z^n$, define
  \begin{equation}
    \underline{X}^a = X_1^{\otimes a_1}\otimes \cdots \otimes X_n^{\otimes a_n},
  \end{equation}
  and define $R_a = \Hom_{\cC}(\cO, \underline{X}^a)$. Then
  \begin{enumerate}
  \item $R_* = \bigoplus_{a \in \Z^n} R_a$ is a $\Z^n$-graded ring,
  \item There exist elements $\tau_1,\dots,\tau_n$ in $\End(\cO)$ satisfying $\tau_i^2 = 1$ such that for $f \in R_a$ and $g \in R_b$,
    \begin{equation}
      fg = \left[\tau_1^{a_1b_1}\cdots \tau_n^{a_nb_n}\right]gf.
    \end{equation}
  \end{enumerate}
\end{prop}

Thus commutativity holds up to the action of certain elements $\tau_i \in \End(\cO)$ attached to the invertible objects $X_i$. These elements are organized as a homomorphism $\tau : \Pic(\cC) \to \Aut(\cO)_2$, where $\Pic(\cC)$ is the group of isomorphism classes of invertible objects in $\cC$, and $\Aut(\cO)_2$ is the $2$-torsion subgroup of $\Aut(\cO)$. It is useful to note that this means that $\tau(X^{\otimes 2})= 1$ for any invertible object $X$. Thus, if we replace every object $X_i$ in the proposition above by $X_i^{\otimes 2}$, we find that commutativity holds in the strict sense.

It is interesting to remark that, in several places in the Floer theory literature where commutativity of a ring such as $R = \bigoplus_{k \geq 0} HF^0(\cO,L^{\otimes k})$ has been observed, one finds for geometric reasons that there is a basis of $R$ (consisting of intersection points of $\cO$ and $L^{\otimes k}$) such that, if $f$ and $g$ are basis elements, then $fg = \pm gf$, where the $\pm$ sign seems as if it could depend in an arbitrary way on the chosen basis elements $f$ and $g$, not just on their homogeneous degrees. One must then check these signs carefully. What Proposition \ref{prop:dugger} shows is that, when the Fukaya category carries a symmetric monoidal structure, these signs can only depend on the choice of the object $L$. Moreover, by replacing $L$ by $L^{\otimes 2}$, one can guarantee that the signs are all trivial. In a sense, this simply pushes the problem into checking that the coherences required of a symmetric monoidal category do indeed hold, but there is a good reason for that to be the case when $G$ is an abelian symplectic groupoid.

\subsection{Higher endomorphisms of the unit object}
\label{sec:higer-unit-object}

As mentioned above, an instance of the the homogeneous coordinate ring construction is the sum of morphisms from $\cO$ to all shifts of $\cO$, which we denote $\End^*(\cO)$:
\begin{equation}
  \End^*(\cO) = \bigoplus_{i\in \Z} \Hom(\cO,\cO[i]) = \bigoplus_{i\in \Z} HF^i(\cO,\cO) = HF^*(\cO,\cO)
\end{equation}
This is because tensoring with $\cO[1]$ is isomorphic to the shift functor, and so $\cO[i]\otimes \cO[j] \cong \cO[i+j]$. Just as $\End(\cO)$ is commutative even when the monoidal structure is not symmetric, one can show that $\End^*(\cO)$ is graded commutative in that case. This relies on the fact that left and right tensor product with $\cO[1]$ are both isomorphic to the shift functor, meaning that $\cO[1]$ can be made into an object of the Drinfeld center of the monoidal category.

We can also remark on the chain-level structure one expects on $CF^{*}(\cO,\cO)$, the cochain complex that computes $HF^{*}(\cO,\cO)$. Generally speaking, $CF^{*}(\cO,\cO)$ has the structure of an $A_{\infty}$-algebra, and we have seen that its cohomology is graded commutative. Furthermore, it should carry a bracket of degree $-1$ making it into an $E_{2}$-algebra. The construction of the bracket comes from attempting to lift the proof of commutativity (see section \ref{sec:unit-object} above) to the chain level. First, left the isomorphism $\Psi$ to a quasi-isomorphism $\psi : CF^{*}(\cO\otimes \cO,\cO\otimes \cO) \to CF^{*}(\cO,\cO)$. Define the four bilinear operations:
\begin{equation}
  m_{1}(a,b) = ab,\ m_{2}(a,b) = \psi(a\otimes b),\ m_{3}(a,b) = ba,\ m_{4}(a,b) = \psi(b\otimes a).
\end{equation}
The argument from section \ref{sec:unit-object} shows that all four of these operations induce the same operation on cohomology. By analyzing the argument further, one sees that it leads to a sequence of chain homotopies
\begin{equation}
  m_{1} \simeq m_{2} \simeq m_{3} \simeq m_{4} \simeq m_{1},
\end{equation}
meaning that there are degree $-1$ operators $P_{12},P_{23},P_{24},P_{41}$ such that $dP_{ij}+P_{ij}d = m_{j}-m_{i}$. The sum $P = P_{12}+P_{23}+P_{34}+P_{41}$ then satisfies $dP+Pd = 0$, so it is a chain map of degree $-1$. This $P$ is the bracket in the $E_{2}$-algebra structure.

\textcolor{highlight}{The full chain-level $E_{2}$-algebra structure contains many more operations than just $P$. Fortunately we have the following theorem stating that they are all determined up to coherent homotopy by the presentation of $\cO$ as the unit object in a monoidal $A_{\infty}$-category.
  \begin{thm}[Deligne conjecture, cf.\ \cite{benzvi-gunningham} Corollary 4.2]
  \label{thm:e2}
  Let $\cA$ be a monoidal $A_{\infty}$-category (= $\otimes$-monoid object in $\Ainfcat$), let $\cO$ be the unit object of $\cA$, and let $\cEnd^{*}(\cO)$ denote the endomorphism $A_{\infty}$-algebra of $\cO$. Then $\cEnd^{*}(\cO)$ carries the structure of an $E_{2}$-algebra.
\end{thm}
\begin{proof}
  The monoidal structure on $\cA$ and the isomorphism $u : \cO \otimes \cO \to \cO$ make $\cEnd^{*}(\cO)$ into a $\otimes$-monoid object in the category of $A_{\infty}$-algebras (a monoidal $A_{\infty}$-category with essentially one object). In other words, $\cEnd^{*}(\cO)$ is an $E_{1}$-algebra in the category of $E_{1}$-algebras. By Lurie's version of the Dunn additivity theorem\footnote{I learned this theorem from David Ben-Zvi.} \cite[Theorem 5.1.2.2]{HA}, it therefore carries the structure of an $E_{2}$-algebra.
\end{proof}
\begin{rmk}
  The original Deligne conjecture corresponds to the case where $A$ is an associative algebra, $\cA = \bimod{A}{A}$ and $\cO = A$ is the diagonal bimodule. There are many proofs of this conjecture available. The generalization of the Deligne conjecture to monoidal unit objects has also been studied by several authors, including Kock-To\"{e}n \cite{kock-toen} and Shoikhet \cite{shoikhet}.
\end{rmk}
\begin{rmk}
  Bottman has shown that there is a connection between $2$-associahedra and the 2-dimensional Fulton-MacPherson operad \cite{bottman-FM}. This suggests a way to obtain an $E_{2}$-algebra from an $(A_{\infty},2)$-algebra \cite[Section 1.5]{bottman-FM}.
\end{rmk}}

It is also interesting to note that, when the Poisson manifold $(M,\pi)$ underlying the symplectic groupoid $G$ is compact, the $A_{\infty}$-algebra $CF^{*}(\cO,\cO)$ is a deformation of the cochain algebra $C^{*}(M)$. Thus we find that the Floer theory of a symplectic groupoid integrating $(M,\pi)$ furnishes a deformation of $C^{*}(M)$ as an $E_{2}$-algebra. 

\subsection{Cases of commutativity}
\label{sec:commutativity-cases}

Now we consider several cases where commutativity can be explained by the preceding arguments.

\subsubsection{$G = T^*M$}
\label{sec:comm-cotangent}

The unit object the zero section $\cO = T^*_MM$. We have
\begin{equation}
  \End^*(\cO) \cong H^*(M)
\end{equation}
with the cup product, which is indeed graded commutative. \textcolor{highlight}{This calculation is due to Fukaya-Oh \cite{fukaya-oh}.}

\subsubsection{$M = \mathfrak{k}^*$, $G= T^*K$}
\label{sec:comm-lie-alg}

The unit object is the cotangent fiber at the identity $\cO = T^*_eK$. As an object of the wrapped Fukaya category
\begin{equation}
  \End^*(T^*_eK) \cong H_*(\Omega K)
\end{equation}
where the product is the Pontryagin product on the based loop space. \textcolor{highlight}{This product is commutative because $\Omega K \simeq \Omega^{2} BK$ has the homotopy type of the double loop space of the classifying space $BK$.} This fact, which generalizes the fact that the fundamental group of a topological group is abelian, is in some sense the source of the Eckmann-Hilton argument historically, \textcolor{highlight}{and the problem of recognizing double loop spaces led to the introduction of the $E_{2}$-operad \cite{may-loop}.}

\subsubsection{$G = \overline{M} \times M$}
\label{sec:comm-pairs}

The unit object is the diagonal $\Delta \subset \overline{M} \times M$. When $M$ is compact we have
\begin{equation}
  \End^*(\Delta) \cong QH^*(M)
\end{equation}
The right-hand side is commutative for operadic reasons (it is an algebra over the hypercommutative operad $H_*(\overline{\mathcal{M}}_{0,n+1})$ \textcolor{highlight}{\cite[Theorem III.1.5, Proposition IV.1.8.1]{manin}}), but the left-hand side, \textcolor{highlight}{seen merely as a Floer cohomology ring, is not manifestly commutative}. Now we see that it is the groupoid structure on $\overline{M} \times M$ that is responsible for the left-hand side being commutative.

Similarly, when $M$ is Weinstein, and we consider the wrapped Fukaya category, \textcolor{highlight}{the theory developed by Ganatra \cite{ganatra-thesis} shows that we have $\End^{*}(\Delta) \cong SH^*(M)$, and $SH^{*}(M)$ is known to be commutative because it carries the structure of a Batalin-Vilkovisky algebra.}

If we take a symplectic automorphism $\phi : M \to M$, there is a corresponding invertible object in $\cF(G)$, namely the graph $\Gamma(\phi)$ of $\phi$. Its $k$-th monoidal power is $\Gamma(\phi^k)$, and $HF(\Delta,\Gamma(\phi^k))$ is $HF(\phi^k)$, the fixed point Floer cohomology. It is possible \textcolor{highlight}{\cite{seidel-phi-squared}} to define a variety of bilinear operations on the space
\begin{equation}
  R = \bigoplus_{k \geq 0} HF(\phi^k),
\end{equation}
but in this case, the conditions of Proposition \ref{prop:dugger} are not satisfied, and these products are not necessarily commutative in any sense: for one thing, the monoidal structure is not symmetric. Even restricting to the monoidal subcategory generated by $\Gamma(\phi)$, one finds that the desired coherences cannot hold. This is related to the fact that $HF(\phi^k)$ carries an action of $\Z/k$, where the generator sends $x \mapsto \phi(x)$ for $x$ a fixed point of \textcolor{highlight}{$\phi^{k}$}. This action appears in the monoidal theory as the action of the conjugation functor $C \mapsto \Gamma(\phi) \otimes C \otimes \Gamma(\phi^{-1})$ on the space $HF(\Delta, \Gamma(\phi^k))$. The nontriviality of the conjugation functor (even on the monoidal subcategory generated by $\Gamma(\phi)$) measures the difference between left and right tensoring with $\Gamma(\phi)$, and the comparison of these two actions is crucial for the proof of Proposition \ref{prop:dugger}.

\subsubsection{Torus fibration}
\label{sec:comm-hms}

In the case $G = T^*M/\Lambda_\Z^*$, the unit object is the zero section. When $M$ is noncompact, one can obtain interesting rings as $\End(\cO)$. For instance, when $M = \R^n$, \textcolor{highlight}{we have $G \cong T^{*}T^{n}$ and $\cO \cong T^{*}_{p}T^{n}$. If we take the wrapped Fukaya category then $\End(\cO) \cong H_{0}(\Omega T^{n})$ by Abouzaid \cite{abouzaid-based-loops}. This is the ring of Laurent polynomials in $n$ variables, and it is commutative.}

The invertible objects are the ones supported on Lagrangian sections of the fibration $G \to M$. Taking such a section $L$, we can form the homogeneous coordinate ring
\begin{equation}
  R = \bigoplus_{k \geq 0} HF^0(\cO,L^{\otimes k})
\end{equation}
which is then commutative. As the name suggests, when the section $L$ corresponds to an ample line bundle, this ring is actually supposed to be the homogeneous coordinate ring of the mirror variety $X$ with respect to the projective embedding determined by $L$.

We expect that this analysis extends to the case of singular torus fibrations $X \to M$ by restricting to the groupoid part $X^{\gp}$.

\begin{example}
  As a specific application, consider the following problem posed to me by Paolo Ghiggini.  Let $T \subset S^3$ be the standard Legendrian embedding of the trefoil knot in the standard contact $S^3$. There is a differential graded algebra associated to any Legendrian link $\Lambda \subset S^3$, the Chekanov-Eliashberg algebra $LCA^*(\Lambda)$. In general, the underlying associative algebra of $LCA^*(\Lambda)$ is a free associative algebra generated by the Reeb chords of $\Lambda$, so it is non-commutative. On the other hand, for the case of $\Lambda = T$, the trefoil, the cohomology algebra $H^*(LCA^*(T),\partial)$ turns out to be commutative; to the author's knowledge this was first observed by Y.~Lekili and appears in the work of Ekholm-Lekili \cite[Section 6.1.5]{ekholm-lekili}.

  We can now explain this as follows. First, it is understood that the complex $(LCA^*(T),\partial)$ calculates the wrapped Floer cohomology $HW^*(L,L)$ of the cocore Lagrangian $L$ in the manifold $X$ obtained by attaching a Weinstein $2$-handle to the symplectic $B^4$ along $T \subset S^3 = \partial B^4$. Second, there is another presentation of the Weinstein manifold $X$, which is as the manifold obtained from the cotangent disc bundle $D^*T^2$ of the 2-torus by handle attachment along the conormal lifts of the $a$ and $b$ curves on $T^2$. Third, there is yet another presentation of $X$ which is as a singular Lagrangian torus fibration over $D^2$ with two singular fibers; each singular fiber is a nodal torus, and the vanishing cycles are the $a$ and $b$ curves from the previous description. In this last description, the cocore $L$ becomes a section of the Lagrangian torus fibration. If we set up the groupoid structure on $X^{\gp}$ so that this $L$ becomes the zero-section, then $L$ becomes the unit object in the monoidal structure on $\cF(X)$, then $H^*(LCA^*(T),\partial) = HW^*(L,L) = \End^*(L)$ is the endomorphism algebra of the unit object, and commutativity follows from the Eckmann-Hilton argument.

  Another remark is that, by regarding $L$ has Lagrangian section of the torus fibration and using a wrapping Hamiltonian pulled back from the base of the torus fibration, it is possible to construct a complex computing $HW^{*}(L,L)$ that is concentrated in degree zero, implying that $H^{*}(LCA^{*}(T),\partial)$ is concentrated in degree zero as well. This is a non-trivial observation since $LCA^{*}(T)$ has many elements of non-zero degree.
\end{example}

\subsubsection{Symplectic fibration} Let $G = G_\mu$ be the symplectic groupoid associated to $\mu \in \Symp(F,\omega_F)$, \textcolor{highlight}{where $F$ is compact.} Recall that $G_\mu$ is an $\overline{F}\times F$ fibration over $T^*S^1$. The unit object $\cO$ is the fibration over $S^1 \subset T^*S^1$ whose fiber is the diagonal $\Delta_F \subset \overline{F}\times F$.

\textcolor{highlight}{The classical cohomology of $\cO$ may be computed using the Serre spectral sequence for the fibration $F \to \cO \to S^{1}$,
  \begin{equation}
    \label{eq:serre}
  H^{*}(\cO) = \ker(\mu^{*}-I) \oplus \coker(\mu^{*}- I)[-1],
\end{equation}
where $\mu^{*} : H^{*}(F) \to H^{*}(F)$ is the action of monodromy of the fibration on the cohomology of the fiber. Since $\mu \in \Symp(F,\omega_{F})$, $\mu^{*}$ is an automorphism of the quantum cohomology algebra $QH^{*}(F)$. Then $QH^{*}(F)^{\mu} := \ker(\mu^{*}-I)$ is a subalgebra, and $QH^{*}(F)_{\mu} := \coker(\mu^{*}-I)$ is a module for $QH^{*}(F)^{\mu}$, so it is natural to expect that \eqref{eq:serre} holds at the quantum level, giving
\begin{equation}
  HF^{*}(\cO,\cO) = QH^{*}(F)^{\mu} \oplus QH^{*}(F)_{\mu}[-1].
\end{equation}
}
The point we wish to make is that, even without calculating what this ring is, we know \emph{a priori} that it is commutative, simply because it is the unit object in a symplectic groupoid.

\appendix
\section{\textcolor{highlight}{Monoidal categories and monoids in higher category theory}}

In this section we will review several notions of monoidal category and monoid object, both in ordinary category theory, and in the theory of higher categories. Most of this material may be found in standard references on higher category theory, and we may not always cite the original references. Our presentation follows Etingof et al. \cite{tensor-cat}, Lurie \cite{DAG-II,goodwillie-I,HA}, and Gaitsgory-Rozenblyum \cite{GR-I}. The websites nLab (\texttt{ncatlab.org}) and Kerodon (\texttt{kerodon.net}) were also consulted.

The symplectically inclined reader will notice the absence of operads in our discussion. Instead we use the simplex category and its variants to formulate associativity and unitality properties. This is the primary method used in the literature on higher categories.

The goal of this Appendix is to arrive at the definition of monoid object that we use in the paper: see Definition \ref{defn:noncart-oo-monoid} and its reinterpretation in Proposition \ref{prop:noncart-oo-monoid}. This definition is a higher-categorical variation of a notion of homotopy monoid object introduced by Leinster \cite{leinster}.

\subsection{The simplex category and simplicial objects}
\label{sec:simplex-cat}

The \emph{simplex category} $\bDelta$ is the category whose objects are the finite nonempty ordinals $[n] = \{0 < 1 < \cdots < n\}$ and whose morphisms are monotonic functions. \textcolor{secondrevision}{In some situations, we shall use the alternative notation $\mathbf{n} = [n-1]$. The \emph{augmented simplex category} $\bDelta_{a}$ also includes the empty set $\mathbf{0} = [-1] = \emptyset$. (In von Neumann's notation for the ordinals, the ordinal $n$ is the object that we write as $\mathbf{n} = [n-1]$. The notation $\mathbf{n}$ emphasizes the cardinality of the set, while the notation $[n]$ emphasizes the dimension of the simplex with this vertex set.)}

The category $\bDelta$ may be presented in terms of generators and relations. Let $\delta^{i} : [n-1] \to [n]$ be the unique monotonic function whose image does not contain $i$, and let $\sigma^{i} : [n+1] \to [n]$ be the unique monotonic function such that $i \in [n]$ has two distinct preimages; these functions are defined for each pair $(n,i)$ such that $0 \leq i \leq n$. The maps $\delta^{i}$ and $\sigma^{i}$ generate the category $\bDelta$, subject to the following relations \cite[\S 1.1.1, taking the opposite]{kerodon}:
\begin{enumerate}
\item For $n \geq 2$ and $0 \leq i < j \leq n$, the relation $\delta^{j} \circ \delta^{i} = \delta^{i} \circ \delta^{j-1}$,
\item For $0 \leq i \leq j \leq n$, the relation $\sigma^{j}\circ \sigma^{i} = \sigma^{i} \circ \sigma^{j+1}$,
\item For $0 \leq i,j \leq n$, the relation
  \begin{equation*}
    \sigma^{j}\circ \delta^{j} = 
    \begin{cases}
      \delta^{i}\circ \sigma^{j-1}, & \text{if $i < j$},\\
      \id, & \text{if $i = j$ or $i = j+1$},\\
      \delta^{i-1}\circ \sigma^{j}, & \text{if $i > j+1$}.
    \end{cases}
  \end{equation*}
\end{enumerate}

We shall also consider the opposite category $\bDelta^{\op}$. We use the notations $d_{i}: [n] \to [n-1]$ and $s_{i} : [n] \to [n+1]$, for $\delta^{i}$ and $\sigma^{i}$ regarded as morphisms in the opposite category. The category $\bDelta^{\op}$ is therefore generated by $d_{i}$ and $s_{i}$ subject to the opposites of the relations written above. The map $d_{i}$ is called a \emph{face map}, $s_{i}$ is called a \emph{degeneracy map}, $\delta^{i}$ is called a \emph{coface map}, and $\sigma^{i}$ is called a \emph{codegeneracy map}.

For any pair $\{i,i+1\} \subset [n]$ of consecutive elements, there is a unique monotonic function $\epsilon_{i,i+1} : [1] \to [n]$ whose image is $\{i,i+1\}$. The opposite $e_{i,i+1} : [n] \to [1]$ is a morphism in $\bDelta^{\op}$ that we call a \emph{principal edge (of the $n$-simplex)}.

A useful variation is the category of \emph{intervals}. This is the subcategory $\bDeltaint \subset \bDelta$ with the same objects, but whose morphisms are required to take the minimum element to the minimum element, and the maximum element to the maximum element. A presentation for $\bDeltaint$ is obtained from the presentation of $\bDelta$ by throwing away the first and last coface maps $\delta^{0},\delta^{n} : [n-1] \to [n]$. There is a duality between face and degeneracy maps that is evident from the relations, and which is expressed by the following isomorphism of categories.\footnote{I learned this proposition from Ezra Getzler.}
\begin{prop}[\cite{nlab} article ``simplex category'']
  \label{prop:simplex-dual}
There is a functor $F : \bDelta_{a} \to \bDeltaint^{\op}$ defined on objects and generating morphisms by
\begin{equation*}
  F([n]) = [n+1], \quad F(\delta^{i}) = s_{i},\quad F(\sigma^{i}) = d_{i+1},
\end{equation*}
and $F$ is an isomorphism of categories $\bDelta_{a } \cong \bDeltaint^{\op}$ between the augmented simplex category and the opposite category of intervals.
\end{prop}

Let $\calC$ be an ordinary category. A \emph{simplicial object} in $\calC$ is a functor $\bDelta^{\op} \to \calC$, a \emph{cosimplicial object} is a functor $\bDelta \to \calC$, an \emph{augmented simplicial object} is a functor $\bDelta_{a}^{\op}\to \calC$, and an \emph{augmented cosimplicial object} is a functor $\bDelta_{a} \to \calC$. According to Proposition \ref{prop:simplex-dual}, a functor $\bDeltaint^{\op} \to \calC$ is essentially the same thing as an augmented cosimplicial object, and a functor $\bDeltaint \to \calC$ is essentially an augmented simplicial object. In particular, any simplicial object gives rise to an augmented cosimplicial object by restricting along the embedding $\bDelta_{a} \cong \bDeltaint^{\op} \to \bDelta^{\op}$.

\subsection{Nerves of monoids and categories}
\label{sec:ordinary-nerves}

A fundamental observation in algebra and its higher-categorical generalizations is that the opposite simplex category $\bDelta^{\op}$ may be used to formulate associativity and unitality axioms. Because the concepts ``monoid'' and ``category'' are defined by such axioms, they may be recast in the form of simplicial objects satisfying certain extra conditions.

\begin{rmk}
  The use of $\bDelta^{\op}$ to formulate associativity axioms is different from the approach using the Stasheff associahedra that is common in work on Floer theory. The latter approach has the benefit that it can be generalized to other operads, so that one can formulate other kinds of algebraic axioms, such as commutative laws and Jacobi identities. On the other hand, the approach using $\bDelta^{\op}$ involves much simpler combinatorics and is fundamental to most of the literature on higher categories.
\end{rmk}

We begin by recalling that a monoid (in the category of sets) consists of a set $M$, a map $\circ : M \times M \to M$, and an element $e \in M$ such that $a\circ (b\circ c) = (a\circ b)\circ c$ and $e\circ a = a = a\circ e$ for all $a,b,c \in M$. From these data we shall construct a simplicial set $BM : \bDelta^{\op} \to \Set$, known as \emph{classifying space of $M$} \cite[Example 1.2.11]{DAG-II}.

Concretely, $BM([n]) = M^{n}$ is the $n$-fold Cartesian product of the set $M$, the degeneracy maps $s_{i} : M^{n-1} \to M^{n}$ insert $e$ in the $i$-th factor, the first face map $d_{0} : M^{n} \to M^{n-1}$ forgets the first factor, the last face map $d_{n} : M^{n} \to M^{n-1}$ forgets the last factor, and the other face maps $d_{i} : M^{n} \to M^{n-1}$ multiply two consecutive elements using $\circ$. The fact that these maps satisfy the relations in $\bDelta^{\op}$ encodes the associativity and unital axioms of $M$. For instance the relation
\begin{equation*}
  d_{1} \circ d_{1} = d_{1}\circ d_{2} : M^{3} \to M
\end{equation*}
is the associative law.

Next, observe that the principal edges define maps $e_{i,i+1}: BM([n]) \to BM([1])$, and taking the product of these maps over all $i$ gives a map
\begin{equation*}
  \prod_{i=0}^{n-1} e_{i,i+1} : BM([n]) \to BM([1])^{n}
\end{equation*}
and this map is nothing but the identity map $M^{n} \to M^{n}$.

\begin{prop}[\cite{DAG-II} Example 1.2.11]
  \label{prop:bmonoid}
  Let $X : \bDelta^{\op} \to \Set$ be a simplicial set, and suppose that, for every $n \geq 0$, the map induced by the principal edges
  \begin{equation*}
    \prod_{i=0}^{n-1} e_{i,i+1} : X([n]) \to X([1])^{n}
  \end{equation*}
  is bijective. Then there is a monoid $(M,\circ,e)$, unique up to isomorphism, such that $X$ is isomorphic to $BM$.
\end{prop}

The key is to see how the multiplication and the unit are recovered from the simplicial structure. The hypothesis of the proposition implies that $X([0])$ is a one-point set, and the unit is image of this point under the degeneracy map $s_{0} : X([0]) \to X([1])$. There is also a bijection $X([2]) \to X([1])^{2}$; taking the inverse of this bijection followed by $d_{1} : X([2]) \to X([1])$ recovers the multiplication.

Now we pass from monoids to (small) categories. Let us keep in mind that a category is a ``monoid with many objects.'' Let $\calC$ be a small category with object set $C$. We introduce another variant of $\bDelta$, the \emph{$C$-colored simplex category} $\bDelta_{C}$. An object of $\bDelta_{C}$ consists of a pair $([n],f : [n] \to C)$ whose first coordinate is an object of $\bDelta$ and whose second coordinate is a coloring of the elements of $[n]$ by elements of $C$. The morphisms in $\bDelta_{C}$ are monotonic functions that are compatible with the colorings in the sense that $p : [n] \to [m]$ gives rise to a morphism from $([n], f \circ p: [n] \to C)$ to $([m],f : [m] \to C)$. We write $[c_{0},c_{1},\dots,c_{n}]$ for the object $([n],f)$ where $f(i) = c_{i}$. The idea is that $\bDelta_{C}^{\op}$ controls the associativity and unitality of any category whose object set is identified with $C$. 

Recall that $\calC$ is a category with object set $C$. Define the \emph{$C$-colored nerve} of $\calC$ to be the functor $N_{C}(\calC) : \bDelta_{C}^{\op}\to \Set$ such that
\begin{itemize}
\item $N_C(\calC)([c])$ is a final object (a one-point set that we could take to be $\{\emptyset\}$ for specificity).
\item $N_{C}(\calC)([c_{0},c_{1},\dots,c_{n}])$ is the set of sequences of composable morphisms $c_{0}\to c_{1}\to \cdots \to c_{n}$; in particular $N_{C}(\calC)([c_{0},c_{1}]) = \Hom_{\calC}(c_{0},c_{1})$.
\item The degeneracy maps $s_{i}: [c_{0},\dots, c_{i}, \dots, c_{n}] \to [c_{0},\dots, c_{i},c_{i}, \dots, c_{n}]$ have the effect of inserting the identity morphism in the $i$-th spot.
\item For $0 < i < n$, The face maps $d_{i} : [c_{0},\dots,c_{i-1},c_{i},c_{i+1},\dots,c_{n}] \to [c_{0},\dots,c_{i-1},c_{i+1},\dots,c_{n}]$ are given by composition of morphisms in $\calC$ at the object $c_{i}$. The face maps $d_{0}$ and $d_{n}$ are projections.
\end{itemize}

There is an analogue of Proposition \ref{prop:bmonoid} in this context:
\begin{prop}[cf. \cite{goodwillie-I} Section 2]
  \label{prop:cnerve}
  Let $C$ be a set and let $X : \bDelta_{C}^{\op} \to \Set$ be a functor such that
  \begin{enumerate}
  \item for each $c$, $X([c])$ is a one-point set,
  \item for each $[c_{0},\dots,c_{n}]$, the map induced by the principal edges in $\bDelta_{C}^{\op}$
  \begin{equation*}
    X([c_{0},c_{2},\dots,c_{n}]) \to X([c_{0},c_{1}])\times X([c_{1},c_{2}]) \times \cdots \times X([c_{n-1},c_{n}])
  \end{equation*}
  is bijective.
  \end{enumerate}
  Then there is a category $\calC$ with object set $C$, unique up to isomorphism, such that $X$ is isomorphic to $N_{C}(\calC)$.
\end{prop}

We must also mention the \emph{(ordinary) nerve} $N(\calC)$ of $\calC$, which is given by the same construction but forgetting the colorings. Then $N(\calC)([0]) = C$ is the set of objects, and $N(\calC)([n])$ is the set of all composable sequences of morphisms between any pair of objects. The purpose of using the $C$-colored version is that it lends itself slightly better to certain higher-categorical generalizations (the notion of Segal categories to be discussed below).

An obvious but important property of the nerve is that it may be used to turn any ordinary category into an $\infty$-category; for instance $N(\bDelta^{\op})$ is an $\infty$-category that plays the role of $\bDelta^{\op}$ in $\infty$-category theory. Thus a simplicial object in an $\infty$-category $\calC$ is a map $X : N(\bDelta^{\op}) \to \calC$.

\subsection{Ordinary monoidal categories and monoid objects in them}
\label{sec:monoids}

Proposition \ref{prop:bmonoid} gives us a way to define a monoid object in any category.
\begin{defn}
  \label{defn:cartmonoid}
  Let $\calC$ be a category that has finite products. A \emph{monoid object} of $\calC$ is a simplicial object $X : \bDelta^{\op} \to \calC$ such that, for every $n \geq 0$, the map induced by the principal edges
  \begin{equation*}
    \prod_{i=0}^{n-1} e_{i,i+1} : X([n]) \to X([1])^{n}
  \end{equation*}
  is an isomorphism in $\calC$ between the object $X([n])$ and the $n$-fold Cartesian product of the object $X([1])$.
\end{defn}

Given a monoid object in the sense just defined, we get an associative multiplication $X([1])^{2} \to X([1])$ by taking the inverse of the ismorphism $X([2])\to X([1])^{2}$ followed by $d_{1} : X([2]) \to X([1])$. We get a unit by taking the degeneracy $s_{0} : X([0]) \to X([1])$, noting that $X([0])$ is necessarily a terminal object according to the definition.

We now observe that the definition of monoid object just given is restrictive in a specific sense, namely that it makes reference to the categorical Cartesian product. The Cartesian product is an instance of a monoidal structure on $\calC$, and we may wish to consider other monoidal structures. For instance, consider the category $\VectK$ of ordinary (not DG) vector spaces over a field $\K$. We would like to say that an associative $\K$-algebra $A$ is the same thing as a monoid object in $\VectK$. But the multiplication map $A \times A \to A$ is \emph{not linear but bilinear}, and so it is not a morphism in $\VectK$. What we must do instead is consider the multiplication as a map $A\otimes_{\K} A \to A$, and this is a morphism in $\VectK$. Thus we must consider $\otimes_{\K}$ as the ambient monoidal structure on $\VectK$ in order for this definition to make sense. 

We now recall the definition of a monoidal structure on an ordinary category.
\begin{defn}[\cite{tensor-cat} Definition 2.1.1]
  \label{defn:ordinarymonoidalcat}
  A \emph{monoidal category} is a tuple $(\calC, \otimes, a, 1_{\calC}, \iota)$, where $\calC$ is a category, $\otimes : \calC \times \calC \to \calC$ is a bifunctor, $1_{\calC}$ is an object of $\calC$, $a$ is a natural isomorphism of functors $\calC^{3} \to \calC$,
  \begin{equation*}
    a_{X,Y,Z} : (X \otimes Y) \otimes Z \to X \otimes (Y \otimes Z),
  \end{equation*}
  called the \emph{associativity constraint},
  and $\iota : 1_{\calC}\otimes 1_{\calC} \to 1_{\calC}$ is an isomorphism. These data are required to satisfy the unit axiom that $X \mapsto 1_{\calC} \otimes X$ and $X \mapsto X \otimes 1_{\calC}$ are autoequivalences of $\calC$, and also Mac Lane's pentagon axiom.
\end{defn}

The associativity constraint $a_{X,Y,Z}$ witnesses that the monoidal operation is associative up to isomorphism, and \emph{Mac Lane's pentagon axiom} is the statement that this associativity is coherent when considering four objects. The axiom may be written as an equation as
\begin{equation}
  \label{eq:maclane}
  a_{W,X,Y\otimes Z}\circ a_{W\otimes X,Y,Z} = (\id_{W}\otimes a_{X,Y,Z})\otimes a_{W,X\otimes Y,Z}\circ (a_{W,X,Y}\otimes \id_{Z})
\end{equation}
where these morphisms relate the five ways of inserting parentheses into the expression $W\otimes X \otimes Y \otimes Z$. 

\begin{rmk}
  The reader who is familiar with $A_{\infty}$-algebras will note that the definition of a monoidal category has some of the same flavor. The associativity constraint $a$ corresponds to the operation $m_{3}$, and the pentagon appearing in the pentagon axiom is an instance of a Stasheff associahedron.
\end{rmk}

We shall also need the concept of a monoidal functor.
\begin{defn}[\cite{tensor-cat} Definition 2.4.1]
  Let $(\calC_1,\otimes_{1})$ and $(\calC_{2},\otimes_{2})$ be monoidal categories. A \emph{monoidal functor} consists of a functor $F : \calC_{1}\to \calC_{2}$ and a natural isomorphism of functors $\calC_{1}\times \calC_{1} \to \calC_{2}$,
  \begin{equation*}
    J_{X,Y} : F(X) \otimes_{2} F(Y) \to F(X\otimes_{1}Y).
  \end{equation*}
  \textcolor{secondrevision}{
    The natural isomorphism $J$ must satisfy a coherence condition with respect to the six different functors $\calC_{1}^{3} \to \calC_{2}$ that may be constructed from $\otimes_{1}$,$\otimes_{2}$, and $F$: this condition is the commutativity of the diagram
    \begin{equation}
      \label{eq:monoidal-functor-coherence}
      \xymatrix{
        (F(X)\otimes_{2} F(Y))\otimes_{2} F(Z) \ar[rr]^{a_{F(X),F(Y),F(Z)}} \ar[d]_-{J_{X,Y}\otimes 1_{F(Z)}} & & F(X) \otimes_{2} (F(Y) \otimes_{2} F(Z)) \ar[d]^-{1_{F(X)}\otimes J_{Y,Z}}\\
        F(X\otimes_{1}Y) \otimes_{2} F(Z) \ar[d]_-{J_{X\otimes_{1} Y,Z}} & & F(X) \otimes_{2} F(Y\otimes_{1}Z) \ar[d]^-{J_{X,Y\otimes_{1}Z}}\\
        F((X \otimes_{1} Y)\otimes_{1} Z) \ar[rr]^{F(a_{X,Y,Z})}&& F(X \otimes_{1} (Y\otimes_{1}Z))
        }
    \end{equation}
    We also require that $F(1_{\calC_{1}})$ is isomorphic to $1_{\calC_{2}}$.}
\end{defn}

Returning to the subject of monoid objects, we remark on an extremely basic fact: \emph{Definition \ref{defn:cartmonoid} does not make sense in a general monoidal category}. The culprit are the first and last face maps $d_{0}, d_{n} : X([n]) \to X([n-1])$, which are meant to be projections. In a monoidal category, where the monoidal structure has nothing to do with Cartesian product, these maps usually cannot be defined. For instance, in the monoidal category $(\VectK,\otimes_{\K})$, there is no such thing as a projection operator $V \otimes_{\K} W \to V$.

The solution to this problem that we shall follow is to modify Definition \ref{defn:cartmonoid} by throwing away the first and last face maps. This has the effect of taking a simplicial object $X : \bDelta^{\op} \to \calC$, and restricting to the category $\bDeltaint^{\op} \subset \bDelta^{\op}$ of intervals (see Section \ref{sec:simplex-cat}). Because $\bDeltaint^{\op} \cong \bDelta_{a}$, we obtain an augmented cosimplicial object $X' : \bDelta_{a} \to \calC$. Now we have a new problem, since this process of throwing away face maps destroys the principal edge maps that are at the heart of Definition \ref{defn:cartmonoid}.

The solution to this new problem is to notice that $\bDelta_{a}$ is \emph{itself a monoidal category}. \textcolor{secondrevision}{In this definition, recall that $\mathbf{n} = [n-1]$ is a set of cardinality $n$, and that $\bDelta_{a}$ contains $\mathbf{0} = [-1] = \emptyset$.}
\begin{prop}[\cite{nlab}, article ``simplex category'']
  \label{prop:simplexmonoidal}
  Define a functor $\oplus : \bDelta_{a} \times \bDelta_{a} \to \bDelta_{a}$ as follows. \textcolor{secondrevision}{Set $\mathbf{n} \oplus \mathbf{m} = \mathbf{n} + \mathbf{m}$, the ordinary sum of ordinals; thus $[n] \oplus [m] = [n+m+1]$.} For $f : [n] \to [k]$ and $g: [m]\to [\ell]$, set $f\oplus g$ to be the function given by the formula
  \begin{equation*}
    (f\oplus g)(i) = \begin{cases} 
      f(i), &\text{if $0 \leq i \leq n$}, \\
      g(i - (n+ 1)) + k+1, &\text{if $n+1 \leq i \leq n+m+1$}.
    \end{cases}
  \end{equation*}
  \textcolor{secondrevision}{Let $\mathbf{0} = [-1]$ be the unit object}, and let $a$ and $\iota$ be identity morphisms. Then these structures define a monoidal structure on $\bDelta_{a}$ (which is strictly associative).
\end{prop}

\begin{defn}
  \label{defn:noncartmonoid}
  Let $(\calC,\otimes)$ be a monoidal category. A \emph{monoid object} in $\calC$ is a monoidal functor $X : (\bDelta_{a},\oplus) \to (\calC,\otimes)$.
\end{defn}
\textcolor{secondrevision}{This definition is motivated a result of Mac Lane \cite[VII.5, Proposition 1]{maclane}. A monoidal category is called \emph{strict} if the associativity isomorphisms $\alpha_{X,Y,Z}$ and the unitality isomorphism $\iota$ are all identities. Then \cite[VII.5, Proposition 1]{maclane} states that, when $(\calC,\otimes)$ is a strict monoidal category, there is a bijection between monoidal functors $(\bDelta_{a},\oplus) \to (\calC,\otimes)$ and monoid objects in a more standard sense (presented below). In other words, $(\bDelta_{a},\mathbf{1})$ is the universal ``monoidal category equipped with a monoid object.'' When $(\calC,\otimes)$ is not strict, there is no longer such a bijection, but since every monoidal category is monoidally equivalent to a strict one (the ``Mac Lane strictness theorem'' \cite[Theorem 2.8.5]{tensor-cat}), Definition \ref{defn:noncartmonoid} has essentially the same generality as the standard definition.}

Observe that, in Definition \ref{defn:noncartmonoid}, a monoid object is an instance of an augmented cosimplicial object. Let us remark a bit more on the structure of such a monoid object. \textcolor{secondrevision}{In this definition, $\Xbf{0} \cong 1_{\calC}$, and the ``underlying object'' of the monoid object is $\Xbf{1}$. Because $X$ is a monoidal functor and $\mathbf{1}^{\oplus n} = \mathbf{n}$ (cardinalities add), we have $\Xbf{n} = X(\mathbf{1} \oplus \cdots \oplus \mathbf{1}) \cong \Xbf{1}^{\otimes n}$. The unit in $\Xbf{1}$ is induced by the unique coface map $\delta^{0}: \mathbf{0} \to \mathbf{1}$, and the product on $\Xbf{1}$ is induced by the unique codegeneracy map $\sigma^{0}: \mathbf{2}\to \mathbf{1}$.}

\begin{rmk}
  At this point, we could attempt to formulate the idea that a monoidal category is the same thing as a monoid object in $\Cat$, the category of categories. This is actually already a higher-categorical notion so we will defer this until later.
\end{rmk}

\textcolor{secondrevision}{Definition \ref{defn:noncartmonoid} is not the ``standard'' definition of a monoid object in a monoidal category, but it is the one whose generalization to the $(\infty,1)$-categorical setting we shall use. For added motivation, we shall now recall the standard definition and show how a monoid object in the sense of Definition \ref{defn:noncartmonoid} gives rise to one in the standard sense.
}

\textcolor{secondrevision}{The standard definition of a monoid object in a monoidal category $(\calC,\otimes)$ is the direct abstraction of the definition of a ordinary monoid in the category of sets. Thus it is an object $A$ of $\calC$, together with morphisms $m : A \otimes A \to A$ and $e: 1_{\calC}\to A$, and these are required to satisfy the associative law and the left and and right unit axioms, expressed by the following commutative diagrams.
  \begin{equation}
    \xymatrix{
      (A \otimes A) \otimes A \ar[d]_{m\otimes 1_{A}} \ar[r]^{a_{A,A,A}}& A \otimes (A \otimes A)\ar[d]^{1_{A}\otimes m}\\
      A \otimes A \ar[d]_{m} & A \otimes A \ar[d]^{m}\\
      A \ar@{=}[r]& A
    }\qquad
    \xymatrix{
      A \ar[r]^-{\cong} \ar@{=}[ddr]& 1_{\calC} \otimes A \ar[d]^{e \otimes 1_{A}}\\
      & A \otimes A \ar[d]^{m}\\
      & A 
    }\qquad
    \xymatrix{
      A \ar[r]^-{\cong} \ar@{=}[ddr]& A \otimes 1_{\calC} \ar[d]^{1_{A} \otimes e}\\
      & A \otimes A \ar[d]^{m}\\
      & A 
    }
  \end{equation}
}

\textcolor{secondrevision}{We may think of Definition \ref{defn:noncartmonoid} as a (highly redundant) way of encoding the same data and conditions. Let $X : (\bDelta_{a},\oplus) \to (\calC,\otimes)$ be a monoidal functor. We shall show that the object $A = \Xbf{1}$ carries the structure of a monoid object in the standard sense.}

\textcolor{secondrevision}{First we construct the product on $\Xbf{1}$; it is a map $m : \Xbf{1} \otimes \Xbf{1} \to \Xbf{1}$. Part of the data of a monoidal functor is an isomorphism $J_{\mathbf{1},\mathbf{1}} : \Xbf{1}\otimes \Xbf{1} \to X(\mathbf{1} \oplus \mathbf{1}) = \Xbf{2}$. Because $X$ is a functor, the morphism $\sigma^{0} : \mathbf{2}\to \mathbf{1}$ induces $X(\sigma^{0}) : \Xbf{2} \to \Xbf{1}$. We define $m$ as the composition $m = X(\sigma^{0}) \circ J_{\mathbf{1},\mathbf{1}}$.
  \begin{equation}
    \xymatrix{
      m : \Xbf{1} \otimes \Xbf{1} \ar[r]^-{J_{\mathbf{1},\mathbf{1}}} & \Xbf{2} \ar[r]^{X(\sigma^{0})} & \Xbf{1}.
    }
  \end{equation}
}

\textcolor{secondrevision}{The structure of the proof that $m$ so defined satisfies the associative law involves composing several commutative diagrams that have three sources: functoriality of $X$, the naturality of the transformation $J$, and also the coherence condition for $J$ that is involved in the definition of a monoidal functor. It is possible to display the entire argument as a single commutative diagram:}

\textcolor{secondrevision}{
\begin{equation}
  \label{eq:big-diagram}
  {\footnotesize
  \xymatrix{
    (\Xbf{1} \otimes \Xbf{1}) \otimes \Xbf{1} \ar@{=}[r] \ar[d]_-{J_{\mathbf{1},\mathbf{1}}\otimes 1_{\Xbf{1}}} & (\Xbf{1} \otimes \Xbf{1}) \otimes \Xbf{1}\ar[d]_-{J_{\mathbf{1},\mathbf{1}}\otimes 1_{\Xbf{1}}} \ar[r]^-{a}& \Xbf{1} \otimes (\Xbf{1} \otimes \Xbf{1}) \ar[d]^-{1_{\Xbf{1}} \otimes J_{\mathbf{1},\mathbf{1}}} \ar@{=}[r] &\Xbf{1} \otimes (\Xbf{1} \otimes \Xbf{1}) \ar[d]^-{1_{\Xbf{1}} \otimes J_{\mathbf{1},\mathbf{1}}} \\
    \Xbf{2} \otimes \Xbf{1} \ar@{=}[r] \ar[d]_-{X(\sigma^{0})\otimes 1_{\Xbf{1}}}& \Xbf{2}\otimes \Xbf{1} \ar[d]_-{J_{\mathbf{2},\mathbf{1}}} & \Xbf{1}\otimes \Xbf{2} \ar@{=}[r]\ar[d]^-{J_{\mathbf{1},\mathbf{2}}} & \Xbf{1}\otimes \Xbf{2} \ar[d]^-{1_{\Xbf{1}} \otimes X(\sigma^{0})}\\
    \Xbf{1} \otimes \Xbf{1} \ar[d]_-{J_{\mathbf{1},\mathbf{1}}} & \Xbf{3} \ar[d]_-{X(\sigma^{0})} \ar@{=}[r]& \Xbf{3} \ar[d]^-{X(\sigma^{1})} & \Xbf{1} \otimes \Xbf{1} \ar[d]^-{J_{\mathbf{1},\mathbf{1}}}\\
    \Xbf{2} \ar@{=}[r] \ar[d]_-{X(\sigma^{0})} & \Xbf{2} \ar[d]_-{X(\sigma^{0})} & \Xbf{2} \ar@{=}[r] \ar[d]^-{X(\sigma^{0})} &  \Xbf{2} \ar[d]^-{X(\sigma^{0})}\\
    \Xbf{1} \ar@{=}[r] & \Xbf{1} \ar@{=}[r]  & \Xbf{1} \ar@{=}[r]& \Xbf{1}
  }}
\end{equation}
The proof consists of the following comments on diagram \eqref{eq:big-diagram}:
\begin{itemize}
\item The arrow in the top row is $a = a_{\Xbf{1},\Xbf{1},\Xbf{1}}$, the associativity constraint.
\item The left-most column in the diagram forms the composition $m \circ (m \otimes 1_{\Xbf{1}})$, and the right-most column forms the composition $m \circ (1_{\Xbf{1}}\otimes m)$. Thus the commutativity of \eqref{eq:big-diagram} implies the associative law for $m$.
\item Consider the sub-diagram formed by the two left-most columns of \eqref{eq:big-diagram}. It suffices to check that the middle rectangle whose vertices are $\Xbf{2}\otimes \Xbf{1}$ and $\Xbf{2}$ commutes. This follows from the fact that $J$ is a natural transformation of functors $\bDelta_{a} \times \bDelta_{a} \to \calC$. Namely, consider the morphism $(\sigma^{0},1_{\mathbf{1}}) : (\mathbf{2},\mathbf{1}) \to (\mathbf{1},\mathbf{1})$ in $\bDelta_{a}\times \bDelta_{a}$. Applying the functor $\Xbf{-} \otimes \Xbf{-}$ yields $X(\sigma^{0})\otimes 1_{\Xbf{1}} : \Xbf{2}\otimes \Xbf{1} \to \Xbf{1}\otimes \Xbf{1}$, while applying the functor $\Xbf{-\oplus -}$ yields $X(\sigma^{0}) : \Xbf{3} \to \Xbf{2}$, and these are the maps appearing in the rectangle under consideration. The equation $J_{\mathbf{1},\mathbf{1}} \circ (X(\sigma^{0})\otimes 1_{\Xbf{1}}) = X(\sigma^{0}) \circ J_{\mathbf{2},\mathbf{1}}$ is an instance of the naturality condition for $J$.
\item Consider the sub-diagram formed by the two right-most columns of \eqref{eq:big-diagram}. This diagram commutes by naturality of $J$, with essentially the same argument as above.
\item Consider the sub-diagram formed by the two middle columns of \eqref{eq:big-diagram}. The equality between the two copies of $\Xbf{3}$ divides this sub-diagram into upper and lower rectangles whose commutativity is proved separately.
  \begin{itemize}
  \item For the upper rectangle, observe that it consists entirely of components of the natural transformation $J$ and the associativity constraint $a$. This rectangle is nothing but an instance of the coherence condition \eqref{eq:monoidal-functor-coherence} appearing in the definition of a monoidal functor, so it commutes by hypothesis.
  \item The commutativity of the lower rectangle follows from the functoriality of $X$. In the category $\bDelta_{a}$, the two sequences of arrows
    \begin{equation}
      \xymatrix{\mathbf{3} \ar[r]^{\sigma^{0}} & \mathbf{2} \ar[r]^{\sigma^{0}} & \mathbf{1}} \text{ and }
      \xymatrix{\mathbf{3} \ar[r]^{\sigma^{1}} & \mathbf{2} \ar[r]^{\sigma^{0}} & \mathbf{1}}
    \end{equation}
    have the same composition, that is, the unique map $\sigma^{0} : \mathbf{3} \to \mathbf{1}$. Applying the functor $X$ to these sequences shows that both compositions in the lower rectangle equal $X(\sigma^{0}) : \Xbf{3} \to \Xbf{1}$.
  \end{itemize}
\end{itemize}}

\subsection{Different versions of $(\infty,1)$-categories}
\label{sec:infty-1-stuff}
There are several versions of the theory $(\infty,1)$-categories. For our present purpose we are interested in DG categories, $A_{\infty}$-categories, and $\infty$-categories (also known as quasicategories or weak Kan complexes).

DG categories and $A_{\infty}$-categories are always linear over a field $\K$, while $\infty$-categories need not be. However, there is a concept of $\K$-linear $\infty$-category that makes comparison between all three kinds possible.

\subsubsection{DG versus $A_\infty$}
\label{sec:dg-vs-ainf}

Recall that a functor between DG categories or $A_{\infty}$-categories is a \emph{quasi-equivalence} if it induces an equivalence on cohomology categories, or in other words, it induces a quasi-isomorphism on all morphism complexes and is essentially surjective. A basic observation is that every DG category is an $A_{\infty}$-category with vanishing higher operations, and a less basic one is that every $A_{\infty}$-category is quasi-equivalent to a DG category.

\textcolor{secondrevision}{There are several ways to construct a DG category equivalent to a given $A_{\infty}$-category $\cA$. One is to take the image of the Yoneda embedding $\cA \to \Mod \cA$. The Yoneda lemma implies that this functor is a quasi-equivalence onto its image, and $\Mod \cA$ is always a DG category. Another way to do this, used for instance by Canonaco-Ornaghi-Stellari \cite{COS}, involves taking the bar construction $B\cA$, which is a ``DG cocategory'', and then taking the cobar construction $\Omega(B\cA)$ of that, which is a DG category. In \cite{COS} it is shown that $\Omega(B\cA)$ is quasi-equivalent to $\cA$.}

\textcolor{secondrevision}{In fact, the results of Canonaco-Ornaghi-Stellari \cite{COS} imply that the homotopy categories of DG categories and $A_{\infty}$-categories, localized with respect to the quasi-equivalences, are equivalent at the $(\infty,1)$-categorical level. The present author has written a note that explains this in detail \cite{dg-versus-a-infinity}.}

\textcolor{secondrevision}{The precise result that we shall use is as follows. Let $\mathrm{dgcat}$ denote the ordinary category whose objects are DG categories and
morphisms are DG functors, and let $\mathrm{a_{\infty}cat}$ denote the ordinary category whose objects are $A_{\infty}$-categories and morphisms are $A_{\infty}$-functors. Let $W_{\mathrm{dg}}$ and $W_{A_{\infty}}$ denote the classes of quasi-equivalences. Define
\begin{equation}
  \DGcat = N(\mathrm{dgcat})[N(W_{\mathrm{dg}})^{-1}], \quad \Ainfcat = N(\mathrm{a_{\infty}cat})[N(W_{A_{\infty}})^{-1}]
\end{equation}
to be localizations of $\infty$-categories. Then $\DGcat$ and $\Ainfcat$ are equivalent as $\infty$-categories \cite[Corollary 5.2]{dg-versus-a-infinity}.}

\subsubsection{$A_{\infty}$-nerve}
\label{sec:ainf-nerve}
It is also possible to take a single $A_{\infty}$-category $\cA$ and construct an $\infty$-category out of it. This is the \emph{$A_{\infty}$-nerve} construction $N_{A_{\infty}}$ due to Faonte \cite{faonte-nerve,faonte-homotopy} and Tanaka \cite{tanaka-thesis} independently. The construction starts with the set $[n] = \{0 < 1 < \cdots < n\}$ considered as a category with a single morphism $i  \to j$ if $i \leq j$. There is then a $\K$-linear category $\Delta^{n}$ that has these morphisms as basis elements; it is a DG category with vanishing differential. The set of $n$-simplices in the nerve is then given by
\begin{equation*}
  N_{A_{\infty}}(\cA)([n]) = \Fun_{A_{\infty}}(\Delta^{n},\cA).
\end{equation*}
where the right-hand side is the set of $A_{\infty}$-functors.

Applying $N_{A_{\infty}}$ to categories of $A_{\infty}$-functors is another way to understand the categories of DG and $A_{\infty}$-categories. According to work of Faonte \cite{faonte-homotopy} and Oh-Tanaka \cite{oh-tanaka}, for any $A_{\infty}$-categories $\cA$ and $\cB$, there is a weak homotopy equivalence
\begin{equation*}
  \Hom_{\Ainfcat}(\cA,\cB) \to N_{A_{\infty}}(\Fun_{A_{\infty}}(\cA,\cB))^{\sim}
\end{equation*}
where the right-hand side is the largest $\infty$-groupoid contained in the $A_{\infty}$-nerve of the $A_{\infty}$-category of $A_{\infty}$-functors from $\cA$ to $\cB$. 

\subsubsection{DG and $A_{\infty}$ versus $\infty$-categories}

Another important piece of work relating DG categories to $\infty$-categories is that of Lee Cohn \cite{cohn}. Cohn considers not the category $\DGcat$ defined above but a further localization $\DGcat_{\mathrm{Morita}}$ where all Morita equivalences have been inverted. A \emph{Morita equivalence} is a functor that induces an equivalence of categories of modules. Cohn shows that $\DGcat_{\mathrm{Morita}}$ is equivalent to the $\infty$-category of idempotent complete $\K$-linear stable $\infty$-categories.

Rather than formally inverting Morita equivalences, another choice is to enlarge the $A_{\infty}$-categories themselves. Given an $A_{\infty}$-category $\cA$, we may form its idempotent complete triangulated envelope $\mathrm{Tw}^{\pi}(\cA)$. Ornaghi \cite{ornaghi} has shown that the $A_{\infty}$-nerve takes idempotent complete triangulated $A_{\infty}$-categories to idempotent complete stable $\infty$-categories. As of this writing, it is not entirely clear whether the $\infty$-categories of idempotent complete triangulated $A_{\infty}$-categories and idempotent complete stable $\infty$-categories are equivalent, but research in this area seems to be moving towards this conclusion.

\subsubsection{Tensor products}
\label{sec:ainf-tensor-product}
We must also discuss the tensor product of DG and $A_{\infty}$-categories. It is a easy to construct a tensor product operation for DG categories by taking pairs of objects and tensor products of morphism complexes. However, this notion is not homotopically stable. According to Tabuada \cite{tabuada}, there is a model structure on the ordinary category of DG categories whose weak equivalences are $W_{\mathrm{dg}}$, and whose associated homotopy $\infty$-category is $\DGcat \cong \Ainfcat$. The derived tensor product \cite{toen-dg} of two DG categories $C$ and $D$ is then defined to be
\begin{equation*}
  C \otimes^{\mathbb{L}} D := Q(C) \otimes D
\end{equation*}
where $Q(C)$ is a functorial cofibrant replacement for $C$ with the same set of objects, and the tensor product on the right-hand side is the naive tensor product of DG categories. The operation $\otimes^{\mathbb{L}}$ does respect quasi-equivalences, and so it descends to an operation on the homotopy category. In the rest of the paper, the operation $\otimes^{\mathbb{L}}$ is denoted simply by $\otimes$, because we have no use for the naive tensor product of DG categories.

\textcolor{thirdrevision}{Multiple approaches to defining the tensor product of $A_{\infty}$-algebras and $A_{\infty}$-categories are available in the literature, some of them very explicit. The underlying problem is to construct a diagonal $\Delta : A_{\infty} \to A_{\infty}\times A_{\infty}$ for the $A_{\infty}$-operad, that is, a diagonal for the associahedra; a diagonal on the multiplihedra may be used to define the tensor product of $A_{\infty}$-morphisms and $A_{\infty}$-functors. See Saneblidze-Umble \cite{saneblidze-umble}, Markl-Shnider \cite{markl-shnider}, and Loday \cite{loday}. Masuda-Thomas-Tonks-Vallette \cite{MTTV} and Laplante-Anfossi--Mazuir \cite{laplante-anfossi-mazuir} have provided convenient realizations of these diagonals in terms of polytopes.}

\textcolor{thirdrevision}{This method of defining tensor products of $A_{\infty}$-algebras and $A_{\infty}$-categories still presents challenges:
 if $\mathrm{a_{\infty}cat}$ denotes the ordinary category of $A_{\infty}$-categories and $A_{\infty}$-functors, then there is apparently no way to use such diagonals to make $\mathrm{a_{\infty}cat}$ into a monoidal category in the sense of Definition \ref{defn:ordinarymonoidalcat}. There is no diagonal on the multiplihedra that is strictly compatible with composition of $A_{\infty}$-functors \cite[Proposition 4.25]{laplante-anfossi-mazuir}; this means that the tensor product does not define a functor $\mathrm{a_{\infty}cat} \times \mathrm{a_{\infty}cat} \to \mathrm{a_{\infty}cat}$, although it is functorial with respect to strict\footnote{An $A_{\infty}$-morphism or $A_{\infty}$-functor is called \emph{strict} if all of its components vanish except for the linear one.} $A_{\infty}$-functors. If we restrict to strict $A_{\infty}$-functors, there is still a problem, since there is no diagonal on the associahedra that is strictly coassociative \cite[Theorem 6.1]{markl-shnider}; this means that the tensor product is only associative up to an $A_{\infty}$-isomorphism that cannot always be strict. However, the diagonal on the associahedra is coassociative up to homotopy, and the tensor product of $A_{\infty}$-functors is compatible with composition up to homotopy \cite[Proposition 4.26]{laplante-anfossi-mazuir}. Thus a homotopy-coherent framework, such as the theory of monoidal $\infty$-categories (Section \ref{sec:monoidal-infinity-cat}), is necessary to handle this case.}

\textcolor{thirdrevision}{For our present purpose, we shall not use an explicit formula for the tensor product of $A_{\infty}$-categories. Instead, we use the equivalence of $\infty$-categories $\DGcat \cong \Ainfcat$ to transfer the derived tensor product on $\DGcat$ over to $\Ainfcat$, and in this way make $\Ainfcat$ into a monoidal $\infty$-category whose product is also denoted by $\otimes$.}

\textcolor{thirdrevision}{In his work on the K\"{u}nneth theorem for Lagrangian Floer cohomology, Amorim \cite{amorim-tensor,amorim-kunneth} used a similar approach. Amorim converts given $A_{\infty}$-algebras $A$ and $B$ into quasi-isomorphic DG algebras, takes the tensor product in DG algebras, and then applies homotopy transfer to obtain an $A_{\infty}$-structure on $A\otimes B$. Amorim works with filtered and possibly curved $A_{\infty}$-algebras, and gives explicit formulas for all of these steps; his construction recovers the Markl-Shnider tensor product in cases where the latter is defined. Thus Amorim's work shows that the tensor product defined in terms of a diagonal can be seen as a particular instance of the general idea of transferring the tensor product from DG to $A_{\infty}$-algebras. It would be an excellent project to make the explicit tensor products of $A_{\infty}$-categories mentioned above fit into the $\infty$-categorical framework.}

There is also a tensor product on $\K$-linear stable $\infty$-categories \cite[\S 4.1]{BFN}, which is the recipient of the universal colimit preserving bifunctor $C \times D \to C\otimes D$. An alternative way to define the tensor product would be to replace $\Ainfcat$ and $\DGcat$ by their localizations at Morita equivalences (or apply $\mathrm{Tw}^{\pi}$ to all of the categories), apply Cohn's result, and then transport the tensor product of $\K$-linear stable $\infty$-categories over. The general formalism used in this paper is not sensitive to this choice.

\subsubsection{$\Sympcat$ as an $(\infty,1)$-category}
\label{sec:symp-as-oo-cat}

These remarks suggest a way to formulate the structure of $\Sympcat$ as an $(\infty,1)$-category. The objects are symplectic manifolds, and the space of $1$-morphisms $X \to Y$ is the $\infty$-groupoid
\begin{equation*}
  N_{A_{\infty}}(\cF(\overline{X} \times Y))^{\sim}
\end{equation*}
where $\cF$ stands for the idempotent complete triangulated Fukaya $A_{\infty}$-category. The composition in this category should be induced by the Ma'u-Wehrheim-Woodward composition functors, but this may not be strictly associative. An natural Ansatz is that these composition functors may be extended to a Segal category enriched over $\infty$-groupoids (see Section \ref{sec:inf2}). There is a Quillen equivalence between such Segal categories and $\infty$-categories \cite{joyal-tierney}, and passing through this equivalence we may regard $\Sympcat$ as an $\infty$-category. This version of $\Sympcat$ would suffice for the construction of monoid objects in the main body of the paper.

\subsection{Monoidal $(\infty,1)$-categories and monoid objects in them}
\label{sec:monoidal-infinity-cat}

It turns out that Definition \ref{defn:cartmonoid} generalizes directly to $\infty$-categories.

\begin{defn}[\cite{HA} Definition 4.1.2.5]
  \label{defn:cartmonoid-oo}
  Let $\calC$ be an $\infty$-category. A \emph{monoid object} of $\calC$ is a simplicial object $X: N(\bDelta^{\op}) \to \calC$, with the property that the principal edges
  \begin{equation*}
    \left\{e_{i,i+1} : X([n]) \to X([1])\right\}_{0\leq i \le n-1}
  \end{equation*}
  exhibits $X([n])$ as a product $X([1])^{n}$. \textcolor{secondrevision}{This condition is called the \emph{Segal condition}.}
\end{defn}
The complexity of this definition is contained in the phrase ``exhibit as a product''. In the context of this definition it means that, for any object $Y$ of $\calC$, the principal edges induce a map on mapping spaces
\begin{equation*}
  \Map_{\calC}(Y,X([n])) \to \prod_{i=0}^{n-1} \Map_{\calC}(Y,X([1])),
\end{equation*}
and this map is required to be a weak homotopy equivalence.

Amazingly, Definition \ref{defn:cartmonoid-oo} can even be applied when $\calC = \Cat_{\infty}$ is the $\infty$-category of $\infty$-categories, and this gives one of the possible definitions of a monoidal $\infty$-category.
\begin{defn}[\cite{DAG-II} Remark 1.2.5, \cite{GR-I} Chapter 1 \S{}3.1]
  \label{defn:monoidal-oo-cat}
  A \emph{monoidal $\infty$-category} is a monoid object in $\Cat_{\infty}$. 
\end{defn}
This means that a monoidal $\infty$-category is a simplicial object $X: N(\bDelta^{\op}) \to \Cat_{\infty}$ with the property that the principal edges induce an equivalence of $\infty$-categories $X([n]) \to X([1])^{n}$.

\begin{defn}
  \label{defn:cartmonoidal-ainfcat}
  A \emph{monoidal $A_{\infty}$-category} is a monoid object in $\Ainfcat$, where the latter is considered as an $\infty$-category. This means that the map $X([n]) \to X([1])^{n}$ is a quasi-equivalence of $A_{\infty}$-categories.
\end{defn}

\textcolor{secondrevision}{These definitions are direct generalizations of Definition \ref{defn:cartmonoid}, which was in turn motivated by Proposition \ref{prop:bmonoid}. Generalizing the discussion that surrounds that proposition and definition, we now explain how these new definitions gives rise to an ``associative bifunctor.''}

\textcolor{secondrevision}{Let $X$ be a monoidal $A_{\infty}$-category in the sense of Definition \ref{defn:cartmonoidal-ainfcat}. Then for each $n \geq 0$, $X([n])$ is an $A_{\infty}$-category, and we focus on $n = 1,2$. Write $\cA$ for $X([1])$; it is the underlying category of the monoidal structure. There are three $A_{\infty}$-functors
  \begin{equation*}
    X(d_{0}),X(d_{1}),X(d_{2}) : X([2]) \to \cA,
  \end{equation*}
  and the hypothesis is that
  \begin{equation*}
    X(d_{2})\times X(d_{0}) : X([2]) \to \cA \times \cA
  \end{equation*}
  is a quasi-equivalence of $A_{\infty}$-categories. We may therefore choose a homotopy inverse to $X(d_{2})\times X(d_{0})$, call it $F : \cA\times \cA \to X([2])$; this $F$ is not strictly unique but it is unique up to homotopy. Then define
  \begin{equation*}
    m = X(d_{1}) \circ F : \cA \times \cA \to \cA.
  \end{equation*}
  This $m$ is the bifunctor for the monoidal product on $\cA$. Note that it involves a choice of homotopy inverse $F$, so it is not uniquely determined by the data of $X$ itself (but it is determined up to homotopy).}

\textcolor{secondrevision}{The argument that the bifunctor $m$ is associative up to homotopy involves $X([3])$, and the following diagram. For brevity we write $X_{n} = X([n])$ and $d_{i}$ for what is really $X(d_{i})$:
  \begin{equation}
    \xymatrix{&& X_{3} \ar[dl]_-{d_{3}\times (d_{0}\circ d_{0})} \ar[dr]^-{d_{1}} &&\\
      & X_{2} \times X_{1} \ar[dl]_-{(d_{2}\times d_{0})\times 1_{X_{1}}} \ar[dr]^-{d_{1}\times 1_{X_{1}}} && X_{2} \ar[dl]_-{d_{2} \times d_{0}}\ar[dr]^-{d_{1}}& \\
      (X_{1}\times X_{1}) \times X_{1} && X_{1} \times X_{1} && X_{1}
      }
    \end{equation}
    Now we argue as follows
    \begin{itemize}
    \item The simplicial relation $d_{i} \circ d_{j} = d_{j-1}\circ d_{i}$ for $i < j$ immediately implies that the square in this diagram is commutative up to homotopy, since it is the image under $X$ of a commutative diagram in $\bDelta^{\op}$.
    \item The Segal condition implies that the ``down and to the left'' arrows in this diagram are all quasi-equivalences. In fact the composite map $X_{3} \to (X_{1}\times X_{1})\times X_{1}$ in this diagram is the Segal map.
    \item The ``zig-zag'' formed by the bottom two rows corresponds to the composite $m \circ (m \times 1_{X_{1}})$. More precisely, we have chosen a homotopy inverse $F$ for $d_{2}\times d_{1}$, and using that to reverse the arrows in the zig-zag we obtain $m \circ (m \times 1_{X_{1}})$.
    \item We may therefore conclude that the functor $m \circ (m \times 1_{X_{1}})$ is homotopic to a functor obtained by inverting the Segal map $X_{3} \to (X_{1}\times X_{1})\times X_{1}$ and then applying $d_{1}\circ d_{1}: X_{3} \to X_{1}$, which is the ``big zig-zag'' in the diagram.
    \item We then consider the other association $m \circ (1_{X_{1}}\times m)$. A directly analogous diagram shows that this functor is homotopic to a functor obtained by inverting the Segal map $X_{3} \to X_{1} \times (X_{1}\times X_{1})$ and then applying $d_{1} \circ d_{2} : X_{3}\to X_{1}$.
    \item Because $d_{1}\circ d_{2} = d_{1} \circ d_{1}$ in $\bDelta^{\op}$, the functors $X_{3} \to X_{1}$ induced by these maps are also homotopic. Thus the functors considered in the previous two bullet points are homotopic to each other, which was to be shown.
    \end{itemize}
  }
  \textcolor{secondrevision}{The preceding argument implicitly constructs a homotopy $H$ between $m\circ(m \times 1_{X_{1}})$ and $m \circ (1_{X_{1}} \times m)$. The next component $X_{4}$ serves to witness the coherence of $H$ when it is used to relate the five different functors $X_{1}^{4} \to X_{1}$ that may be built from $m$. The higher components $X_{n}$ serve to witness higher ``coherences between the coherences.''} 

Thus far, we have only defined monoid objects in an $\infty$-category $\calC$ in the case where the ambient monoidal structure on $\calC$ is the categorical Cartesian product. We seek to generalize Definition \ref{defn:noncartmonoid} to the case where $(\calC,\otimes)$ is a monoidal $\infty$-category. Recall that $(\bDelta_{a},\oplus)$ is a monoidal category whose monoidal operation is strictly associative. Passing to the nerve $N(\bDelta_{a})$, which is an $\infty$-category, we again obtain a strictly associative bifunctor $\oplus : N(\bDelta_{a}) \times N(\bDelta_{a}) \to N(\bDelta_{a})$. Thus $(N(\bDelta_{a}),\oplus)$ is a monoidal $\infty$-category in the sense of Definition 
\ref{defn:monoidal-oo-cat}.

Strictly speaking, the object that is referred to in Definition \ref{defn:monoidal-oo-cat} is not the pair $(N(\bDelta_{a}),\oplus)$ itself but the simplicial $\infty$-category $X_{(N(\bDelta_{a}),\oplus)} : N(\bDelta^{\op}) \to \Cat_{\infty}$ obtained by ``applying the classifying space construction with respect to $\oplus$.'' That is, $X_{(N(\bDelta_{a}),\oplus)}([n])$ is the $n$-fold Cartesian product $N(\bDelta_{a})^{\times n}$, the face maps are given by projection and $\oplus$, degeneracies are units, and so on.

\begin{defn}
  Let $(\calC_{1},\otimes_{1})$ and $(\calC_{2},\otimes_{2})$ be monoidal $\infty$-categories, whose structures are encoded by simplicial $\infty$-categories $X_{(\calC_{1},\otimes_{1})}$ and $X_{(\calC_{2},\otimes_{2})}$. A \emph{monoidal functor} $F:  (\calC_{1},\otimes_{1}) \to (\calC_{2},\otimes_{2})$ is a map of simplicial $\infty$-categories $X_{(\calC_{1},\otimes_{1})} \to X_{(\calC_{2},\otimes_{2})}$.
\end{defn}

\begin{defn}
  \label{defn:noncart-oo-monoid} Let $(\calC,\otimes)$ be a monoidal $\infty$-category. A \emph{$\otimes$-monoid object in $(\calC,\otimes)$} is a monoidal functor $Z: (N(\bDelta_{a}),\oplus) \to (\calC,\otimes)$.\footnote{$\otimes$-monoid objects are also commonly referred to as \emph{algebras} in $(\calC,\otimes)$.}
\end{defn}

\textcolor{secondrevision}{This definition is a direct generalization of Definition \ref{defn:noncartmonoid}.} It may be challenging to unpack since our object $X_{(N(\bDelta_{a}),\oplus)}$ is ``(co)simplicial'' in (at least) two distinct ways. Let $Z$ be a $\otimes$-monoid object.
Then $Z([1])$ is a map $N(\bDelta_{a}) \to \calC$, and so it is an augmented cosimplicial object in $\calC$. The point is that this augmented cosimplicial object is further constrained by the fact that $Z$ is a monoidal functor. Because \textcolor{secondrevision}{$\mathbf{1}^{\oplus k} = \mathbf{k}$} in $\bDelta_{a}$, and $Z$ is a monoidal functor, there is a single object \textcolor{secondrevision}{$A = Z([1])(\mathbf{1})$} in $\calC$ such that \textcolor{secondrevision}{$Z([1])(\mathbf{k})$ is equivalent to $A^{\otimes k}$} in $\calC$. It also implies that there are maps $m : A \otimes A \to A$ and $e : 1_{\calC} \to A$ such that the restriction of $Z([1])$ to the generating coface and codegeneracy maps in $\bDelta_{a}$ is equivalent to a diagram of the form
\begin{equation}
  \label{eq:cosimp-alg}
  \xymatrix { 1_{\calC} \ar[r] & A \ar@<1ex>[r]  \ar@<-1ex>[r]  &  A\otimes A \ar@<0ex>[l] \ar@<0ex>[r] \ar@<2ex>[r] \ar@<-2ex>[r] & A\otimes A \otimes A \ar@<1ex>[l] \ar@<-1ex>[l] } \cdots
\end{equation}
where the left-to-right maps come from $e$, and the right-to-left maps come from $m$.  The statement that $Z$ is a monoidal functor also encodes an extension of this diagram to a fully homotopy coherently associative and unital multiplication on the object $A$.

\textcolor{secondrevision}{Indeed, the homotopy coherent associativity of $m : A \otimes A \to A$ in this context may be explained in terms of the diagram \eqref{eq:big-diagram} (where $X = Z([1])$ and $A = \Xbf{1}$). Recall that this diagram is used to prove the associativity of $m$ in the classical context, and that the proof proceeded by showing that each of the smaller rectangles commutes due to a certain \emph{condition}: functoriality of $X$, naturality of $J$, and the coherence condition for $J$. In the $\infty$-categorical context, each of these \emph{conditions} is replaced by \emph{data} that witness the homotopy commutativity of each rectangle in \eqref{eq:big-diagram}.}

Conversely, giving a $\otimes$-monoid object of $(\calC,\otimes)$ is essentially the same as giving maps $m : A\otimes A\to A$ and $e :1_{\calC} \to A$ such that the diagram \eqref{eq:cosimp-alg} admits an extension to a \textcolor{secondrevision}{monoidal functor $(N(\bDelta_{a}),\oplus) \to (\calC,\otimes)$}, together with a choice of such extension. We summarize this discussion in the following proposition.
\begin{prop}
  \label{prop:noncart-oo-monoid}
  To give a $\otimes$-monoid object in a monoidal $\infty$-category $(\calC,\otimes)$, it is equivalent to give a object $A$ and maps $m : A \otimes A \to A$ and $e : 1_{\calC} \to A$, such that the diagram \eqref{eq:cosimp-alg} constructed from these data admits an extension to a \textcolor{secondrevision}{monoidal functor $(N(\bDelta_{a}),\oplus) \to (\calC,\otimes)$}, together with a choice of such extension.
\end{prop}

\begin{rmk}
  The definition of monoid object in $(\calC,\otimes)$ presented here is different from the one in \cite{DAG-II}. Let us comment on the relationship. In \cite{DAG-II}, a monoidal $\infty$-category is regarded as a coCartesian fibration of $\infty$-categories $p : \calC^{\otimes} \to N(\bDelta^{\op})$ satisfying a Segal condition, where the underlying $\infty$-category is the fiber of this fibration at $[1]$, that is, $\calC = \calC^{\otimes}_{[1]}$. It is proved that this definition is equivalent to Definition \ref{defn:monoidal-oo-cat}. A \emph{monoid object} in $(\calC,\otimes)$ is then defined to be a section $X : N(\bDelta^{\op}) \to \calC^{\otimes}$ of the coCartesian fibration $p$. Thus a monoid object is in particular a simplicial object, but in the larger category $\calC^{\otimes}$, not $\calC$ itself.

  We now sketch how to pass from such a monoid object to one in the sense of Definition \ref{defn:noncart-oo-monoid}. Take the map $X : N(\bDelta^{\op})\to \calC^{\otimes}$, and restrict it along the embedding $N(\bDelta_{a}) \to N(\bDelta^{\op})$ to obtain an augmented cosimplicial object $X'$ in $\calC^{\otimes}$. Now $X'$ is a section of the restricted fibration $p' : \calC^{\otimes}|_{N(\bDelta_{a})} \to N(\bDelta_{a})$. Therefore $X'([n])$ is an object of $\calC^{\otimes}_{[n]}$, and this object may be projected to $\calC^{\otimes}_{[1]} = \calC$ using the $\otimes$-product. The result is an augmented cosimplicial object in $\calC$ of the correct form to be a monoid object in the sense of our definition.
\end{rmk}

\subsection{The monoidal $\infty$-categories $(\Ainfcat,\otimes)$ and $(\Sympcat,\times)$ and monoid objects in them.}
\label{sec:monoids-in-symp}

Now we turn to the main application of the preceding theory in the main body of the paper. For triangulated $A_{\infty}$-categories $A$ and $B$, we have two notions of product, the ordinary product $A\times B$ and the $\K$-linear tensor product $A\otimes B$. For any third $A_{\infty}$-category $C$, functors $A\times B\to C$ are $A_{\infty}$-bifunctors, but functors $A\otimes B \to C$ are $A_{\infty}$-functors that are $\K$-bilinear, in the sense that they commute with the operation of tensoring an object in $A$ or $B$ with any given chain complex over $\K$.

For instance, suppose that $X$ and $Y$ are perfect stacks \cite{BFN}. If $\Perf$ denotes the $\infty$-category of perfect complexes, then we have \cite[Theorem 1.2]{BFN}
\begin{equation*}
  \Perf(X) \times \Perf(Y) = \Perf(X \coprod Y),\quad \Perf(X) \otimes \Perf(Y) = \Perf(X\times Y).
\end{equation*}
The same thing happens for Fukaya categories: $\cF(X)\times \cF(Y)$ is at least morally the Fukaya category of the disjoint union $X \coprod Y$, while $\cF(X)\otimes \cF(Y)$ is comparable to $\cF(X\times Y)$. 

We turn our attention now to $\Sympcat$, with the understanding that our discussion is conditional on certain Ans\"{a}tze about this object. Given two symplectic manifolds $X$ and $Y$, we may form their Cartesian product $X \times Y$. This is the same thing as the categorical Cartesian product \emph{in the category of smooth manifolds}, in the sense that, for any manifold $Z$ a smooth map $Z \to X\times Y$ is the same as a pair of smooth maps $Z \to X$ and $Z \to Y$. However, \emph{the Cartesian product $X \times Y$ is not the categorical Cartesian product in $\Sympcat$}. For if it were, then we would have projection morphisms $X \times Y \to X$ and $X \times Y \to Y$, but these maps are not given by Lagrangian correspondences; also, the statement ``a Lagrangian correspondence from $Z$ to $X\times Y$ is the same thing as a pair of Lagrangian correspondences from $Z$ to the factors'' is obviously false.

In summary, the tensor product on $\Ainfcat$ and the Cartesian product on $\Sympcat$ are definitely \emph{non-Cartesian} monoidal structures, and this is why we need to use the theory introduced above.

The statement that $(\Ainfcat,\otimes)$ is a monoidal $\infty$-category is a purely algebraic statement, and we shall not elaborate on it further. The statement that $(\Sympcat,\times)$ is a monoidal $\infty$-category is worth understanding. As suggested in Section \ref{sec:symp-as-oo-cat}, this $\infty$-category would be the one where the space of $1$-morphisms $X \to Y$ is the $\infty$-groupoid
\begin{equation*}
  N_{A_{\infty}}(\cF(\overline{X}\times Y))^{\sim}
\end{equation*}
obtained as the largest Kan complex contained in the $A_{\infty}$-nerve, but since that object is derived from $\cF(\overline{X}\times Y)$ it suffices to consider the latter. The functor $\times : \Sympcat \times \Sympcat \to \Sympcat$ takes a pair of symplectic manifolds $(X,Y)$ to their product $X \times Y$. Considering pairs of $1$-morphisms $(X_{1}\to X_{2}, Y_{1}\to Y_{2})$, we must also have a functor
\begin{equation*}
  \times : \cF(\overline{X}_{1}\times X_{2}) \times \cF(\overline{Y}_{1}\times Y_{2}) \to \cF(\overline{X}_{1}\times \overline{X}_{2} \times Y_{1} \times Y_{2}).
\end{equation*}
This functor should be defined by taking a pair of Lagrangian correspondences to their Cartesian product (and rearranging the factors). Since the underlying geometric operation is Cartesian product at all levels, it seems very likely that, in the final analysis of $\Sympcat$, this bifunctor can be made strictly associative.

We should also discuss the meaning of the statement that the Fukaya category functor
\begin{equation*}
  \cF: (\Sympcat,\times) \to (\Ainfcat,\otimes)
\end{equation*}
is a monoidal functor. In order for this to be true, we would need to know that $\cF(\point)$ is the category of chain complexes, and that $\cF(X\times Y) \cong \cF(X)\otimes \cF(Y)$. The statement about $\cF(\point)$ is true as a purely formal matter, but the second claim has geometric content, since it says that $\cF(X\times Y)$ is generated by product Lagrangians. In general, all that is obvious is that there is a functor $\cF(X)\otimes \cF(Y) \to \cF(X\times Y)$ that takes a pair of Lagrangians to its product. This means precisely that $\cF$ is a \emph{lax monoidal functor} from $(\Sympcat,\times)$ to $(\Ainfcat,\otimes)$.

In the theory of ordinary monoidal categories, it is known that lax monoidal functors take monoid objects to monoid objects \cite[article ``monoidal functor'']{nlab}. We could attempt to formulate this fact at the $\infty$-categorical level, which would imply that any monoid object in $(\Sympcat,\times)$ goes over to a monoid object in $(\Ainfcat,\otimes)$ after taking the Fukaya category. Instead we shall content ourselves to writing out the definition of monoid objects in the two categories in a way that makes the parallel between them obvious.

\begin{defn}[Monoid object in $\Sympcat$]
  \label{defn:monoid-in-symp}
  A \emph{$\times$-monoid object in $\Sympcat$} consists of a symplectic manifold $G$ and (generalized) Lagrangian correspondences $m : G\times G \to G$ and $e : \point \to G$, such that the diagram
  \begin{equation}
  \label{eq:cosimp-in-symp}
  \xymatrix { \point \ar[r] & G \ar@<1ex>[r]  \ar@<-1ex>[r]  &  G\times G \ar@<0ex>[l] \ar@<0ex>[r] \ar@<2ex>[r] \ar@<-2ex>[r] &G \times G \times G \ar@<1ex>[l] \ar@<-1ex>[l] } \cdots
\end{equation}
where the left-to-right correspondences come from $e$, and the right-to-left correspondences come from $m$, extends to \textcolor{secondrevision}{a monoidal functor $(N(\bDelta_{a}),\oplus) \to (\Sympcat,\times)$}, together with a choice of such extension.
\end{defn}

\begin{defn}[Monoid object in $\Ainfcat$]
  \label{defn:monoid-in-ainfcat}
  A \emph{$\otimes$-monoid object in $\Ainfcat$} consists of an $A_{\infty}$-category $\cF$ and $A_{\infty}$-functors $m : \cF \otimes\cF \to \cF$ and $e : \Ch_{\K} \to \cF$, such that the diagram
  \begin{equation}
  \label{eq:cosimp-in-ainfcat}
  \xymatrix { \Ch_{\K} \ar[r] & \cF \ar@<1ex>[r]  \ar@<-1ex>[r]  &  \cF \otimes \cF \ar@<0ex>[l] \ar@<0ex>[r] \ar@<2ex>[r] \ar@<-2ex>[r] &\cF \otimes \cF \otimes \cF\ar@<1ex>[l] \ar@<-1ex>[l] } \cdots
\end{equation}
where the left-to-right functors come from $e$, and the right-to-left functors come from $m$, extends to \textcolor{secondrevision}{a monoidal functor $(N(\bDelta_{a}),\oplus) \to (\Ainfcat,\otimes)$}, together with a choice of such extension.
\end{defn}

\subsection{Modules}
\label{sec:modules}
We have so far considered monoid objects, but we also need to discuss module objects over monoid objects. The main difference is that we replace the indexing category $\bDelta$ with a new version $\bDelta^{+}$ where some simplices have marked vertices. There is not really any new complexity to the notion of homotopy associativity, so we will be brief.

We use the category $\bDelta^{+}$ defined in \cite[Chapter 1 \S{}3.4]{GR-I}. The category $\bDelta^{+}$ has two kinds of objects: for each $n \geq 0$, we have the linearly ordered set $[n] = \{0 < 1 < \cdots < n\}$ and also the linearly ordered set $[n]^{+} =\{0 < 1 < \cdots < n < +\}$ that has a new element $+$ added on. The element $+$ marks the position of the module in the constructions. The morphisms in $\bDelta^{+}$ are defined as follows
\begin{enumerate}
\item morphisms $[n] \to [m]$ are monotonic maps,
\item morphisms $[n] \to [m]^{+}$ are monotonic maps whose image does not contain $+$,
\item there are no morphisms $[n]^{+}\to [m]$,
\item morphisms $[n]^{+}\to [m]^{+}$ are monotonic maps that map $+$ to $+$, and such that the preimage of $+$ is $+$.
\end{enumerate}

The idea behind this definition is that given a monoid $A$ and a module $M$ (in the category of sets), we define $X([n]) = A^{n}$ and $X([n]^{+}) = A^{n}\times M$. There are maps $A^{m} \to A^{n}$ that encode the monoid structure of $A$, there are maps $A^{m}\times M\to A^{n}$ that project out the $M$ factor and use the monoid structure of $A$, there are no meaningful maps $A^{m} \to A^{n}\times M$, and there are maps $A^{m}\times M \to A^{n}\times M$ that encode the module structure of $M$.

Observe that the category $\bDelta^{+}$ contains $\bDelta$ as a full subcategory. We first consider the Cartesian case. 
\begin{defn}[\cite{GR-I}, Chapter 1 \S{}3.4]
  \label{defn:cartmonoidmodulepair}
  Let $\calC$ be an $\infty$-category. A \emph{monoid-module pair} in $\calC$ consists of a map $X: N((\bDelta^{+})^{\op}) \to \calC$, with the property that the restriction of this map to $N(\bDelta^{\op})$ is a monoid object in the sense of Definition \ref{defn:cartmonoid-oo}, and such that the map
  \begin{equation*}
    X([n]^{+}) \to X([n]) \times X([0]^{+})
  \end{equation*}
  induced by the maps $[n] \to [n]^{+}$ sending $i \mapsto i$ and $[0]^{+} \to [n]^{+}$ sending $0 \mapsto n$ and $+ \mapsto +$, exhibits the left-hand side as a product.
\end{defn}
The object $X([0]^{+})$ is the underlying object of the module in the monoid-module pair. Applying this definition with $\calC = \Cat_{\infty}$ or $\calC = \Ainfcat$ \textcolor{secondrevision}{(with the Cartesian monoidal structure on $\calC$)} gives one definition of module category for a monoidal category.

To generalize to the case $(\Ainfcat,\otimes)$ or $(\Sympcat,\times)$, where the ambient monoidal structure is not Cartesian, we follow the same strategy of throwing away morphisms in $\bDelta^{+}$ that correspond to projections, meaning that the maximum element must always go to the maximum element, and the minimum element must always go to the minimum element. This has the effect that there are no morphisms between $[n]$ and $[m]^{+}$ in either direction. \textcolor{secondrevision}{Thus the resulting category decomposes into a disjoint union $\bDeltaint \coprod \bDeltaint^{+}$, where $\bDeltaint^{+}$ is the full subcategory of objects of the form $[m]^{+}$. The equivalence $\bDelta_{a} \cong \bDeltaint^{\op}$ and the monoidal structure on $\bDelta_{a}$ may be extended as follows.}

\textcolor{secondrevision}{Recall the monoidal category $(\bDelta_{a},\oplus)$ whose objects are finite ordinals $\mathbf{n}$. For each $n \geq 0$, let $\mathbf{n}^{+}$ be $\mathbf{n}$ with a new maximal element $+$ adjoined. Define a new category $\bDelta_{a}^{+}$ whose objects are $\mathbf{n}^{+}$, and such that the morphisms $\mathbf{n}^{+} \to \mathbf{m}^{+}$ are monotonic maps that map $+$ to $+$. Then $\bDelta_{a}^{+}$ is a module category (in the strict sense, and hence in any weaker or homotopy coherent sense) for the strict monoidal category $(\bDelta_{a},\oplus)$: the action on objects is $\mathbf{n} \oplus \mathbf{m}^{+} = (\mathbf{n}+\mathbf{m})^{+}$, that is, concatenation of totally ordered sets, and the action on morphisms is a concatenation operation similar to the one described in Proposition \ref{prop:simplexmonoidal}.}

\textcolor{secondrevision}{The equivalence of categories $\bDelta_{a} \cong \bDeltaint^{\op}$ from Proposition \ref{prop:simplex-dual} generalizes to an equivalence $\bDelta_{a}^{+} \cong (\bDeltaint^{+})^{\op}$. These equivalences send $\mathbf{n}$ to $[n]$ and $\mathbf{n}^{+}$ to $[n]^{+}$.}

\textcolor{secondrevision}{Recall that in $(\bDelta_{a},\oplus)$, the object $\mathbf{1}$ is a monoid object. In $\bDelta^{+}$, the object $\mathbf{0}^{+} = \{+\}$ is a module object over $\mathbf{1}$, where the action is given by the unique morphism $\mathbf{1}\oplus\mathbf{0}^{+} = \mathbf{1}^{+} \to \mathbf{0}^{+}$. (In this sentence we have invoked the concept of a ``module $B$ over a monoid $A$ where $A$ is an object of a monoidal category $\calC$ and $B$ is an object of a module category $\mathcal{M}$ for $\calC$''. This concept exists in the classical case but we have not developed it. The present paragraph is meant to serve as motivation for the formal definition that follows.) Just as $(\bDelta_{a},\mathbf{1})$ is the universal ``monoidal category with a monoid object,'' the tuple $(\bDelta_{a},\bDelta_{a}^{+},\mathbf{1},\mathbf{0}^{+})$ is the universal ``monoidal category with a module category and a monoid object and a module object.''}

\textcolor{secondrevision}{For the next definition, observe that any monoidal $\infty$-category $(\calC,\otimes)$ is a module over itself.}

\textcolor{secondrevision}{\begin{defn}
    \label{defn:noncartmonoidmodulepair}
    Let $(\calC,\otimes)$ be a monoidal $\infty$-category. Let $X_{(\calC,\calC)} : N((\bDelta^{+})^{\op}) \to \Cat_{\infty}$ be monoid-module pair corresponding to $\calC$ regarded as a module over itself, and let $X_{(\bDelta_{a},\bDelta_{a}^{+})}: N((\bDelta^{+})^{\op}) \to \Cat_{\infty}$ be the monoid-module pair corresponding to the classical monoidal category $\bDelta_{a}$ and its module category $\bDelta_{a}^{+}$. Then we define a \emph{$\otimes$-monoid-module pair in $\calC$} to be a morphism $F : X_{(\bDelta_{a},\bDelta_{a}^{+})} \to X_{(\calC,\calC)}$.
  \end{defn}}
    
\textcolor{secondrevision}{More generally, if $Y : N((\bDelta^{+})^{\op}) \to \Cat_{\infty}$ is any monoid-module pair in $\Cat_{\infty}$, then an \emph{$Y$-monoid-module pair} is a morphism $ F: X_{(\bDelta_{a},\bDelta_{a}^{+})} \to Y$.}

\textcolor{secondrevision}{Let us unpack this definition slightly: given a morphism $F : X_{(\bDelta_{a},\bDelta_{a}^{+})} \to X_{(\calC,\calC)}$, the evaluation on $[1]$ is a map $F([1]) : N(\bDelta_{a}) \to \calC$, and evaluating this on $\mathbf{1}$ gives an object $A = F([1])(\mathbf{1})$ of $\calC$ that is the underlying object of the monoid in the $\otimes$-monoid-module pair. Evaluating $F$ on $[0]^{+}$ gives a map $F([0]^{+}) : N(\bDelta^{+}) \to \calC$, and evaluating this on $\mathbf{0}^{+}$ gives an object $M = F([0]^{+})(\mathbf{0}^{+})$ of $\calC$ that is the underlying object of the module.}


\subsection{Category objects and $(\infty,2)$-categories}
\label{sec:inf2}

Considering the characterization of the nerve of a category expressed by Proposition \ref{prop:cnerve}, we may define the notion of a ``category object'' or ``Segal object'' in an $\infty$-category, and this provides an approach to the theory of $(\infty,2)$-categories. For instance this approach is used in \cite{goodwillie-I, GR-I}.

\begin{defn}
  \label{defn:catobj}
  Let $S$ be a set, and let $\calC$ be an $\infty$-category. A \emph{category object in $\calC$ with object set $S$} is a map $X : N(\bDelta_{S}^{\op}) \to \calC$ with such that
  \begin{enumerate}
  \item for any $s \in S$, $X([s])$ is a final object of $\calC$,
  \item for any sequence $[s_{0},\dots,s_{n}]$, the \emph{Segal map} (induced by the principal edges)
  \begin{equation*}
    X([s_{0},s_{2},\dots,s_{n}]) \to X([s_{0},s_{1}])\times X([s_{1},s_{2}]) \times \cdots \times X([s_{n-1},s_{n}])
  \end{equation*}
  exhibits the left-hand side as a product in $\calC$.
  \end{enumerate}
\end{defn}
Applying this definition with $\calC = \Cat_{\infty}$ is one of the known ways to define the concept of $(\infty,2)$-category. 
\begin{defn}
  A \emph{Segal category enriched over $\infty$-categories} is a category object in $\Cat_{\infty}$ in the sense of Definition \ref{defn:catobj}. A \emph{Segal category enriched over $A_{\infty}$-categories} is a category object in $\Ainfcat$ in the same sense, meaning that the Segal map is always an quasi-equivalence of $A_{\infty}$-categories.  In this definition we regard $\Cat_{\infty}$ and $\Ainfcat$ as $\infty$-categories.
\end{defn}

Let us unpack this definition a little bit in the $A_{\infty}$ case. Let $\cC$ be a Segal category enriched over $A_{\infty}$-categories with object set $S$. Then for any objects $s_{0},s_{1} \in S$, there is an $A_{\infty}$-category $\Hom_{\cC}(s_{0},s_{1})$. Given three objects $s_{0},s_{1},s_{2}$, we have a composition $A_{\infty}$-functor
\begin{equation*}
  \Hom_{\cC}(s_{0},s_{1})\times \Hom_{\cC}(s_{1},s_{2}) \to \Hom_{\cC}(s_{0},s_{2})
\end{equation*}
obtained by taking a homotopy inverse to the Segal map at level $[s_{0},s_{1},s_{2}]$ followed by the middle face map $d_{1}$. These composition $A_{\infty}$-functors are then associative up to an $A_{\infty}$ natural isomorphism of $A_{\infty}$-functors, together with all higher coherences for this associativity.

\begin{rmk}
  A Segal category enriched over $A_{\infty}$-categories is a kind of $(\infty,2)$-category that is a mixture of two approaches, where we are using the Stasheff associahedra to govern the homotopy associativity of the $2$-morphisms and higher morphisms, but using the simplex category $\bDelta^{\op}$ to govern the homotopy associativity of the $1$-morphisms. The notion of $(A_{\infty},2)$-categories defined by Bottman and Carmeli \cite{bottman,bottman-carmeli} uses the $2$-associahedra to govern all morphisms simultaneously.
\end{rmk}

\begin{rmk}
  It is possible to iterate the definition of Segal category to obtain theories of $(\infty,n)$-categories for all $n \geq 0$. An important point here is that, in passing from $(\infty,n-1)$-categories to $(\infty,n)$-categories, one first needs to define the correct notion of weak equivalence of $(\infty,n-1)$-categories so that the meaning of the Segal condition at the next level is correct. The purpose of our remarks is merely to show a way that the concept of $(\infty,2)$-category can be defined, rather than to study the homotopy theory of $(\infty,2)$-categories themselves.
\end{rmk}

Faonte \cite{faonte-homotopy} has used the concept of a Segal category enriched over $\infty$-categories as a natural way to formulate $\Ainfcat$ as an $(\infty,2)$-category. In this version of $\Ainfcat$, the objects are $A_{\infty}$-categories, and for two $A_{\infty}$-categories $A$ and $B$, the $\infty$-category of $1$-morphisms is
\begin{equation*}
  \Hom_{\Ainfcat}(A,B) = N_{A_{\infty}}(\Fun_{A_{\infty}}(A,B)),
\end{equation*}
where $\Fun_{A_{\infty}}$ is the $A_{\infty}$-category of unital $A_{\infty}$-functors, and $N_{A_{\infty}}$ is the nerve construction of Faonte and Tanaka. Faonte constructed a strictly associative composition in this setting.

Finally we can formulate our Ansatz about $\Sympcat$ as an $(\infty,2)$-category, namely that it may be constructed as a Segal category enriched over $A_{\infty}$-categories, where the binary composition functors are the Ma'u-Wehrheim-Woodward functors. This concept is formulated as Definitions \ref{defn:good-sol1} and \ref{defn:good-sol2} in the main body of the paper.

\bibliographystyle{plain}
\bibliography{monoidal}

\end{document}